\documentclass[10pt]{article}
\usepackage{caption}
\usepackage{epsf,epsfig,amsfonts,amsgen,amsmath,amstext,amsbsy,amsopn,amsthm
}
\usepackage{ebezier,eepic}
\usepackage{color}
\usepackage{multirow}
\usepackage{mathrsfs}
\usepackage{tikz}
\usepackage{enumerate}
\usepackage{enumitem}
\usepackage{url}
\usepackage{pgf,tikz,pgfplots}
\pgfplotsset{compat=1.18}
\usepackage{mathrsfs}
\usepackage{amssymb}

\usepackage{mathtools}
 
 \usepackage{algorithm}
\usepackage{algpseudocode}
\usepackage{mdframed}

\usepackage{centernot}
\usepackage[mathlines,left]{lineno}
\usetikzlibrary{arrows}

\newtheorem{dfn}{Definition}[section]
\newtheorem{thm}[dfn]{Theorem}
\newtheorem{lem}[dfn]{Lemma}

\newtheorem{conjecture}[dfn]{Conjecture}

\newtheorem{Claim}[dfn]{Claim}

\newcommand{\aA}{{\mathcal{A}}} 
\newcommand{\dD}{{\mathcal{D}}} 
\newcommand{\eE}{{\mathcal{E}}} 
\newcommand{\fF}{{\mathcal{F}}} 
\newcommand{\gG}{{\mathcal{G}}}
\newcommand{\hH}{{\mathcal{H}}}
\newcommand{\kK}{{\mathcal{K}}}
\newcommand{\pP}{{\mathcal{P}}}

\newcommand{\rR}{{\mathcal{R}}}

\newcommand{\fR}{{\mathfrak{R}}}  
\newcommand{\fr}{{\mathfrak{r}}}  

\def\CN{\mathrm{CN}}

\def\cd{\mathrm{cd}}

\let\svthefootnote\thefootnote
\newcommand\blankfootnote[1]{
	\let\thefootnote\relax\footnotetext{#1}
	\let\thefootnote\svthefootnote
}

\newcommand{\1}{{\uppercase\expandafter{\romannumeral1}}}
\newcommand{\2}{{\uppercase\expandafter{\romannumeral2}}}
\newcommand{\3}{{\uppercase\expandafter{\romannumeral3}}}
\newcommand{\4}{{\uppercase\expandafter{\romannumeral4}}}

\usetikzlibrary{arrows.meta,calc,fit,positioning,shapes.geometric}
\definecolor{edgeblue}{RGB}{0,82,164}
\definecolor{deletebrick}{RGB}{210,45,45}
\definecolor{restoreolive}{RGB}{0,145,82}
\definecolor{compatviolet}{RGB}{132,35,190}
\definecolor{incompatorange}{RGB}{230,105,0}
\definecolor{neutralstroke}{RGB}{45,45,45}

\tikzset{
  xvertex/.style={circle,draw=neutralstroke,fill=black!1,
    minimum size=6.2mm,inner sep=0pt,font=\small},
  yvertex/.style={circle,draw=neutralstroke,fill=black!6,
    minimum size=5.8mm,inner sep=0pt,font=\small},
  pairvertex/.style={circle,draw=edgeblue,fill=edgeblue!8,
    minimum size=7.2mm,inner sep=0pt,font=\small},
  live/.style={draw=edgeblue,line width=0.92pt},
  deleted/.style={draw=deletebrick,densely dashed,line width=0.90pt},
  available/.style={draw=restoreolive,densely dotted,line width=1.10pt},
  relation/.style={draw=incompatorange,line width=1.10pt},
  compat/.style={draw=compatviolet,double=white,double distance=0.55pt,
    line width=0.78pt},
  incompat/.style={draw=incompatorange,line width=1.10pt},
  protected/.style={draw=restoreolive,densely dotted,line width=1.0pt},
  paneltitle/.style={font=\small\bfseries,align=center},
  note/.style={font=\scriptsize,align=center},
  arrow/.style={-{Latex[length=2.2mm,width=1.5mm]},draw=neutralstroke,
    line width=0.82pt},
  bracebox/.style={draw=black!35,rounded corners=2pt,inner sep=3pt}
}

\begin{document}

	\title{Extremal hypergraphs without generalized $4$-cycles}
	
	\author{
		Hao Huang\thanks{Department of Mathematics, National University of Singapore, 119076, Singapore. Research supported by a start-up grant at NUS and an MOE Academic Research Fund (AcRF) Tier 1 grant A-8003627. Email: {\tt huanghao@nus.edu.sg}.
        }
		~~~~
		Jie Ma\thanks{School of Mathematical Sciences, University of Science and Technology of China, Hefei, Anhui 230026, and Yau Mathematical Sciences Center, Tsinghua University, Beijing 100084, China. Research supported by National Key Research and Development Program of China 2023YFA1010201 and National Natural Science Foundation of China grant 12125106. Email: {\tt jiema@ustc.edu.cn}.
        }
		~~~~
		Tianchi Yang\thanks{School of Mathematics, Georgia Institute of Technology, Atlanta, GA 30332, US. Email: {\tt tyang439@gatech.edu}.
		}
	}
	
	\date{\today}
	\maketitle

\begin{abstract}
In 1977, Erd\H{o}s posed the problem of determining the maximum number $f_r(n)$ of edges in an $n$-vertex $r$-uniform hypergraph in which all disjoint pairs of edges have distinct unions.
F\"uredi later conjectured that, for every fixed $r\ge 4$ and all sufficiently large $n$, $f_r(n)=\binom{n-1}{r-1}+\lfloor \frac{n-1}{r}\rfloor$.  
In this paper, we prove this conjecture and determine all extremal configurations.  
Our proof combines a stability theorem for such dense hypergraphs with a delicate deletion argument applied to an associated bipartite $3$-graph.
The stability theorem also resolves a conjecture of Mubayi.
\end{abstract}

\section{Introduction}
In this paper, we study a hypergraph Tur\'an problem for generalized $4$-cycles, posed by Erd\H{o}s~\cite{Er77} nearly half a century ago.
Given an integer $r\geq 2$, a \emph{generalized $4$-cycle} is an $r$-uniform hypergraph (or $r$-graph, for short) consisting of four distinct $r$-sets $A, B, C, D$ such that $A \cup B = C \cup D$ and $A \cap B = C \cap D = \emptyset$.
Throughout, we denote by $\mathcal{C}_4^r$ the family of all $r$-uniform generalized $4$-cycles.
Erd\H{o}s~\cite{Er77} asked for the maximum number $f_r(n)$ of edges in an $n$-vertex $r$-graph containing no member of $\mathcal{C}_4^r$.
The case $r = 2$ corresponds to the Tur\'an problem for the $4$-cycle, one of the cornerstone problems of extremal graph theory.
It is well known that $f_2(n) = \left(\frac{1}{2} + o(1)\right) n^{3/2}$; see \cite{Br66,ERS66,KST54} for the classical results, and \cite{MaYa23} for more refined bounds.

The cases $r\geq 3$ have also received considerable attention.
It was noted in~\cite{Er77} that Bollob\'as and Erd\H{o}s (unpublished) proved $f_3(n)=O(n^2)$.
In 1984, F\"uredi~\cite{Fu84} established in a breakthrough work that $f_r(n)<\frac72\binom{n}{r-1}$ for all $n\geq 2r$.
His proof introduced an elegant graph removal lemma, which has since become an influential tool in many subsequent works.
F\"uredi \cite{Fu84} also gave the lower bound $f_r(n)\ge \binom{n-1}{r-1}+\left\lfloor\frac{n-1}{r}\right\rfloor$ for all $r\ge 3$, via the $r$-graph obtained from an $r$-uniform full star by adding a maximum matching among the vertices other than the center (see the precise construction (C1) we will describe soon). 
The case $r=3$ is more intricate: F\"uredi \cite{Fu84} gave a slightly better lower bound, using the $3$-graph built from a Steiner system $S(n,5,2)$ by replacing each block with the complete $3$-graph $K_5^3$ on its vertices.
This yields $f_3(n)\ge \binom{n}{2}$ whenever $n\equiv 1,5\pmod{20}$.
Motivated by these constructions, F\"uredi~\cite{Fu84} made the following famous conjecture, which was later collected by Erd\H{o}s himself in his problem list~\cite{Er97} and is now known as the Erd\H{o}s--F\"uredi conjecture~\cite{CG}.

\begin{conjecture}[Erd\H{o}s--F\"uredi]\label{conj:EF}
For all sufficiently large $n$, $f_3(n)\le \binom{n}{2}$.  Moreover, for every fixed $r\ge 4$ and all sufficiently large $n$,
$f_r(n)=\binom{n-1}{r-1}+\left\lfloor\frac{n-1}{r}\right\rfloor.$
\end{conjecture}

Twenty years later, Mubayi and Verstra\"ete \cite{MuVe04} improved F\"uredi's upper bound to $f_r(n)\le 3\binom{n}{r-1}+O(n^{r-2})$ for all $r\geq 3$, through a more refined analysis of the approach in \cite{Fu84}.
Pikhurko and Verstra\"ete \cite{PiVe09} further strengthened this to $f_r(n)/\binom{n}{r-1}\le \min\{1+2/\sqrt r,\,7/4\}+o(1)$; in particular, the constant tends to $1$ as $r\to\infty$.
They extended the graph removal lemma of \cite{Fu84} from the bipartite setting to general graphs, combining it with a randomized technique.

\medskip

Our main result confirms the Erd\H{o}s--F\"uredi conjecture for every $r\ge 4$, and moreover determines all extremal configurations.

\begin{thm}\label{thm:extremal}
For every fixed $r\ge 4$ and all sufficiently large $n$,
\[
        f_r(n)=
        \binom{n-1}{r-1}+\left\lfloor\frac{n-1}{r}\right\rfloor.
\]
Moreover, every extremal $r$-graph is isomorphic to one of the following two constructions:
\begin{itemize}
    \item[(C1).] The $r$-graph with vertex set $[n]$ and edge set
    \[
    \left\{\,A\in \binom{[n]}{r}: n\in A\,\right\}
    \cup
    \left\{\{kr+1,kr+2,\dots,kr+r\}: 0\le k\le \left\lfloor\frac{n-1}{r}\right\rfloor-1\right\}.
    \]
    \item[(C2).] When $r\mid n$, the $r$-graph obtained from (C1) by replacing the edge
    $\{n-r+1,\dots,n\}$ with $\{n-r,\dots,n-1\}$.
\end{itemize}
\end{thm}

The proof of this result relies on two key lemmas, which we state in the next section (Lemmas~\ref{lem:one-root} and~\ref{lem:CN-stability}).
Together, these lemmas give a two-step scheme for approximating the star structure of the extremal configuration.
Conceptually, the proof proceeds in two rounds of uniformity reduction.
We first project the original $r$-graph to a bipartite $3$-graph whose edges are the triples $Axy$ with $|A|=r-2$ and $A\cup\{x,y\}$ an $r$-edge.
We then use Baranyai's theorem (decomposing all $(r-2)$-sets into perfect matchings) to partition it into balanced bipartite $3$-graphs and then analyze each of these $3$-graphs via its $2$-uniform link graphs, where a graph-deletion argument can be applied.
The assumption $r\geq 4$ is essential at several points, including the nontrivial decomposition of the $(r-2)$-sets.
For an overview of the proof, we refer the reader to Section~\ref{sec:ProofOverview}.

These lemmas also yield a stability result for dense $\mathcal C_4^r$-free $r$-graphs.
Motivated by his earlier work on the generalized triangle, Mubayi~\cite{Mu07} conjectured that an analogous phenomenon holds for generalized $4$-cycles in $r$-graphs with $r\geq 4$.
In particular, our stability result confirms his conjecture as follows.

\begin{thm}\label{thm:stability}
Fix $r\ge 4$.  For every $\delta>0$, there exist $\epsilon>0$ and $n_0$ such that, for all $n\ge n_0$, every $n$-vertex $\mathcal C_4^r$-free $r$-graph with at least $(1-\epsilon)\binom{n-1}{r-1}$ edges contains a vertex $v$ that belongs to at least
$(1-\delta)\binom{n-1}{r-1}$ edges.
\end{thm}

\medskip

The rest of the paper is organized as follows.
In Section~2, we introduce the notation, state the two main lemmas (namely Lemmas~\ref{lem:one-root} and \ref{lem:CN-stability}), and give a more detailed proof overview.
In Section~\ref{sec:main}, we derive Theorems~\ref{thm:extremal} and \ref{thm:stability} from these lemmas.
In Section~\ref{sec:intersecting-families}, we prove  Lemma~\ref{lem:one-root} by showing that after deleting few edges, the common neighborhood of every pair of vertices becomes an $(r-1)$-uniform star.
In Section~\ref{sec:stability-com}, we show that in a dense $\mathcal C_4^r$-free $r$-graph, the aforementioned ``local stars'' have essentially the same center, using
a crucial deletion lemma whose proof is given in Section~\ref{sec:remove-diag}.
Sections~\ref{sec:stability-com} and~\ref{sec:remove-diag} together prove Lemma~\ref{lem:CN-stability}.
Finally, Section~7 contains some remarks and further questions.

\section{Notation, key lemmas, and proof overview}\label{sec:ProofOverview}
\paragraph{Notation.} Throughout this paper, we regard $r$ as a fixed positive integer, and all asymptotic notation is taken as $n\to\infty$.
For a positive integer $n$, let $[n]=\{1,2,\ldots,n\}$, and write $\binom{[n]}{r}$ for the family of all $r$-element subsets of $[n]$.
An $r$-uniform hypergraph (or {\it $r$-graph}, for short) on vertex set $[n]$ is a family $\fF\subseteq \binom{[n]}{r}$; its number of edges is denoted by $|\fF|$.

Let $\fF$ be an $r$-graph with vertex set $[n]$.
For a set $T\subseteq [n]$ with $|T|<r$, the {\it link} of $T$ in $\fF$ is
\[
 N_{\fF}(T)=\{S\subseteq [n]\setminus T: |S|=r-|T| \text{ and } S\cup T\in \fF\},
\]
and we write $d_{\fF}(T)=|N_{\fF}(T)|$.
When $T=\{v\}$, we also write $N_{\fF}(v)$, or $\fF_v$, for the link of $v$.
A family of sets is called a \emph{star} if either the family is empty, or there is a vertex contained in every member of the family. In the latter case such a vertex is called a \emph{center}.  If all members have size $k$, we also call it a $k$-uniform star (or $k$-star, for short).

If $X_1,\ldots,X_t \subseteq V(\fF)$ are sets of the same size smaller than $r$, we write
\[
 \CN_{\fF}(X_1,\ldots,X_t)
 =N_{\fF}(X_1)\cap\cdots\cap N_{\fF}(X_t)
\]
for their {\it common neighborhood}, and set
\[
 \cd_{\fF}(X_1,\ldots,X_t)=|\CN_{\fF}(X_1,\ldots,X_t)|.
\]
In particular, for vertices $x,y$, the family $\CN_{\fF}(x,y)$ consists of all $(r-1)$-sets which, together with each of $x$ and $y$, form edges of $\fF$.
Let $\mathcal A$ be a family of subsets of a ground set $V$. 
A \emph{cover} of $\mathcal A$ is a subset of $V$ having nonempty intersection with every member of $\mathcal A$. The \emph{covering number} of $\mathcal A$ is the minimum size of such a cover.
We say the family $\mathcal A$ is \emph{intersecting} if any two members intersect.

\paragraph{Key Lemmas.}
We now state two technical lemmas that are key to the proofs of our main results.
The first one provides a ``local starification'' statement: after deleting a small number of edges, all common neighborhoods of pairs of vertices become stars.

\begin{lem} \label{lem:one-root}
Let $r\ge 4$ be an integer, let $\delta\in (0,1/6)$, and $n$ be a sufficiently large integer compared with $r$ and $\delta$. 
Then any $n$-vertex $\mathcal{C}_4^r$-free $r$-graph $\fF$ contains a spanning subgraph $\gG$ such that $|\gG|\geq |\fF| - 4n^{r-1-\delta}$ and for any two vertices $x$ and $y$, $\CN_{\gG}(x,y)$ is an $(r-1)$-star. 
\end{lem}

The second lemma asserts that, in a dense $r$-graph, if the common neighborhoods of all pairs of vertices are already stars, then almost all of these ``local stars'' have a common center.

\begin{lem} \label{lem:CN-stability}
Let $r\ge 4$ be an integer, let $\epsilon\in (0, 1/420r)$, and $n$ be a sufficiently large integer compared with $r$ and $\epsilon$. 
Suppose $\gG$ is an $n$-vertex $\mathcal C_4^r$-free $r$-graph such that $|\gG|\ge (1-\epsilon)\binom{n-1}{r-1}$ and for any two vertices $x$ and $y$, $\CN_{\gG}(x,y)$ is an $(r-1)$-star. Then there exists a vertex $u$ and a subset $W\subseteq V(\gG)$ such that $|W|\ge (1-420\epsilon r)n$ and for any two vertices $x,y\in W$,
$\CN_{\gG}(x,y)$ is an $(r-1)$-star centered at $u$. 
\end{lem}

We remark that a half-way proof of Lemma~\ref{lem:CN-stability} already suffices to give the upper bound $|\gG|\le (1+o(1))\binom{n}{r-1}$; see the implication for \eqref{equ:upbd-G}.
Together with Lemma~\ref{lem:one-root}, this implies that \( f_r(n)=(1+o(1))\binom{n}{r-1}.\)

\paragraph{Proof Overview.}
We now sketch the main steps of the proof in more detail.
The overall strategy is to approximate the star structure of the extremal configuration using the two key lemmas stated above. 
Let $\fF$ be an $n$-vertex $\mathcal{C}_4^r$-free $r$-graph with $|\fF|\geq (1-o(1))\binom{n-1}{r-1}$.
The argument proceeds in two steps:
\begin{itemize}
    \item Lemma~\ref{lem:one-root} allows us to delete a small number of edges from $\fF$ to obtain a subgraph $\gG$ in which the common neighborhood \(\CN_{\gG}(x,y)\) of any two vertices $x$ and $y$ forms an $(r-1)$-star.  
    \item Lemma~\ref{lem:CN-stability} then shows that most of these stars \(\CN_{\gG}(x,y)\) share a common center.
\end{itemize}
In Section~\ref{sec:main}, we show that Lemmas~\ref{lem:one-root} and~\ref{lem:CN-stability} together suffice to derive Theorems~\ref{thm:extremal} and~\ref{thm:stability}.
To see this, we apply these lemmas to $\mathcal{F}$ to obtain a common center $v$, and then decompose $\mathcal{F}$ into a large star centered at $v$ together with the remaining edges.
The key observation is that a Berge path of length two formed by two non-star edges would, together with the large star, very likely create a generalized $4$-cycle.
A double-counting argument for these Berge paths then shows that almost all edges of $\fF$ lie in the star, which proves the stability result (Theorem~\ref{thm:stability}). 
For the extremal result (Theorem~\ref{thm:extremal}), the same argument is refined to show that the number of non-star edges is at most \(1\), yielding precisely the constructions \((\mathrm{C1})\) and \((\mathrm{C2})\).

The proof of Lemma~\ref{lem:one-root} (given in Section~\ref{sec:intersecting-families}) rests on structural properties of intersecting families. Since $\fF$ is $\mathcal C_4^r$-free, every common neighborhood $\CN_{\fF}(x,y)$ is an intersecting $(r-1)$-graph, so our aim is to delete few edges and make each a star (i.e., of covering number one). To this end, we design a potential function that selects a dense subgraph $\fF_1\subseteq\fF$ favoring pairs whose common neighborhoods have small covering number. We show that in $\fF_1$, no common neighborhood has covering number at least three and that pairs of covering number two are sparse and highly structured. Removing the edges associated with these pairs leaves a subgraph $\gG$ in which every common neighborhood is an $(r-1)$-star.

The proof of Lemma~\ref{lem:CN-stability} is the technical core of the paper and occupies Sections~\ref{sec:stability-com} and~\ref{sec:remove-diag}. We encode each edge of the \(r\)-graph $\gG$ as a triple \(Axy\), where \(A\in \binom{[n]}{r-2}\) and \(x,y\in [n]\), and regard the resulting object as the so-called bipartite \(3\)-graph $\hH=
\hH(\aA,\mathcal{B})$ with the two parts \(\aA:=\binom{[n]}{r-2}\) and \(\mathcal{B}:=[n]\). 
We say two vertices $A, A'\in \aA$ are compatible if $A\cap A'\neq \emptyset$.
Using Baranyai's theorem, we randomly partition \(\aA=\binom{[n]}{r-2}\) into subsets $\aA_i$ of size $n$, in each of which every vertex has bounded compatibility degree.

A technically involved yet crucial ingredient is a deletion lemma on bipartite $3$-graphs (Lemma~\ref{lem:3-graph-deletion}). 
The lemma relies on an abstract framework associated with the
compatibility relation defined above.
Very roughly speaking, its proof builds on modifications of the graph-deletion lemmas of F\"uredi \cite{Fu84} and Pikhurko and Verstra\"ete \cite{PiVe09}, which are applied to each of the link graphs while accommodating several additional restrictions tailored to our purposes. 
We refer the reader to the beginning of Section~\ref{sec:remove-diag} for a proof sketch.

We then apply Lemma~\ref{lem:3-graph-deletion} to each induced bipartite $3$-graph $\hH_i:=\hH(\aA_i,\mathcal{B})$.\footnote{It is important for our quantitative estimates that we take $|\aA_i|=|\mathcal{B}|=n$.}
It removes relatively few triples and leaves a simpler structure: for almost every pair of vertices \(x,y\in \mathcal{B}\), its common neighborhood (i.e., an $(r-2)$-graph) is either compatible (i.e., an intersecting family) or belongs to a small exceptional configuration. 
The proof of Section~\ref{sec:stability-com} then combines this structure with delicate density arguments to locate a positive fraction of the $\hH_i$'s. Each such $\hH_i$ contains a compatible subfamily $\mathcal{S}\subseteq \aA_i$ of size $r-2$ such that for most pairs $\{x,y\}$ in $\mathcal{B}$, their common neighborhood in $\hH_i$ is a star whose center lies in every member of $\mathcal{S}$; see Lemma~\ref{lem:cluster-S-set-W}.
Finally, we show that there exist two such $\hH_i$'s with a particular property that forces a single shared center for the common neighborhoods in $\gG$ of all pairs of vertices in a subset of size $(1-o(1))n$.
This establishes Lemma~\ref{lem:CN-stability}.

\section{From the key lemmas to the main results}\label{sec:main}
Assuming Lemmas~\ref{lem:one-root} and~\ref{lem:CN-stability}, we now show how the stability (Theorem~\ref{thm:stability}) and the exact extremal result (Theorem~\ref{thm:extremal}) follow.
The key intermediate step is a deficit lemma: once a large set of local stars has a common center, the missing part of the full star must be small.

Throughout this section, let $\mathcal{F}$ denote a $\mathcal C_4^r$-free $r$-graph with vertex set $[n]$.
We partition $\mathcal{F}$ into two edge-disjoint subgraphs $\mathcal{G}$ and $\mathcal{H}$,
where $\mathcal{G}$ is the maximal $r$-star (say with center $n$) in $\mathcal{F}$, and $\mathcal{H}$ consists of the remaining edges of $\mathcal{F}$.
In the applications below, the center $n$ will be chosen to be the common center supplied by Lemma~\ref{lem:CN-stability}; this choice determines the partition $\mathcal{F}=\mathcal{G}\cup\mathcal{H}$.
Let $|\mathcal{G}|=\binom{n-1}{r-1}-m$. Note that $\mathcal{H}\subseteq \binom{[n-1]}{r}$.

We first introduce the notation needed for the key technical lemma (Lemma~\ref{lem:G-H-U-C4}),
which shows that under mild conditions $\mathcal{G}$ is close to a full $r$-star, equivalently that $m$ is small.
Together with Lemmas~\ref{lem:one-root} and~\ref{lem:CN-stability}, this lemma yields the stability result (Theorem~\ref{thm:stability}) and is then used to establish Theorem~\ref{thm:extremal}.

\subsection{Preliminaries: the $i$-graphs $\mathcal{Q}_i$ and related intersecting families}\label{subsec:Qi}
In this subsection, we introduce an auxiliary $i$-graph $\mathcal{Q}_i$ for each $1\leq i\leq r-1$ and record the properties needed later.
These auxiliary hypergraphs record which star edges are missing or sparse; they will be used repeatedly in the counting arguments below.
To be precise, let
\[\mathcal{Q}_{r-1}=\left\{T\in \binom{[n-1]}{r-1}:T\cup\{n\}\notin \mathcal{G} \right\}\]
and for $1\le i\le r-2$, let
\[\mathcal{Q}_i=\left\{T\in \binom{[n-1]}{i}:d_{\mathcal{G}}(T)\leq (1/2+n^{-0.1})\binom{n-1}{r-i-1} \right\}.\]
Note that $|\mathcal{Q}_{r-1}|= \binom{n-1}{r-1}-|\mathcal{G}|=m$,
and for $1\le i\le r-2$, every set $T\in \mathcal{Q}_i$ has $d_{\mathcal{Q}_{r-1}}(T)=\binom{n-1}{r-i-1}-d_\mathcal{G}(T)\geq (\frac{1}{2}-n^{-0.1})\binom{n-1}{r-i-1}$.
Hence, we obtain
\begin{equation}\label{equ:Q-r-1}
	|\mathcal{Q}_{i}|\leq \frac{\binom{r-1}{i}\cdot |\mathcal{Q}_{r-1}|} {(\frac{1}{2}-n^{-0.1})\binom{n-1}{r-i-1}} \leq r^{2r}mn^{i+1-r}\text{ for all } 1\leq i\leq r-1.
\end{equation}
	
The following lemma records a property used repeatedly below.

\begin{lem}\label{lem:N-Q-int}
Let $\mathcal{F}$, $\mathcal{G}$ and $\mathcal{H}$ be defined as in the beginning of this section.
Let $1\le t\le r-1$, and let $S\subseteq [n-1]$ be any $(r-t)$-set.
Then $N_{\mathcal{H}}(S)\backslash \mathcal{Q}_t$ is an intersecting $t$-family.
In particular, for $t=1$, we have $|N_{\mathcal{H}}(S)\backslash \mathcal{Q}_1|\leq 1$.
\end{lem}
\begin{proof}
Suppose for a contradiction that there exist two disjoint $t$-sets $T, T'\in N_{\mathcal{H}}(S)\setminus \mathcal{Q}_t$.

We claim that there exists some $(r-t)$-set $S'\in N_{\mathcal{G}}(T) \cap N_{\mathcal{G}}(T')$ with $S'\cap S = \emptyset$.
For the case $t=r-1$, since $T, T'\notin \mathcal{Q}_{r-1}$, we have $S':=\{n\}\in N_{\mathcal{G}}(T) \cap N_{\mathcal{G}}(T')$,
and clearly $S'\cap S=\emptyset$.
For $1\leq t\leq r-2$, by the definition of $\mathcal{Q}_t$, we have
$d_{\mathcal{G}}(T), d_{\mathcal{G}}(T') \geq \left( \frac{1}{2} + n^{-0.1} \right) \binom{n-1}{r-t-1},$
and hence $|N_{\mathcal{G}}(T) \cap N_{\mathcal{G}}(T')|\geq 2n^{-0.1} \binom{n-1}{r-t-1}$.
Note that any set $S'\in N_{\mathcal{G}}(T) \cap N_{\mathcal{G}}(T')$ has size $r-t$ and contains the fixed vertex $n$.
Thus at most $|S|\cdot \binom{n-1}{r-t-2}<n^{-0.1} \binom{n-1}{r-t-1}$ such $(r-t)$-sets $S'$ intersect $S$.
Thus there exists the desired $S'\in N_{\mathcal{G}}(T) \cap N_{\mathcal{G}}(T')$ with $S'\cap S=\emptyset$.

Using the claim, we see that the four edges $S\cup T$, $S \cup T'$, $S' \cup T$, and $S' \cup T'$ would form a generalized $4$-cycle in $\mathcal{F}=\mathcal{G} \cup \mathcal{H}$, a contradiction.
For the case $t=1$, $N_{\mathcal{H}}(S) \setminus \mathcal{Q}_1$ is a set of vertices, and thus being intersecting implies that its size is at most one.
\end{proof}

Before proving Lemma~\ref{lem:G-H-U-C4}, we record the simple consequence of Lemmas~\ref{lem:one-root} and~\ref{lem:CN-stability} that will supply the exceptional set $U$.

\begin{lem}\label{lem:U-set}
Let $\mathcal{F}$, $\mathcal{G}$ and $\mathcal{H}$ be defined as in the beginning of this section.
Suppose $W$ is a subset of $V(\mathcal{F})$ such that for any two vertices $x,y\in W$,
$\CN_{\mathcal{F}}(x,y)$ is an $(r-1)$-star with common center $n$.
Let $U=[n-1]\backslash W$.
Then for any $(r-1)$-set $S\subseteq [n-1]$, we have $|N_{\mathcal{H}}(S)\backslash U|\leq 1$.
\end{lem}

\begin{proof}
Suppose for a contradiction that $|N_{\mathcal{H}}(S)\backslash U|\geq 2$.
Then there exist two vertices $x,y\in N_{\mathcal{H}}(S)\cap W$.
By the hypothesis, $\CN_{\mathcal{F}}(x,y)$ is an $(r-1)$-star with center $n$.
Since $S\in \CN_{\mathcal{F}}(x,y)$, this would force $n\in S\subseteq [n-1]$, which is a contradiction.
\end{proof}

\subsection{The star-deficit lemma}\label{subsec:star-deficit}

The relevant obstruction is a Berge path of length two inside $\mathcal H$: if $\mathcal G$ were a full star, such a path would immediately create a generalized $4$-cycle.
When $\mathcal G$ is missing $m$ star edges, the auxiliary graphs $\mathcal Q_i$ measure how many of these paths can be hidden by missing star edges.
The lemma below counts such paths from above in terms of the exceptional set $U$ and from below under the assumption that $m$ is large, thereby showing that $\mathcal G$ remains close to a full $r$-star.

\begin{lem}\label{lem:G-H-U-C4}
Let $r\geq 3$ and $n$ be sufficiently large.
Let $\mathcal{F}$, $\mathcal{G}$, $\mathcal{H}$, and $m$ be defined as in the beginning of this section.
Suppose that $|\mathcal{H}|\geq \left(1-\frac{1}{2r}\right)m$ and there exists a set $U\subseteq [n-1]$ such that every $(r-1)$-set $S\subseteq [n-1]$ satisfies $|N_{\mathcal{H}}(S)\backslash U|\le 1$.
If $|U|\leq \frac{n}{50r^{4r+3}}$ and $|U|^{r-1}\leq \frac{m}{12r}$,
then $m\leq n\log n$.
\end{lem}

\begin{proof}
Suppose for a contradiction that $m>n\log n$.
We obtain a contradiction by double-counting the number of Berge paths of length two as follows.
Define
\[\mathcal{O}=\{(S,\{T,T'\}): T,T'\in \mathcal{Q}_{r-2}\mbox{ are disjoint}, |S|=2, \mbox{ and } S\cup T,S\cup T'\in \mathcal{H}\}.\]
We first prove an upper bound for the size of $\mathcal{O}$.

\begin{Claim}\label{clm:O-up-bd}
It holds that \[|\mathcal{O}|\leq r^{4r}n^{-2}m^{2}(|U|+3).\]
\end{Claim}
\begin{proof}
Fix a disjoint pair $\{T,T'\}\subseteq \mathcal{Q}_{r-2}$.
We assert that there are at most $|U|+3$ distinct 2-sets $S$ such that $(S,\{T,T'\})\in \mathcal{O}$, equivalently $\cd_{\mathcal{H}}(T,T')\leq |U|+3$.
To see this, observe that $\mathcal{H}$ is $\mathcal{C}_4^r$-free,
so the graph $\CN_{\mathcal{H}}(T,T')$ is either a triangle or a star.
If $\CN_{\mathcal{H}}(T,T')$ is a triangle, then clearly $\cd_{\mathcal{H}}(T,T')\leq 3$.
Assume now $\CN_{\mathcal{H}}(T,T')$ is a star with center $v$.
Then $\cd_{\mathcal{H}}(T,T') \leq d_{\mathcal{H}}(T \cup \{v\})\leq |U| + 1$,
where the last inequality follows from the hypothesis on $U$.
Thus, using \eqref{equ:Q-r-1} we have
\[|\mathcal{O}|\leq (|U|+3)\binom{|\mathcal{Q}_{r-2}|}{2}
\overset{\eqref{equ:Q-r-1}}{\leq} (|U|+3)\binom{r^{2r}n^{-1}m}{2}\le r^{4r}n^{-2}m^{2}(|U|+3),\]
as claimed.
\end{proof}

Before proceeding to the lower bound, we introduce some new notation.
Define
\[\Delta_0= 1 \text{ and } \Delta_t=\max_{S\in \binom{[n-1]}{r-t}}\{d_{\mathcal{H}}(S)\}\text{ for each } 1\leq t\leq r-1.\]

\begin{Claim}\label{clm:N-Q}
Let $1\leq t\leq r-1$. Then every $(r-t)$-set $S$ in $[n-1]$ satisfies
\[
|N_{\mathcal{H}}(S)\backslash \mathcal{Q}_t|\leq r \Delta_{t-1},
\mbox{ and moreover,~ } 
\Delta_t\leq r^{3r}n^{t+1-r}m+r^t.
\]
\end{Claim}
\begin{proof}
Fix an $(r-t)$-set $S\subseteq [n-1]$.
By Lemma~\ref{lem:N-Q-int}, $N_{\mathcal{H}}(S)\backslash \mathcal{Q}_t$ is an intersecting $t$-family.
If $t=1$, then $|N_{\mathcal{H}}(S)\backslash \mathcal{Q}_1|\leq 1\leq r\Delta_0$.
For $2\leq t\leq r-1$, there is nothing to prove if $N_{\mathcal{H}}(S)\backslash \mathcal{Q}_t=\emptyset$; otherwise, take one $t$-set $\{v_1,\ldots,v_t\}\in N_{\mathcal{H}}(S)\backslash \mathcal{Q}_t$.
Then every $t$-set in $N_{\mathcal{H}}(S)\backslash \mathcal{Q}_t$ contains some $v_i$ where $i\in [t]$,
implying that
\[|N_{\mathcal{H}}(S)\backslash \mathcal{Q}_t|\le \sum_{i\in [t]}|N_{\mathcal{H}}(S\cup \{v_i\})| \le t\Delta_{t-1} \le r \Delta_{t-1}.\]
Thus, $\Delta_t\leq |\mathcal{Q}_t| +r \Delta_{t-1}$ holds for $1\leq t\leq r-1$.
Using this inequality, an induction on $t$ gives
\[\Delta_t\leq \sum_{i=1}^{t} r^{t-i} |\mathcal{Q}_i|+r^{t}.\]
Using \eqref{equ:Q-r-1} that $|\mathcal{Q}_i|\leq r^{2r}mn^{i+1-r}$ for each $i$,
we derive that $\Delta_t\leq r^{3r}n^{t+1-r}m+r^t,$ as desired.
\end{proof}

The next claim supplies the mass needed for the lower bound on $|\mathcal{O}|$.

\begin{Claim}\label{clm:sum-N-Q}
It holds that
		\[\sum_{v\in [n-1]}|N_{\mathcal{H}}(v)\backslash \mathcal{Q}_{r-1}|\geq \frac{m}{3r}.\]
\end{Claim}
\begin{proof}
Let $\mathcal{H}'$ be obtained from $\mathcal{H}$ by deleting all edges $\{v\}\cup S$, where $S\in N_{\mathcal{H}}(v) \setminus \mathcal{Q}_{r-1}$ for every vertex $v\in [n-1]$.
We shall use the following property of $\mathcal{H}'$ that for every edge $T\in \mathcal{H}'$,
all its $(r-1)$-subsets belong to $\mathcal{Q}_{r-1}$. Also, we have
\begin{equation}\label{equ:Hprime}
|\mathcal{H}'|\geq |\mathcal{H}|- \sum_{v\in [n-1]}|N_{\mathcal{H}}(v)\backslash \mathcal{Q}_{r-1}|.
\end{equation}
Next, we define an $r$-partite $r$-graph $\mathcal{L}=\mathcal{L}(V_1,V_2,\ldots,V_r)$ from $\mathcal{H}'$, where the parts satisfy $V_1=V_2=\ldots=V_r=V(\mathcal{H})$, as follows:
for each edge $\{v_1,\ldots,v_r\}\in \mathcal{H}'$, we add all possible permutations $(v_{i_1},\ldots,v_{i_r})$ of it into $\mathcal{L}$.
Then $|\mathcal{L}|=r!\cdot |\mathcal{H}'|$.
Moreover, we define $\mathcal{L}_{r}=\mathcal{L}$, and
for every $i\in [r]$, let $\mathcal{L}_{i-1}$ denote the induced subgraph $\mathcal{L}(V_1,\ldots, V_{i-1}, U_i,\ldots, U_r)$ of $\mathcal{L}$,
where $U_i=\ldots=U_r=U$.
Note that $\mathcal{L}_0\subseteq \mathcal{L}_1\subseteq \ldots \subseteq \mathcal{L}_r=\mathcal{L}$,
where $\mathcal{L}_0=\mathcal{L}(U,\dots,U)$ and thus $|\mathcal{L}_0|=r!\cdot |\mathcal{H}'[U]|$.

We show that $|\mathcal{L}_{i-1}|\geq |\mathcal{L}_{i}|-(r-1)!m$ holds for every $2\leq i\leq r$,
and moreover, $|\mathcal{L}_{0}|\geq |\mathcal{L}_{1}|-|U|^{r-1}$.
To see this, consider any $(r-1)$-set $S=(v_1,\ldots, v_{i-1}, u_{i+1},\ldots, u_r)\in (V_1,\ldots, V_{i-1}, U_{i+1},\ldots, U_r)$ with $d_{\mathcal{L}_i}(S)\geq 1$, where $i\in [r]$.
Let $S^*=\{v_1,\ldots, v_{i-1}, u_{i+1},\ldots, u_r\}$.
Then $d_{\mathcal{H}'}(S^*)=d_{\mathcal{L}}(S)\geq d_{\mathcal{L}_i}(S)\geq 1$.
By the aforementioned property of $\mathcal{H}'$, we deduce that $S^*\in \mathcal{Q}_{r-1}$.
Thus, the number of such $S$ is at most $(r-1)!|\mathcal{Q}_{r-1}|=(r-1)!m$.
On the other hand,
$|N_{\mathcal{L}_{i-1}}(S)\backslash U|\leq |N_{\mathcal{H}'}(S^*)\backslash U|\leq |N_{\mathcal{H}}(S^*)\backslash U|\le 1$,
where the last inequality follows from the hypothesis on $U$.
Thus $\mathcal{L}_{i-1}$ is obtained from $\mathcal{L}_{i}$ by removing at most one edge for each such set $S$.
Thus $|\mathcal{L}_{i-1}|\geq |\mathcal{L}_{i}|-(r-1)!m$ holds for every $2\leq i\leq r$.
For $i=1$, the same argument applies, but the number of such $S$ is at most $|U|^{r-1}$.
So we obtain $|\mathcal{L}_{0}|\geq |\mathcal{L}_{1}|-|U|^{r-1}$.

Combining these inequalities, we obtain
\begin{equation}\label{equ:Hprime-U}
r!\cdot|\mathcal{H}'[U]|=|\mathcal{L}_0|\geq |\mathcal{L}|-(r-1)(r-1)!m-|U|^{r-1}=r!\cdot |\mathcal{H}'|-(r-1)(r-1)!m-|U|^{r-1}.
\end{equation}
Since $\mathcal{H}'[U]$ is a subgraph of $\mathcal{H}$, it is $\mathcal{C}_4^r$-free. By applying F\"uredi's theorem \cite{Fu84}, we obtain the upper bound
$|\mathcal{H}'[U]| \leq 3.5 \binom{|U|}{r-1}.$
This, together with \eqref{equ:Hprime-U}, implies that
\begin{equation}\label{equ:Hprime-2}
3.5r|U|^{r-1}\geq r!\cdot 3.5\binom{|U|}{r-1}\geq r!\cdot |\mathcal{H}'[U]|\geq r!|\mathcal{H}'|-(r-1)(r-1)!m-|U|^{r-1}.
\end{equation}
Since $r\geq 3$,  inequalities \eqref{equ:Hprime} and \eqref{equ:Hprime-2}  further imply that
\[2 |U|^{r-1}+\frac{r-1}{r}m\geq \frac{3.5r+1}{r!}|U|^{r-1}+\frac{r-1}{r}m \overset{\eqref{equ:Hprime-2}}{\geq} |\mathcal{H}'|\overset{\eqref{equ:Hprime}}{\geq} |\mathcal{H}|- \sum_{v\in [n-1]}|N_{\mathcal{H}}(v)\backslash \mathcal{Q}_{r-1}|.\]
Since $|\mathcal{H}|\geq \left(1-\frac1{2r}\right)m$ and $|U|^{r-1}\leq \frac{m}{12r}$,
we obtain
\[\sum_{v\in [n-1]}|N_{\mathcal{H}}(v)\backslash \mathcal{Q}_{r-1}|\geq |\mathcal{H}|-\frac{r-1}{r}m-2|U|^{r-1}\geq \frac{m}{2r}- 2|U|^{r-1}\geq \frac{m}{3r},\]
which proves Claim~\ref{clm:sum-N-Q}.
\end{proof}

Let $u\in [n-1]$ be a vertex.
In the next claim, we use the size of the $(r-1)$-graph $N_{\mathcal{H}}(u)\backslash \mathcal{Q}_{r-1}$ to give a lower bound on the size of the subfamily $\mathcal{O}(u)$ of $\mathcal{O}$.
Here, $\mathcal{O}(u)$ is defined as \[\mathcal{O}(u)=\{(S,\{T,T'\}) \in \mathcal{O}: u\in S\}.\]

\begin{Claim}\label{clm:O-lw-bd}
For every vertex $u\in [n-1]$, it holds that
\[|\mathcal{O}(u)|\geq \frac{1}{2r}|N_{\mathcal{H}}(u)\backslash \mathcal{Q}_{r-1}|^2-\frac{3}{2}r^2\Delta_{r-3}|N_{\mathcal{H}}(u)\backslash \mathcal{Q}_{r-1}|. \]		
\end{Claim}
\begin{proof}
By Lemma~\ref{lem:N-Q-int}, $N_{\mathcal{H}}(u)\backslash \mathcal{Q}_{r-1}$ is an intersecting $(r-1)$-family.
Fix a set $\{v_1,v_2,\ldots,v_{r-1}\}$ in $N_{\mathcal{H}}(u)\backslash \mathcal{Q}_{r-1}$.
Then every $(r-1)$-set in $N_{\mathcal{H}}(u)\backslash \mathcal{Q}_{r-1}$ contains some $v_i$.
For every $i\in [r-1]$, let
\[N^*_{\mathcal{H}}(uv_i)=\left\{T\in N_{\mathcal{H}}(uv_i): T\cup \{v_i\}\in N_{\mathcal{H}}(u)\backslash \mathcal{Q}_{r-1}\right\}\]
be a subset of $N_{\mathcal{H}}(uv_i)$.
Then we have
\begin{equation}\label{equ:N-uvi}
\max_{i\in [r-1]}|N^*_{\mathcal{H}}(u v_i)|\leq |N_{\mathcal{H}}(u)\backslash \mathcal{Q}_{r-1}|\leq \sum_{i\in [r-1]}|N^*_{\mathcal{H}}(u v_i)|.
\end{equation}
We also note that $(\{u,v_i\},\{T,T'\})\in \mathcal{O}(u)$ if and only if $T,T'\in N_{\mathcal{H}}(u v_i)\cap \mathcal{Q}_{r-2}$ form a disjoint pair.
Let $\alpha_i$ be the number of disjoint pairs in $N_{\mathcal{H}}(u v_i)\cap \mathcal{Q}_{r-2}$.
Then $|\mathcal{O}(u)|\geq \sum_{i=1}^{r-1} \alpha_i$.

Next, we estimate $\alpha_i$.
By Claim~\ref{clm:N-Q}, $|N_{\mathcal{H}}(uv_i)\backslash \mathcal{Q}_{r-2}|\leq r\Delta_{r-3}$ and thus
\begin{equation}\label{equ:N-uvi-2}
|N_{\mathcal{H}}(uv_i)\cap \mathcal{Q}_{r-2}|\geq |N_{\mathcal{H}}(u v_i)|- r\Delta_{r-3}\geq |N^*_{\mathcal{H}}(u v_i)|- r\Delta_{r-3}.
\end{equation}
Fix any set $T\in N_{\mathcal{H}}(uv_i)\cap \mathcal{Q}_{r-2}$ with $T=\{w_1,w_2,\ldots,w_{r-2}\}$.
Then the number of sets $T'\in N_{\mathcal{H}}(uv_i)\cap \mathcal{Q}_{r-2}$ with $T'\cap T \neq \emptyset$ is at most $\sum_{j\in [r-2]}|N_{\mathcal{H}}(u v_i w_j)|\leq (r-2)\Delta_{r-3}$.
Hence, $T$ is disjoint with at least $|N_{\mathcal{H}}(uv_i)\cap \mathcal{Q}_{r-2}|-r\Delta_{r-3}$ sets in $N_{\mathcal{H}}(uv_i)\cap \mathcal{Q}_{r-2}$.
Using \eqref{equ:N-uvi} and \eqref{equ:N-uvi-2}, this implies that
\begin{align*}
|\mathcal{O}(u)|\geq \sum_{i=1}^{r-1} \alpha_i &\geq \sum_{i=1}^{r-1}\frac12|N_{\mathcal{H}}(u v_i)\cap \mathcal{Q}_{r-2}|\cdot \big(|N_{\mathcal{H}}(uv_i)\cap \mathcal{Q}_{r-2}|-r\Delta_{r-3}\big)\\
		&\overset{\eqref{equ:N-uvi-2}}{\geq} \sum_{i=1}^{r-1}\frac12\big(|N^*_{\mathcal{H}}(u v_i)|- r\Delta_{r-3}\big)\cdot \big(|N^*_{\mathcal{H}}(u v_i)|-2r\Delta_{r-3}\big)\\
		&= \sum_{i=1}^{r-1}\frac12\big(|N^*_{\mathcal{H}}(uv_i) |^2 - 3r\Delta_{r-3}\cdot |N^*_{\mathcal{H}}(uv_i) |+2r^2\Delta_{r-3}^2\big)\\
 &\geq \frac{1}{2(r-1)}\left(\sum_{i\in [r-1]}|N^*_{\mathcal{H}}(u v_i)|\right)^2-\frac{3r\Delta_{r-3}}{2}\left(\sum_{i\in [r-1]}|N^*_{\mathcal{H}}(u v_i)|\right)\\
 &\overset{\eqref{equ:N-uvi}}{\geq} \frac{1}{2(r-1)}|N_{\mathcal{H}}(u)\backslash \mathcal{Q}_{r-1}|^2-\frac{3(r-1)r\Delta_{r-3}}{2}|N_{\mathcal{H}}(u)\backslash \mathcal{Q}_{r-1}|.
\end{align*}
This proves Claim~\ref{clm:O-lw-bd}.
\end{proof}

Finally, we combine the preceding estimates to obtain a contradiction for $|\mathcal{O}|$.
By Claim~\ref{clm:O-lw-bd}, we see that $|\mathcal{O}|=\frac12\sum_{u\in [n-1]}|\mathcal{O}(u)|$ is at least
\begin{align*}	
&\frac{1}{4r}\cdot\left(\sum_{u\in [n-1]}|N_{\mathcal{H}}(u)\backslash \mathcal{Q}_{r-1}|^2\right)-\frac{3r^2}{4}\cdot\Delta_{r-3}\cdot \left(\sum_{u\in [n-1]}|N_{\mathcal{H}}(u)\backslash \mathcal{Q}_{r-1}|\right)\\
\geq & \frac{n-1}{4r}\cdot\left(\frac{\sum_{u\in [n-1]}|N_{\mathcal{H}}(u)\backslash \mathcal{Q}_{r-1}|}{n-1}\right)^2-\frac{3r^2}{4}\cdot \Delta_{r-3}\cdot\left(\sum_{u\in [n-1]}|N_{\mathcal{H}}(u)\backslash \mathcal{Q}_{r-1}|\right)
\end{align*}
From Claims~\ref{clm:N-Q} and \ref{clm:sum-N-Q},
we have $\Delta_{r-3}\leq r^{3r}n^{-2}m+r^{r-3}$ and $\sum_{u\in [n-1]}|N_{\mathcal{H}}(u)\backslash \mathcal{Q}_{r-1}|\geq \frac{m}{3r}$.
Since we assume $m>n\log n$,
we obtain that for $n \rightarrow \infty$,
$$(n-1)\Delta_{r-3} \le r^{3r}n^{-1}m + n r^{r-3}=o(m),$$
and thus
\begin{equation*}
\frac{\sum_{u\in [n-1]}|N_{\mathcal{H}}(u)\backslash \mathcal{Q}_{r-1}|}{(n-1)\cdot \Delta_{r-3}}\geq \frac{m}{3r(n-1)\cdot \Delta_{r-3}}\to \infty \mbox{ as } n\to \infty.
\end{equation*}
Since $n$ is sufficiently large,
we obtain
\begin{align*}
|\mathcal{O}|\geq \big(1-o_n(1)\big)\cdot \frac{n-1}{4r}\cdot \left(\frac{ \sum_{v\in [n-1]}|N_{\mathcal{H}}(v)\backslash \mathcal{Q}_{r-1}|}{n-1} \right)^2 \geq \big(1-o_n(1)\big)\cdot\frac{n-1}{4r}\frac{m^2}{9r^2(n-1)^2}\geq \frac{m^2}{40r^3n}.
		\end{align*}
Together with Claim~\ref{clm:O-up-bd}, this gives
\[r^{4r}n^{-2}m^{2}(|U|+3)\geq |\mathcal{O}|\geq \frac{m^2}{40r^3n}, \mbox{ and hence } |U|\geq \frac{n}{40r^{4r+3}}-3.\]
This contradicts the condition that $|U|\leq \frac{n}{50r^{4r+3}}$, thus proving Lemma~\ref{lem:G-H-U-C4}.
\end{proof}

We remark that the proof of Lemma~\ref{lem:G-H-U-C4} in fact implies that $m \leq f(n)$ for any function $f(n)$ satisfying $\frac{f(n)}{n} \to \infty$ as $n \to \infty$. Here, for later convenience, we set $f(n)=n\log n$.

\subsection{Proof of Theorem~\ref{thm:stability}}
Combining Lemma~\ref{lem:G-H-U-C4} with Lemma~\ref{lem:one-root} and \ref{lem:CN-stability} that we will establish in later sections, we derive Theorem~\ref{thm:stability}, a stability result for $\mathcal{C}_4^r$. This resolves a conjecture of Mubayi in \cite{Mu07}.

We first apply Lemma~\ref{lem:one-root} to make local neighborhoods stars, then Lemma~\ref{lem:CN-stability} to align their centers on a large set.
Lemma~\ref{lem:U-set} converts this alignment into the exceptional-set hypothesis of Lemma~\ref{lem:G-H-U-C4}, forcing the star deficit to be small.

\begin{proof}[\bf Proof of Theorem~\ref{thm:stability}]
Let $r\geq 4$ be fixed. For every $\delta > 0$, we aim to prove the existence of $\epsilon = \epsilon(\delta) > 0$ and an integer $n_0 = n_0(\delta)$ such that the following holds:
Let $n\geq n_0$ and let $\mathcal{F}^*$ be any $\mathcal{C}_4^r$-free $r$-graph with vertex set $[n]$ and with at least $(1-\epsilon)\binom{n-1}{r-1}$ edges. Then there exists a vertex $u\in [n]$ that belongs to at least $(1-\delta)\binom{n-1}{r-1}$ edges of $\mathcal{F}^*$.

Throughout the rest of the proof, assume $n$ is sufficiently large compared with $r,\delta$, and choose \[\epsilon=\min\left\{\frac{\delta}{2r},\frac{1}{21000r^{4r+4}},\left(\frac{1}{6\cdot 2^{r-1}(r-1)!(420r)^{r-1}}\right)^{1/(r-2)}\right\}.\]
By Lemma~\ref{lem:one-root} (choosing $\delta$ therein as $1/10$), $\mathcal{F}^*$ contains a spanning subgraph $\mathcal{F}$ such that \[|\mathcal{F}|\geq |\mathcal{F}^*| - 4n^{r-1-1/10}\geq \left(1-\frac{\epsilon}2\right)\binom{n-1}{r-1}\] and for any two vertices $x$ and $y$, $\CN_{\mathcal{F}}(x,y)$ is an $(r-1)$-star.
Since $\epsilon/2<1/420r$, by Lemma~\ref{lem:CN-stability},
there exists a vertex say $n$ and a subset $W\subseteq [n]$ such that $|W|\geq (1-420\epsilon r)n$ and for any two vertices $x,y\in W$,
$\CN_{\mathcal{F}}(x,y)$ is an $(r-1)$-star with common center $n$.

Following the initial notation from the beginning of this section,
let $\mathcal{F} = \mathcal{G} \cup \mathcal{H}$, where $\mathcal{G}$ is the maximal $r$-star with center $n$ contained in $\mathcal{F}$
and $\mathcal{H}$ consists of the remaining edges of $\mathcal{F}$.
Let $|\mathcal{G}| = \binom{n-1}{r-1} - m$.
If $m<2\epsilon r \binom{n-1}{r-1}\leq \delta\binom{n-1}{r-1}$,
then $n$ is the desired vertex. Thus we may assume that
\begin{equation}\label{equ:m-stability}
m\geq 2\epsilon r\binom{n-1}{r-1}.
\end{equation}
Let $U=[n-1]\backslash W$.
By Lemma~\ref{lem:U-set}, any $S\in \binom{[n-1]}{r-1}$ satisfies $|N_{\mathcal{H}}(S)\backslash U|\leq 1$.
Moreover,
\[
 |U|\leq n-|W|\leq 420\epsilon rn\leq \frac{n}{50r^{4r+3}},
\]
and
\[
 |U|^{r-1}\leq (420\epsilon rn)^{r-1}
 \leq \frac{\epsilon(n/2)^{r-1}}{6(r-1)!}
 \leq \frac{\epsilon}{6}\binom{n-1}{r-1}
 \leq \frac{m}{12r},
\]
where the last inequality uses \eqref{equ:m-stability}.
Moreover, using \eqref{equ:m-stability} again, we obtain
\[|\mathcal{H}|=|\mathcal{F}|-|\mathcal{G}|\geq m-\frac{\epsilon}{2}\binom{n-1}{r-1}\geq \left(1-\frac{1}{2r}\right)m.\]
Applying Lemma~\ref{lem:G-H-U-C4}, we obtain $m\leq n\log n$, contradicting \eqref{equ:m-stability}, completing the proof.
\end{proof}

\subsection{Proof of Theorem~\ref{thm:extremal}}\label{subsec:pf-extremal}

The stability in Theorem~\ref{thm:stability} shows that an extremal hypergraph must always contain an almost full star. We will use
the $\mathcal Q_i$ hierarchy and the exact Berge-path count to force the star deficit $m$ to equal either $0$ or $1$, 
corresponding to constructions (C1) and (C2) respectively.
This approach proves our main Theorem~\ref{thm:extremal},
determinining the exact Tur\'{a}n number $f_r(n)$ of $\mathcal{C}_4^r$ and characterizing all extremal hypergraphs for any $r \geq 4$.

\begin{proof}[\bf Proof of Theorem~\ref{thm:extremal}.]
Let $r\geq 4$.
We first observe that by the construction (C1), it holds that $f_r(n)\geq \binom{n-1}{r-1}+\left\lfloor \frac{n-1}{r} \right\rfloor$.
Let $\mathcal{F}$ be any $n$-vertex $\mathcal{C}_4^r$-free $r$-graph with vertex set $[n]$ and at least $\binom{n-1}{r-1} + \left\lfloor \frac{n-1}{r} \right\rfloor$ edges.
It remains to show that, for every sufficiently large $n\geq N_0$,
$\mathcal{F}$ must be isomorphic to one of the constructions (C1) or (C2) (and thus $|\mathcal{F}|=\binom{n-1}{r-1} + \left\lfloor \frac{n-1}{r} \right\rfloor$).

Select a positive constant $\delta$ that is sufficiently small relative to $r$, and choose $N_0$ sufficiently large such that $N_0 \geq n_0(\delta)$,
where $n_0(\delta)$ is the constant provided by Theorem~\ref{thm:stability}.
By Theorem~\ref{thm:stability}, there exists a vertex (without loss of generality, say $n$) that belongs to at least $(1 - \delta)\binom{n-1}{r-1}$ edges of $\mathcal{F}$.
Partition $\mathcal{F} = \mathcal{G} \cup \mathcal{H}$ as before,
where $\mathcal{G}$ is the maximal $r$-star with center $n$ contained in $\mathcal{F}$,
and $\mathcal{H}$ consists of the remaining edges of $\mathcal{F}$.
Let $|\mathcal{G}| = \binom{n-1}{r-1} - m$.
Summarizing, we have
\begin{equation}\label{equ:HplusG}
|\mathcal{H}|\geq m+ \left\lfloor \frac{n-1}{r} \right\rfloor,~~~
|\mathcal{G}|\geq (1 - \delta)\binom{n-1}{r-1}, \mbox{~~~and~~~} 0\leq m\leq \delta\binom{n-1}{r-1}.
\end{equation}
A sequence of claims will show that $m$ must be either $0$ or $1$.

In the first claim, we use Lemma~\ref{lem:G-H-U-C4} to prove that $m \leq n \log n$.
Recall the $i$-graphs $\mathcal{Q}_i$ defined in Subsection~\ref{subsec:Qi} for $1\leq i\leq r-1$.

\begin{Claim}\label{clm:mltnlogn}
It holds that $m \leq n \log n$.
\end{Claim}
\begin{proof}
Let $U=\mathcal{Q}_1$, which is a subset of $[n-1]$.
We check that $U$ satisfies the conditions in Lemma~\ref{lem:G-H-U-C4}.
By Lemma~\ref{lem:N-Q-int}, we see that every $(r-1)$-set $S\subseteq [n-1]$ satisfies $|N_{\mathcal{H}}(S)\backslash U|\leq 1$.
Using \eqref{equ:Q-r-1} and the fact that $m \leq \delta\binom{n-1}{r-1}$, where $\delta$ is chosen sufficiently small in terms of $r$, we have
\[
|U| = |\mathcal{Q}_1|\overset{\eqref{equ:Q-r-1}}{\leq} r^{2r} n^{2 - r} m < r^{2r} \delta n \leq \frac{n}{50r^{4r + 3}}\mbox{~~and consequently,~~}|U|^{r-1}\leq (r^{2r} n^{2 - r} m)^{r-1}\leq \frac{m}{12r}.
\]
Also $|\mathcal{H}|\geq \left(1-\frac{1}{2r}\right)m$.
Hence, by Lemma~\ref{lem:G-H-U-C4}, we obtain $m\leq n\log n$.
\end{proof}

Using \eqref{equ:Q-r-1} and Claim~\ref{clm:mltnlogn}, for any $1\leq i\leq r-3$ and sufficiently large $n$, we have
\begin{equation}\label{equ:Qieqempty}
|\mathcal{Q}_{i}|\overset{\eqref{equ:Q-r-1}}{\leq} r^{2r}mn^{i+1-r}\leq r^{2r}n^{-1}\log n<1, \mbox{ implying that } \mathcal{Q}_{i}=\emptyset \text{ for every } 1\leq i\leq r-3.
\end{equation}
We will also require the analysis on $\mathcal{Q}_{r-2}$ and $\mathcal{Q}_{r-1}$.
Before proceeding, we record a degree estimate for subsets of $\mathcal{H}$.

\begin{Claim}\label{clm:dT-1}
Let $3 \leq i \leq r-1$. Then every $i$-set $T \subseteq [n-1]$ satisfies $d_{\mathcal{H}}(T) \leq 1$.
\end{Claim}
\begin{proof}
Suppose, for a contradiction, that there exists an $i$-set $T$ with $3\leq i \leq r-1$ such that $d_{\mathcal{H}}(T) \geq 2$.
Then there exist two distinct subsets $A, B\in N_{\mathcal{H}}(T)$.
Let $S=(A\cap B)\cup T$.
Then $3\leq |T|\leq |S|\leq r-1$. Moreover, $A\setminus B$ and $B\setminus A$ are disjoint members of $N_{\mathcal{H}}(S)$, so $N_{\mathcal{H}}(S)$ contains two disjoint sets.
Note that $1\leq r-|S|\leq r-3$, so by \eqref{equ:Qieqempty}, we have $\mathcal{Q}_{r-|S|}=\emptyset$.
However, Lemma~\ref{lem:N-Q-int} implies $N_{\mathcal{H}}(S)=N_{\mathcal{H}}(S)\backslash \mathcal{Q}_{r-|S|}$ is an intersecting family, leading to a contradiction. This completes the proof.
\end{proof}

Recall that $|\mathcal{Q}_{r-1}|=m$.
By Claim~\ref{clm:dT-1}, every $(r-1)$-set $T \in \mathcal{Q}_{r-1}$ is contained in at most one edge in $\mathcal{H}$,
that is, at most one of the sets $N_{\mathcal{H}}(v)$ for $v\in [n-1]$.
Hence we have
\begin{equation}\label{equ:rH-m}
\sum_{v\in [n-1]}|N_{\mathcal{H}}(v )\backslash \mathcal{Q}_{r-1}|
\geq \left(\sum_{v\in [n-1]}|N_{\mathcal{H}}(v )|\right)-| \mathcal{Q}_{r-1}|
=r|\mathcal{H}|-m.
\end{equation}

In the remainder of the proof, we consider the following collection of ``edges'' in $\mathcal{H}$.
Define \[\mathcal{\widetilde O}=\{(S,T): |S|=2, ~ T\in \mathcal{Q}_{r-2}, \text{ and } S\cup T\in \mathcal{H}\}.\]

We establish upper and lower bounds on $|\mathcal{\widetilde{O}}|$ in the following two claims.

\begin{Claim}\label{clm:Oprime-up-bd}
It holds that $|\mathcal {\widetilde O}| \leq |\mathcal{Q}_{r-2}|(|\mathcal{Q}_2|+1).$
\end{Claim}
\begin{proof}
For each $T \in \mathcal{Q}_{r-2}$, Lemma~\ref{lem:N-Q-int} implies that $N_{\mathcal{H}}(T) \setminus \mathcal{Q}_2$ is an intersecting 2-family.
If $N_{\mathcal{H}}(T) \setminus \mathcal{Q}_2$ contains two sets, say $\{a,b\}$ and $\{a,c\}$,
then $d_{\mathcal{H}}(T\cup \{a\}) \geq 2$,
contradicting Claim~\ref{clm:dT-1}.
Thus, $|N_{\mathcal{H}}(T)| \leq |\mathcal{Q}_2| + 1$.
In total, we have
$|\mathcal{\widetilde{O}}| \leq |\mathcal{Q}_{r-2}| (|\mathcal{Q}_2| + 1)$, as claimed.
\end{proof}

\begin{Claim}\label{clm:Oprime-lw-bd}
It holds that $2|\mathcal{\widetilde O}|\geq r|\mathcal{H}|-m-n+1.$
\end{Claim}
\begin{proof}
For a vertex $v\in [n-1]$, set
\[\mathcal{\widetilde O}(v)=\{(S,T) \in \mathcal{\widetilde O} : v\in S\}.\]

We will show that, with at most one exception, each $W\in N_{\mathcal{H}}(v) \setminus \mathcal{Q}_{r-1}$ contains a vertex $w$ such that $(\{v, w\}, W\setminus \{w\}) \in \mathcal{\widetilde{O}}(v)$.
Assume that there exist two exceptions $W_1,W_2\in N_{\mathcal{H}}(v)\backslash \mathcal{Q}_{r-1}$.
By Lemma~\ref{lem:N-Q-int}, $N_{\mathcal{H}}(v)\backslash \mathcal{Q}_{r-1}$ is an intersecting family.
If $|W_1\cap W_2|\ge 2$, then the degree $d_{\mathcal{H}}((W_1\cap W_2)\cup \{v\})\ge 2$ with $3\le |(W_1\cap W_2)\cup \{v\}|\le r-1$,
which contradicts Claim~\ref{clm:dT-1}.
Consequently, we have $W_1\cap W_2 =\{w\}$ for some vertex $w$.
By Lemma~\ref{lem:N-Q-int}, $N_{\mathcal{H}}(\{v,w\})\backslash \mathcal{Q}_{r-2}$ is an intersecting family.
Since $W_1\backslash \{w\}, W_2\backslash \{w\}\in N_{\mathcal{H}}(\{v,w\})$ are disjoint,
one of them must be contained in $\mathcal{Q}_{r-2}$, say $W_1\backslash \{w\}$.
It follows that $(\{v,w\},W_1\backslash \{w\})\in \mathcal{\widetilde O}(v)$, which is a contradiction.

Note that the pairs $(\{v, w\}, W \setminus \{w\}) \in \mathcal{\widetilde{O}}(v)$ corresponding to distinct $W \in N_{\mathcal{H}}(v) \setminus \mathcal{Q}_{r-1}$ are clearly distinct.
Hence
$|\mathcal{\widetilde O}(v)|\ge |N_{\mathcal{H}}(v)\backslash \mathcal{Q}_{r-1}|-1$.
Summing up, we have
\[2|\mathcal{\widetilde O}|=\sum_{v\in [n-1]}|\mathcal{\widetilde O}(v)|
\geq \sum_{v\in [n-1]}(|N_{\mathcal{H}}(v)\backslash \mathcal{Q}_{r-1}|-1)
\geq r|\mathcal{H}|-m-(n-1),\]
where the last inequality uses \eqref{equ:rH-m}. This proves the claim.
\end{proof}

Combining Claims~\ref{clm:Oprime-up-bd} and \ref{clm:Oprime-lw-bd} and applying \eqref{equ:HplusG}, we obtain
\begin{equation}\label{equ:main-thm-last}
2|\mathcal{Q}_{r-2}|(|\mathcal{Q}_2|+1)\geq 2|\mathcal{\widetilde O}|\geq r|\mathcal{H}|-m-n+1\overset{\eqref{equ:HplusG}}{\geq} (r-1)m + r\left\lfloor \frac{n-1}{r} \right\rfloor-n+1.
\end{equation}
Note that $r\geq 4$ and by Claim~\ref{clm:mltnlogn}, $m\leq n\log n$.
If $m\geq n^{0.5}$, then by \eqref{equ:Q-r-1},
$|\mathcal{Q}_2|\leq r^{2r}n^{3-r}m\leq n^{0.2}$ and $|\mathcal{Q}_{r-2}|\leq r^{2r}n^{-1}m\leq n^{0.2}$.
Hence, we obtain $2|\mathcal{Q}_{r-2}|(|\mathcal{Q}_2|+1)< n^{0.5}<(r-1)m + r\left\lfloor \frac{n-1}{r} \right\rfloor -n+1$,
contradicting \eqref{equ:main-thm-last}.
Therefore, we may assume that $0\leq m< n^{0.5}$.
It follows from \eqref{equ:Q-r-1} that $|\mathcal{Q}_{r-2}|\le r^{2r}n^{-1}m<1$ and thus $\mathcal{Q}_{r-2}=\emptyset$.
Then, \eqref{equ:main-thm-last} becomes that
\begin{equation}\label{equ:main-thm-last2}
0\geq r|\mathcal{H}|-m-n+1\geq (r-1)m + r\left\lfloor \frac{n-1}{r} \right\rfloor-n+1\geq (r-1)m + (n-r) -n+1=(r-1)(m-1).
\end{equation}

Thus $m\in \{0,1\}$.
We analyze the cases $m = 0$ and $m = 1$ separately.
First, consider the case when $m = 0$. In this case, $\mathcal{G}$ is a full $r$-star with center $n$, and \eqref{equ:HplusG} implies $|\mathcal{H}| \geq \left\lfloor \frac{n-1}{r} \right\rfloor $, whereas \eqref{equ:main-thm-last2} yields $|\mathcal{H}| \leq \frac{n-1}{r}$.
Thus, we conclude that $|\mathcal{H}|=\left\lfloor \frac{n-1}{r} \right\rfloor$.
Note that $\mathcal{H}\subseteq \binom{[n-1]}{r}$.
If $\mathcal{H}$ contains two distinct sets $A,B$ with $A\cap B\neq \emptyset$, then the fact that $\mathcal{G}$ is a full $r$-star would enforce a generalized $4$-cycle in $\mathcal{F}=\mathcal{G}\cup \mathcal{H}$, a contradiction.
Therefore, $\mathcal{H}$ must consist of $\left\lfloor \frac{n-1}{r} \right\rfloor$ disjoint sets in $[n-1]$.
Consequently, $\mathcal{F}$ is isomorphic to construction (C1).

It remains to treat the case $m = 1$.
Here, $\mathcal{G}$ is a full $r$-star missing one edge, which we may assume is $\{n - r + 1, \ldots, n\}$.
All inequalities in \eqref{equ:main-thm-last2} must hold as equalities, which implies $r|\mathcal{H}| = n$ and thus $r \mid n$.
It follows that $|\mathcal{H}| = \frac{n}{r} = \left\lfloor \frac{n - 1}{r} \right\rfloor + 1$.
Since $\mathcal{H} \subseteq \binom{[n - 1]}{r}$, there must exist a pair of intersecting edges, say $S, S' \in \mathcal{H}$.
Let $A = S \cap S'$ be non-empty.
If $|A| = 1$, set $B = \emptyset$;
otherwise $|A| \geq 2$, then for sufficiently large $n$, there exist many disjoint subsets $B \subseteq [n - 1] \setminus (S \cup S')$ with $|B| = |A| - 1$.
Define $T_B = \{n\} \cup B \cup (S \setminus A)$ and $T'_B = \{n\} \cup B \cup (S' \setminus A)$.
Then $T_B$, $T'_B$, $S$, and $S'$ form a generalized $4$-cycle.
Since $\mathcal{F} = \mathcal{G} \cup \mathcal{H}$ is $\mathcal{C}_4^r$-free, this forces, for every choice of $B$, one of $T_B$ or $T'_B$ (say $T_B$) to be the missing edge $\{n - r + 1, \ldots, n\}$ of $\mathcal{G}$.
This further implies that $B$ must be uniquely determined, i.e., $B = \emptyset$.
Therefore, $A = \{w\}$ and $S = \{n - r + 1, \ldots, n - 1, w\}$ for some vertex $w \in [n - r]$.
The above argument also shows that $S$ and $S'$ form the unique pair of intersecting edges in $\mathcal{H}$.
Since $r \mid n$ and $|\mathcal{H}| = \left\lfloor \frac{n - 1}{r} \right\rfloor + 1$,
we conclude that $\mathcal{H}$ consists of $\left\lfloor \frac{n - 1}{r} \right\rfloor$ disjoint $r$-sets in $[n - r]$ and the set $S = \{n - r + 1, \ldots, n - 1, w\}$ for some vertex $w \in [n - r]$.
Thus, $\mathcal{G} \cup \mathcal{H}$ is isomorphic to construction (C2).
\end{proof}

\section{Local starification}\label{sec:intersecting-families}
In this section, we prove Lemma~\ref{lem:one-root}. 
Let $r\geq 4$ and $\fF$ be any $\mathcal{C}_4^r$-free $r$-graph with vertex-set $[n]$.
Our goal is to show that, by deleting a set $\mathcal{R}$ of $n^{\,r-1-\delta}$ edges for some $\delta > 0$, 
any two vertices $x$ and $y$ have the property that 
the $(r-1)$-graph $\CN_{\fF \setminus \mathcal{R}}(x,y)$ is a star.

Before proceeding, we provide a detailed proof overview.
A key starting point is that, since $\fF$ is $\mathcal{C}_4^r$-free, the common neighborhood $\CN_{\fF}(x,y)$ of any pair of vertices $x$ and $y$ is an intersecting family in $\binom{[n]}{r-1}$. 
So it suffices for us to reduce the covering number of every intersecting family $\CN_{\fF}(x,y)$ to one (i.e., to a star) by deleting a small number of edges. 
To do so, we choose a minimal subgraph $\gG\subseteq\fF$ using the potential function that will be defined in \eqref{equ:g-function}, which rewards pairs of vertices whose common neighborhoods have small covering number.
This potential function rules out covering number at least three (see Claim~\ref{clm:P3}) and shows that pairs of vertices with covering number two are sparse (see Claim~\ref{clm:No-P2}).
After deleting all edges arising from the common neighborhoods of these sparse pairs, every common neighborhood of a pair of vertices in the remaining subgraph is a star.

With the above proof strategy in mind, we first state and prove several results on intersecting families in Subsection~\ref{Subsec:intersecting-families}, and then present the proof of Lemma~\ref{lem:one-root} in Subsection~\ref{Subsec:one-root}.

\subsection{Intersecting-family tools}\label{Subsec:intersecting-families}
We first state two classical results on intersecting families.

\begin{thm}[Erd\H{o}s-Ko-Rado Theorem \cite{EKR61}]\label{thm:EKR}
	Let $n\ge 2k$ and let $\fF\subseteq \binom{[n]}{k}$ be an intersecting family. Then $|\fF|\le \binom{n-1}{k-1}.$
\end{thm}

\begin{thm}[Hilton-Milner Theorem \cite{HM67}] \label{thm:HM}
	Let $n>2k$ and let $\fF\subseteq \binom{[n]}{k}$ be an intersecting family which is not a star.
	Then $|\fF|\le \binom{n-1}{k-1}- \binom{n-k-1}{k-1}+1.$
\end{thm}

The following result, due to Frankl~\cite{Fr80}, provides upper bounds on the size of intersecting families with covering number at least three (see Theorem~7.4 in the survey of Frankl and Tokushige~\cite{FT16}). 
Here, the $o(1)$ term denotes a function that tends to $0$ as $n\to\infty$ while $k$ is fixed.
	
\begin{thm}[Frankl \cite{Fr80}]\label{lem:covering-number-3}
 Let $n>k\geq 3$ and let $\fF\subseteq \binom{[n]}{k}$ be an intersecting family with covering number at least three. 
 Then, for $k=3$, we have $|\fF|\leq k^k$, and for $k\ge 4$, 
 \[|\fF|\le (k^2-k+1+o(1))\cdot \binom{n-3}{k-3}.\]
\end{thm}

We prove two additional lemmas on intersecting families that will be needed in our proof. 
The first lemma concerns intersecting families with covering number two.

\begin{lem}\label{lem:cover2-abc}
Let $n>2k$, $k\ge 3$ and let $\fF \subseteq \binom{[n]}{k}$ be an intersecting family with covering number two. 
If no element is contained in all but at most $k^k \binom{n}{k-3}$ members of $\fF$, 
then there exist three elements such that every set in $\fF$ contains at least two of them.
\end{lem}

\begin{proof}
Let $\{a,b\}$ be a cover of $\fF$. 
Define 
\[\fF(a)=\{S\in \fF: a\in S,~ b\notin S\} \text{ ~and~ } \fF(b)=\{T\in \fF: b\in T,~ a\notin T\}.\]
If $|\fF(a)| < k^k \binom{n}{k-3}$, then $b$ would be contained in all but at most $k^k \binom{n}{k-3}$ members of $\fF$, contradicting our assumption. 
Hence, $|\fF(a)|\ge k^k\binom{n}{k-3}$ and similarly, we have $|\fF(b)|\ge k^k\binom{n}{k-3}$.

We claim that there exists an element $c\in [n]\setminus \{a,b\}$ such that every set in $\fF(a)$ contains $c$.
Let $\overline{\fF}(a)=\{S\setminus \{a\}: S\in \fF(a)\}$.
It is equivalent to showing that $\overline{\fF}(a)$ is a star.
First, suppose that $\overline{\fF}(a)$ has two disjoint $(k-1)$-sets $S_1$ and $S_2$. 
Since $\fF$ is intersecting, 
every $k$-set $T \in \fF(b)$ satisfies $T\cap S_1\neq \emptyset$ and $T\cap S_2\neq \emptyset$. 
Thus $|\fF(b)|\le |S_1|\cdot |S_2|\cdot \binom{n}{k-3}<k^k\binom{n}{k-3}$, which contradicts the property established in the previous paragraph.
Therefore, $\overline{\fF}(a)$ is an intersecting family in $\binom{[n]}{k-1}$. 
If $\overline{\fF}(a)$ is not a star, then by Theorem~\ref{thm:HM}, we obtain
\[|\fF(a)|=|\overline{\fF}(a)|\le \binom{n-1}{k-2}- \binom{n-k}{k-2}+1<k^k\binom{n}{k-3},\] 
which yields another contradiction. This proves the claim.

Since $|\fF(a)|\geq k^k \binom{n}{k-3}$ and each set in $\fF(a)$ contains $a$ and $c$, 
we greedily choose $k$ sets $S_1, \dots,S_k \in \fF(a)$ such that $S_i\setminus \{a,c\}$ are pairwise disjoint for all $1\leq i \ne j\leq k$. 
For each $k$-set $T\in \fF(b)$, since $(T\setminus \{b\})\cap S_i\neq \emptyset$ for $i\in [k]$, it forces $c\in T$. 
Note that every set in $\fF\setminus (\fF(a)\cup \fF(b))$ contains both $a$ and $b$. 
Consequently, every set in $\fF$ contains at least two elements in $\{a,b,c\}$.
\end{proof}

The last lemma in this subsection is a straightforward consequence of the Erd\H{o}s-Ko-Rado theorem.

\begin{lem}\label{lem:counting-C4-pairs}
Let $k\ge 2$ and $0\le \delta<1/2$. For sufficiently large $n$, 
if $\fF$ is an $n$-vertex $k$-graph with at least $n^{k-\delta}$ edges, then there are at least $0.2n^{2-2\delta}$ pairs of vertices $\{u,v\}$ such that $\CN_{\fF}(u,v)$ is not intersecting.
\end{lem}
\begin{proof}
Let $u,v\in V(\fF)$.
By Theorem~\ref{thm:EKR}, if $\cd_{\fF}(u,v)> \binom{n-2}{k-2}\ge \binom{n-3}{k-2}$, then $\CN_{\fF}(u,v)$ is not intersecting.
Thus, it suffices to count pairs $\{u,v\}$ with $\cd_{\fF}(u,v)> \binom{n-2}{k-2}$. 
As $|\fF|\geq n^{k-\delta}$,
we have
\[
\sum_{\{u,v\}\in \binom{V(\fF)}{2}} \cd_{\fF}(u,v)
= \sum_{T\in \binom{V(\fF)}{k-1} } \binom{d_{\fF} (T)}{2} 
\ge \binom{n}{k-1}\binom{ k|\fF| /\binom{n}{k-1} }{2} \ge \frac 12 |\fF|^2/\binom{n}{k-1} \ge \frac 12 n^{k+1-2\delta}.
\]
Since $\max_{\{u,v\}} \cd_{\fF}(u,v)\leq \binom{n-2}{k-1}$, the number of pairs $\{u,v\}$ with $\cd_{\fF}(u,v)> \binom{n-2}{k-2}$ is at least 
\[
\frac{\frac{1}{2} n^{\,k+1-2\delta} - \binom{n}{2}\binom{n-2}{k-2}}{\binom{n-2}{k-1}} \geq 0.2 n^{\,2-2\delta}.
\] 
This completes the proof.
\end{proof}
	
\subsection{Proof of Lemma~\ref{lem:one-root}}\label{Subsec:one-root}
Before presenting the proof of Lemma~\ref{lem:one-root}, we introduce some notation.
Throughout the rest of this section,
let $r\geq 4$ and let $\fF$ be an $n$-vertex $\mathcal{C}_4^r$-free $r$-graph.
For a subgraph $\gG \subseteq \fF$, define 
	\[\mathcal P (\gG)=\left\{\{x,y\}\in \binom{V(\gG)}{2}: \cd_{\gG}(x,y)\ge 1\right\} 
	\text{ and } 
	\mathcal P_0 (\gG)=\left\{\{x,y\}\in \binom{V(\gG)}{2}: \cd_{\gG}(x,y)=0\right\}.\]
For any vertices $x$ and $y$, let $\fR_{\gG}(x,y)$ denote a minimum vertex cover of $\CN_{\gG}(x,y)$. 
Note that $\CN_{\gG}(x,y)$ may have several minimum vertex covers, and we fix one arbitrarily.
We refer to the elements of $\fR_{\gG}(x,y)$ as the {\bf roots} of $\{x,y\}$.
According to the sizes of the minimum covers considered, we partition $\mathcal{P}(\gG)$ into the following three disjoint classes:
	\[\mathcal P_1(\gG)=\{\{x,y\}\in \mathcal{P}(\gG): |\fR_{\gG}(x,y)|= 1 \},\]
	\[\mathcal P_2(\gG)=\{\{x,y\}\in \mathcal{P}(\gG): |\fR_{\gG}(x,y)|= 2 \},\]
	\[\mathcal P_3(\gG)=\{\{x,y\}\in \mathcal{P}(\gG): |\fR_{\gG}(x,y)|\ge 3 \}.\]
For any subgraph \(\gG \subseteq \fF\), we define the auxiliary function (where \(\delta \in (0,1/6)\) is a fixed constant)
\begin{equation}\label{equ:g-function}
   g(\gG):= |\fF|- 3n^{r-1-\delta}+\left(|\mathcal P_1(\gG)|+2|\mathcal P_2(\gG)|+3|\mathcal P_3(\gG)| \right)\cdot n^{r-3-\delta}. 
\end{equation}

\begin{proof}[\bf Proof of Lemma~\ref{lem:one-root}] 
Consider any $n$-vertex $\mathcal C_4^r$-free $r$-graph $\fF$, where $n$ is sufficiently large compared with $r\geq 4$ and $\delta\in (0,1/6)$. 
Then $g(\fF)\le |\fF|- 3n^{r-1-\delta}+ 3\binom{n}{2} n^{r-3-\delta}<|\fF|$.
Among all spanning subgraphs \(\gG\) of \(\fF\), choose \(\gG\) to be one with the fewest edges subject to the condition 
\[ |\gG| - g(\gG) \ge 0.\]
Consequently, we have
\begin{equation} \label{equ:mathcal-G}
 |\gG|\geq g(\gG)\geq |\fF|- 3n^{r-1-\delta}.
\end{equation}

For the remainder of the proof, we focus on this subgraph $\gG$ and show that $\mathcal P_3(\gG)=\emptyset$ and $|\mathcal P_2(\gG)|=O(n^{1+5\delta})$. 
These bounds allow us to delete at most $O(n^{r-2+5\delta})$ edges from $\gG$ to eliminate all pairs in $\mathcal P_2(\gG)$, yielding the desired subgraph of $\fF$. 
We proceed with a sequence of claims.

\begin{Claim}\label{clm:many-root-edges}
For any subgraph $\gG'\subseteq \gG$, if there exist $0\leq i<j\leq 3$ with $\mathcal{P}_i(\gG')\cap \mathcal{P}_j(\gG)\neq \emptyset$, then $|\gG\setminus \gG'|\ge n^{r-3-\delta}$.
\end{Claim}

\begin{proof}
Since $\gG'\subseteq \gG$, we have $|\fR_{\gG'}(x,y)| \le |\fR_{\gG}(x,y)|$ for every pair $\{x,y\}$. 
Thus, for each \(i \in [3]\),
\begin{equation}\label{equ:P_j-sum}
\bigcup_{j=i}^3 \mathcal P_j(\gG') \subseteq \bigcup_{j=i}^3 \mathcal P_j(\gG).
\end{equation}
Consequently, $|\mathcal P_1(\gG')|+2|\mathcal P_2(\gG')|+3|\mathcal P_3(\gG')|\leq |\mathcal P_1(\gG)|+2|\mathcal P_2(\gG)|+3|\mathcal P_3(\gG)|$ and thus $g(\gG')\leq g(\gG)$.
 
Assume now that $|\gG|-|\gG'|< n^{r-3-\delta}$. 
By the minimality of $\gG$, we have $ |\gG|-g(\gG)\ge 0 >|\gG'| -g(\gG').$ 
It follows that $0\leq g(\gG )-g(\gG')<|\gG|-| \gG'|< n^{r-3-\delta},$ which forces 
\[|\mathcal P_1(\gG')|+2|\mathcal P_2(\gG')|+3|\mathcal P_3(\gG')|= |\mathcal P_1(\gG)|+2|\mathcal P_2(\gG)|+3|\mathcal P_3(\gG)|.\] 
This further shows that \eqref{equ:P_j-sum} is an equality for all $1\leq i\leq 3$, and consequently 
$\mathcal P_j(\gG')=\mathcal P_j(\gG)$ for all $1\leq j\leq 3$,
contradicting the assumption that $\mathcal{P}_\ell(\gG')\cap \mathcal{P}_j(\gG)\neq \emptyset$ for some $0\leq \ell<j\leq 3$.
\end{proof}

\begin{Claim}\label{clm:P3}
We have \(\mathcal P_3(\gG) = \emptyset\). 
Moreover, for each pair \(\{x,y\} \in \mathcal P_2(\gG)\), there exist three vertices such that every set in \(\CN_{\gG}(x,y)\) contains at least two of them, and thus
\( \cd_{\gG}(x,y) \le 3 n^{\,r-3}. \)
\end{Claim}	

\begin{proof}
Suppose $\{x,y\}\in \mathcal{P}_3(\gG)$. Since $\gG$ is $\mathcal C^r_4$-free, $\CN_{\gG}(x,y)\subseteq \binom{V(\gG)}{r-1}$ is an intersecting family with covering number at least 3. 
As $r\geq 4$, by Theorem~\ref{lem:covering-number-3}, $\cd_{\gG}(x,y)\le ((r-1)^2-(r-1)+1+o(1))\binom{n-3}{(r-1)-3}+27< n^{r-3-\delta}$. 
Deleting $\cd_{\gG}(x,y)< n^{r-3-\delta}$ edges produces a subgraph $\gG'\subseteq \gG$ with $\cd_{\gG'}(x,y)=0$. 
Then \(\mathcal{P}_0(\gG') \cap \mathcal{P}_3(\gG) \neq \emptyset\), 
contradicting Claim~\ref{clm:many-root-edges}. 
This proves \(\mathcal P_3(\gG) = \emptyset\).

Consider $\{x,y\}\in \mathcal{P}_2(\gG)$. 
Then $\CN_{\gG}(x,y)\subseteq \binom{V(\gG)}{r-1}$ is an intersecting family with covering number two.
We assert that no element lies in all but $(r-1)^{ r-1 }\binom{n}{(r-1)-3}<n^{r-3-\delta}$ members of $\CN_{\gG}(x,y)$. 
Otherwise, deleting those exceptional edges would produce a subgraph $\gG'\subseteq \gG$ in which $\CN_{\gG'}(x,y)$ is a star, equivalently $|\fR_{\gG'}(x,y)|=1$, or empty. Thus either \(\mathcal{P}_1(\gG') \cap \mathcal{P}_2(\gG) \neq \emptyset\) or \(\mathcal{P}_0(\gG') \cap \mathcal{P}_2(\gG) \neq \emptyset\), both contradicting Claim~\ref{clm:many-root-edges}.
By this assertion and Lemma~\ref{lem:cover2-abc}, there exist three vertices such that each $(r-1)$-set in $\CN_{\gG}(x,y)$ contains at least two of them. 
Hence, $\cd_{\gG}(x,y)\le 3\binom{n}{(r-1)-2}< 3n^{r-3}$.
\end{proof}

For each $(x,y)\in \mathcal P_2(\gG)$, we call the three vertices identified in Claim~\ref{clm:P3} the {\it degenerate roots} of $\{x,y\}$.
By Claim~\ref{clm:P3}, we have $\mathcal{P}(\gG)=\mathcal{P}_1(\gG)\cup \mathcal{P}_2(\gG)$, with the following properties:
\begin{itemize}
	\item Each pair $\{x,y\}\in \mathcal P_1(\gG)$ has a unique chosen root, and every set in $\CN_{\gG}(x,y)$ contains the root.
	\item Each pair $\{x,y\}\in \mathcal P_2(\gG)$ has three degenerate roots, and every set in $\CN_{\gG}(x,y)$ contains at least two of them. 
\end{itemize}
Call a vertex $v$ {\bf heavy} if it is a root or degenerate root of at least $0.1 n^{2-2\delta}$ pairs $\{x,y\}$.
Then, the number of heavy vertices is at most $3\binom{n}{2}/0.1 n^{2-2\delta}<15 n^{2\delta}$.
		
\begin{Claim}\label{clm:two-heavy}
If the degenerate roots of \(\{x,y\} \in \mathcal P_2(\gG)\) are \(a,b\) and \(c\), then at least two are heavy.
\end{Claim}

\begin{proof}
It suffices to show that any pair in $\{a,b,c\}$ has at least one heavy vertex. 
By symmetry, we consider the pair $\{a,b\}$.
If we remove all edges of the form $xabS$ with $S\in \CN_{\gG}(xab,yab)$, the remaining common neighborhood of $x,y$ is a star centered at $c$. By Claim~\ref{clm:many-root-edges}, $\cd_{\gG}(xab,yab )\ge n^{r-3-\delta}$. 
			
In the case of $r=4$, $\CN_{\gG}(xab,yab)$ is a set of vertices. 
For any $u,v\in \CN_{\gG}(xab,yab)$, we have $xab,yab\in \CN_{\gG}(u,v)$. 
If $\{u,v\}\in \mathcal{P}_1(\gG)$, then the unique root of $\{u,v\}$ is either $a$ or $b$; 
if $\{u,v\}\in \mathcal{P}_2(\gG)$, the set $xab$ contains two of the three degenerate roots of $\{u,v\}$. 
In both cases, one of $a$ and $b$ must be either the root or a degenerate root of $\{u,v\}$. 
There are at least $\binom{\cd_{\gG}(xab,yab )}{2}\ge \binom{n^{1-\delta}}{2}\ge 0.2 n^{2-2\delta}$ such pairs $\{u,v\}$. 
Hence, at least one vertex in $\{a,b\}$ is heavy. 
			
It remains to treat the case $r\ge 5$. 
Then $\CN_{\gG}(xab,yab)$ is an $(r-3)$-graph with at least $n^{(r-3)-\delta}$ edges. 
Lemma~\ref{lem:counting-C4-pairs} gives at least $0.2n^{2-2\delta}$ pairs $\{u,v\}$ such that for each pair, there exist two disjoint $(r-4)$-sets $S,T $ with $uS,uT,vS,vT\in \CN_{\gG}(xab,yab)$. 
Then $xabS, yabT \in \CN_{\gG}(u,v)$. 
Since $S\cap T=\emptyset$, these two common-neighbor sets intersect only in $\{a,b\}$. Hence, since each edge contains either the unique root or two degenerate roots of $\{u,v\}$, 
we deduce that one of $a, b$ must be the root or a degenerate root of $\{u,v\}$. 
Hence, with at least $0.2n^{2-2\delta}$ such pairs $\{u,v\}$, at least one vertex in $\{a,b\}$ is heavy, 
proving the claim. 
\end{proof}

\begin{Claim}\label{clm:No-P2}
We have $|\mathcal P_2(\gG)|< 100n^{1+5\delta}.$
\end{Claim}
\begin{proof}
Suppose, for a contradiction, that $|\mathcal P_2(\gG)|\ge 100n^{1+5\delta} $. 
By averaging, there is a vertex $x$ appearing in at least $m=200n^{5\delta}$ pairs in $\mathcal{P}_2(\gG)$. 
Denote these pairs by $\{x,y_i\}\in \mathcal P_2(\gG)$ for $i\in [m]$.
Recall that in total there are at most $15 n^{2\delta}$ heavy vertices, and
by Claim~\ref{clm:two-heavy}, each $\{x,y_i\}\in \mathcal P_2(\gG)$ has at least two heavy degenerate roots. 
By averaging, there exist at least $m':= m/ \binom{15 n^{2\delta}}{2}> n^{\delta}$ pairs $\{x,y'_i\}\in \mathcal{P}_2(\gG)$ with $i\in [m']$, sharing the same two heavy degenerate roots $a$ and $b$ and another degenerat root $c_i$. 
After deleting all edges containing $xab$ in $\gG$, which is at most $n^{r-3}$ edges, we get a new graph $\gG'$. 
Then for each $i\in [m']$, $\CN_{\gG'}(x,y'_i)$ is a star (possibly empty) whose edge contains $c_i$, and hence $\{x,y'_i\}\notin \mathcal P_2(\gG')$.
Thus, we have $|\mathcal P_2(\gG)|-|\mathcal P_2(\gG')|\ge m'> n^{\delta}$. 
By \eqref{equ:P_j-sum} and the fact $\mathcal{P}_3(\gG')=\mathcal{P}_3(\gG)=\emptyset$, we obtain 
\[(|\mathcal P_1(\gG)|+2|\mathcal P_2(\gG)|)-(|\mathcal P_1(\gG')|+2|\mathcal P_2(\gG')|)\ge |\mathcal P_2(\gG)|-|\mathcal P_2(\gG')| > n^{\delta},\] and
\[g(\gG)-g(\gG')> n^{\delta}\cdot n^{r-3-\delta}=n^{r-3} \ge |\gG|-| \gG' |.\] 
It follows that $|\gG'|-g(\gG') \ge |\gG| -g(\gG)\ge 0,$ 
contradicting the minimality of $\gG$. 
\end{proof}
		
We are now ready to complete the proof of Lemma~\ref{lem:one-root}.
Let $\widetilde{\fF}$ denote the subgraph obtained from $\gG$ by deleting all edges of the form $xS$ for $S\in \CN_{\gG}(x,y)$ for each pair $\{x,y\}\in \mathcal P_2(\gG)$.

We show that $\widetilde{\fF}$ is the desired subgraph of $\fF$. 
By Claim~\ref{clm:P3}, $\cd_{\gG}(x,y)\le 3n^{r-3}$ for each $\{x,y\}\in \mathcal P_2(\gG)$.
By Claim~\ref{clm:No-P2}, $|\mathcal P_2(\gG)|< 100n^{1+5\delta}$.
Thus $|\gG\setminus \widetilde{\fF}|\le |\mathcal P_2(\gG)|\cdot 3n^{r-3}\le 300n^{r-2+5\delta}$. 
Hence, we have (since $\delta\in (0,1/6)$ and $n$ is sufficiently large)
\[|\widetilde{\fF}|\ge |\gG|-300n^{r-2+5\delta}\overset{\eqref{equ:mathcal-G}}{\ge} |\fF|- 3n^{r-1-\delta}-300n^{r-2+5\delta}\geq |\fF|- 4n^{r-1-\delta}.\] 		
Moreover, $\mathcal P_2(\widetilde{\fF})=\emptyset$.
Since $\mathcal P_3(\widetilde{\fF})\subseteq \mathcal P_3(\gG )$, 
Claim~\ref{clm:P3} shows $\mathcal P_3(\widetilde{\fF})=\emptyset$. 
Therefore, for any vertices $x$ and $y$, $\CN_{\widetilde{\fF}}(x,y)$ is a star.
This proves Lemma~\ref{lem:one-root}.
\end{proof}

\section{Proof of Lemma~\ref{lem:CN-stability}} \label{sec:stability-com}
The goal of this section is to prove Lemma~\ref{lem:CN-stability}, assuming an abstract deletion lemma, Lemma~\ref{lem:3-graph-deletion}, whose proof is postponed to Section~\ref{sec:remove-diag}.
Throughout this section, let $r\geq 4$, let $\epsilon\in (0,1/420r)$, and let $n\gg r,\epsilon$ be sufficiently large. We fix an $n$-vertex $\mathcal C_4^r$-free $r$-graph $\gG$ such that $|\gG|\ge (1-\epsilon)\binom{n-1}{r-1}$ and
\begin{equation}\label{prop:root-function}
\text{for any vertices \(x\) and \(y\), \(\CN_{\gG}(x,y)\) is an \((r-1)\)-star with center \(\fr(x,y)\).}\footnote{If \(\CN_{\gG}(x,y)=\emptyset\), we denote \(\fr(x,y)\) by an arbitrary vertex.}
\end{equation}

Thus each pair $(x,y)$ has a local center $\fr(x,y)$, and the task is
to show that these local centers agree on almost all pairs from a large
vertex set. The proof proceeds through four steps. First, we encode
every $r$-edge of $\gG$ as a triple $(A,\{x,y\})$ with $|A|=r-2$,
thereby passing to a bipartite $3$-graph $\hH$. Second, we decompose
the part $\aA=\binom{[n]}{r-2}$ into subsets $\aA_i$, each of size at most $n$.
Third, we apply Lemma~\ref{lem:3-graph-deletion} to the subgraph of
$\hH$ induced by each $\aA_i$, which ensures that the common
neighborhood of most vertex pairs is a compatible subfamily in
$\aA_i$. Finally, we use an edge-coloring argument to extract a large
set of vertices such that the local center of each pair lies in the
common intersection of a small compatible subfamily; see
Lemma~\ref{lem:cluster-S-set-W}. This eventually leads to a single
common vertex as the local center for all pairs in a large vertex
subset, which is precisely what is needed for Lemma~\ref{lem:CN-stability}.

\subsection{Good bipartite 3-graphs and the deletion lemma}\label{subsec:bipartite-3-graphs} 
In this subsection, we first introduce an abstract framework -- the setting of good bipartite $3$-graphs from Definition~\ref{Prop:3graph-relation} -- which, at first sight, appears unrelated to our study. We will explain later (see Lemma~\ref{clm:H-properties}) why this framework captures the essential information about $\gG$ needed in the upcoming proofs.
We then state a deletion lemma (Lemma~\ref{lem:3-graph-deletion}) in the language of this framework, which will serve as a key ingredient in the proof of Lemma~\ref{lem:CN-stability}.

A 3-graph $\fF$ on the vertex set $X \cup Y$ is called a {\bf bipartite 3-graph}, denoted by $\fF(X,Y)$, if every edge $e \in \fF$ satisfies $|e\cap X|=1$ and $|e\cap Y|=2$, i.e., $e\in X\times \binom{Y}{2}$. 
We consider the set $X$ equipped with a symmetric binary relation $\leftrightarrow$. 
Under this relation, any two elements $x, x' \in X$ are either \textit{compatible}, denoted by $x \leftrightarrow x'$, or \textit{incompatible}, denoted by $x \nleftrightarrow x'$.
A subset $X' \subseteq X$ is \textit{compatible} if all its elements are pairwise compatible.

Before presenting the framework in full detail,
we briefly explain how bipartite 3-graphs come into play in our proofs.
We transform the $r$-graph $\gG$ into a bipartite 3-graph $\hH=\hH(X,Y)$, where $X=\binom{V(\gG)}{r-2}$ and $Y=V(\gG)$; see Definition~\ref{dfn:bipartite-3-graph-H}.
Each edge of $\hH$ is a triple in $X\times \binom{Y}{2}=\binom{V(\gG)}{r-2}\times \binom{V(\gG)}{2}$, and corresponds to a unique edge of $\gG$.
The compatibility relation on $X$ is defined as follows:
two $(r-2)$-sets $A, A'\in X$ are compatible (written $A\leftrightarrow A'$) if and only if $A\cap A'\neq \emptyset$.
Accordingly, a subset $X'\subseteq X$ is called \emph{compatible} if its members are pairwise compatible; equivalently, viewing $X'$ as a family of $(r-2)$-sets, $X'$ is an intersecting family.
We also call such $X'$ a \emph{compatible $(\leftrightarrow)$ family}.
Before proceeding, we highlight an important fact, which motivates much of the notation in the framework developed below.
Since $\gG$ is $\mathcal C_4^r$-free,
\begin{equation}\label{equ:A-fact}
   \text{for any $A, A'\in X$ with $A\nleftrightarrow A'$, the graph $\CN_{\hH}(A,A')$ is either a star or a triangle in $Y$.}\footnote{When we view such $A,A'$ as two $(r-2)$-subsets of $V(\gG)$, the graph $\CN_{\gG}(A,A')$ coincides with $\CN_{\hH}(A,A')$, and we will use the two interchangeably.} 
\end{equation}

We now return to the setting of a bipartite 3-graph $\fF(X,Y)$.
For two vertices $x,x'\in X$, the situation is particularly favorable when the common neighborhood $\CN_{\fF}(x,x')$ is a star; we therefore keep track of the pairs for which this fails.
In view of this and \eqref{equ:A-fact}, we introduce the following two sets:
\[
\alpha(\fF)=\left\{
\{x,x'\}\in \binom{X}{2}: x\leftrightarrow x' \text{ and } \CN_{\fF}(x,x') \text{ is not a star}
\right\},
\]
and
\[
\beta(\fF)=\left\{
\{x,x'\}\in \binom{X}{2}: x\nleftrightarrow x' \text{ and } \CN_{\fF}(x,x') \text{ is a triangle}
\right\}.
\]
For $\{x,x'\}\in \beta(\fF)$, we call the triangle $\CN_{\fF}(x,x')$ a \emph{$\beta(\fF)$-triangle}.  

With this preparation, we are now ready to formally define the framework.
\begin{dfn}[Good bipartite $3$-graphs]\label{Prop:3graph-relation}
A bipartite $3$-graph $\fF=\fF(X,Y)$ equipped with a compatibility relation $\leftrightarrow$ on $X$ is called \textit{good} if each of the following properties holds.
 \begin{enumerate}
			\item[(1).] For $x\nleftrightarrow x' $ in $X$, $\CN_{\fF}(x,x')$ is either a star (possibly empty), or a triangle. 
 \item[(2).] All $\beta(\fF)$-triangles are pairwise edge disjoint.
			\item[(3).] For each pair of $u,v\in Y$ with $\CN_{\fF}(u,v)\neq \emptyset$, there exists a vertex $w\in Y$ such that if $y\in Y\setminus \{u,v,w\}$, then $\CN_{\fF}(uy,vy)$ is a compatible subset of $X$. 		
	\end{enumerate}
\end{dfn}

Note that when $\fF$ is good, $\alpha(\fF)$ and $\beta(\fF)$ record precisely the compatible and the incompatible non-star pairs, respectively.
We will also use repeatedly the simple observation that every subgraph of a good bipartite $3$-graph, equipped with the same compatibility relation, remains good. Indeed, passing to a subgraph only deletes edges from the relevant common neighborhoods, so the three conditions above are preserved.

\medskip

The following lemma plays a central role in the proof of Lemma~\ref{lem:CN-stability}. Its proof is given in Section~\ref{sec:remove-diag}, and it is the only result from that section needed here.
For a $3$-graph $\fF$ and two vertices $a,b\in V(\fF)$, we write $N_\fF(ab)$ for the link $N_\fF(\{a,b\})$ of the pair $\{a,b\}$ in $\fF$.
Roughly speaking, the lemma says that for every good bipartite $3$-graph $\fF(X,Y)$, one can delete a few edges so that in the remaining $3$-graph, the link of every pair $\{y,y'\}$ of $Y$ is a compatible subset of $X$, unless $\{y,y'\}$ is itself an edge of a $\beta(\fF)$-triangle.
 
\begin{lem}\label{lem:3-graph-deletion}
		Let $\fF= \fF(X,Y)$ be a good bipartite 3-graph equipped with a compatibility relation $\leftrightarrow$ on $X$.
		If each vertex in $X$ is compatible with at most $O_r(1)$ vertices, then there exists an edge set $\rR\subseteq \fF$ of size
			\begin{equation}\label{equ:rR}
			|\rR|\le \binom{|X|}{2}+\binom{|Y|}{2}+ O_r \left((| X|+|Y|)(|\alpha(\fF)|+1) \right)
			\end{equation}           
		such that for any $y,y' \in Y$, exactly one of the following alternatives holds:
 \begin{itemize}
 \item $N_{\fF\setminus \rR}(y y')$ is a compatible subset of $X$;
 \item $N_{\fF\setminus \rR}(y y')=\{x,x'\}$ for two incompatible ($\nleftrightarrow$) vertices $x,x'$, where $\{x,x'\}\in \beta(\fF)$.
 \end{itemize}
 In the second alternative, $yy'$ is necessarily an edge of the corresponding $\beta(\fF)$-triangle $\CN_{\fF}(x,x')$.
 Moreover, when $abc$ is a $\beta(\fF)$-triangle, one of the following holds:
		\begin{itemize}
			\item At least one of $N_{\fF \setminus \rR}(ab),$ $N_{\fF \setminus \rR}(bc)$ and $N_{\fF \setminus \rR}(ca)$ has size at most one, or
			\item At least two of $N_{\fF \setminus \rR}(ab)$, $N_{\fF \setminus \rR}(bc)$ and $N_{\fF \setminus \rR}(ca)$ are compatible subsets of $X$. 
		\end{itemize}
	\end{lem}

\subsection{The projection from $\gG$ to $\hH(\aA,[n])$}\label{subsec:transform}
 
We now verify that the abstract deletion lemma applies to the bipartite $3$-graph naturally associated with $\gG$.
This projection is the bridge between local centers in the original $r$-graph and compatible families in the $X$-part of a bipartite $3$-graph.

\begin{dfn}[Projection and compatibility]\label{dfn:bipartite-3-graph-H}	
	Denote $\aA=\binom{[n]}{r-2}$.
	Let $\hH=\hH(\aA,[n])$ be the bipartite $3$-graph whose edges are the triples $Auv$, with $A\in\aA$ and $u,v\in[n]$, such that
\[
        A\cap\{u,v\}=\emptyset
        \qquad\text{and}\qquad
        A\cup\{u,v\}\in\gG .
\]
Here $A$ is regarded as a vertex in the $\aA$-part of $\hH$, while $u$ and $v$ are vertices in the $[n]$-part.
	For any two sets $A,A'\in \aA$, if $A\cap A'\neq\emptyset$, we say they are compatible, denoted by $A\leftrightarrow A'$; otherwise, they are incompatible, denoted by $A\nleftrightarrow A'$. 
\end{dfn}
Since every edge of $\gG$ contains $\binom{r}{2}$ pairs $\{u,v\}$, we have
\begin{equation}\label{equ:HG}
    |\hH|=\binom{r}{2}|\gG|.
\end{equation}

For the remainder of this section, we fix the bipartite 3-graph $\hH=\hH(\aA,[n])$ and the compatible/incompatible (equivalent to intersecting/non-intersecting respectively) relation on $\aA$.

\begin{lem}\label{clm:H-properties}
The bipartite 3-graph $\hH=\hH(\aA,[n])$ is good with respect to the fixed compatible/incompatible relation on $\aA$, meaning that it satisfies Definition~\ref{Prop:3graph-relation}.
Moreover, if $abc$ is a $\beta(\hH)$-triangle, then $\fr(a,b)=c$, $\fr(b,c)=a$ and $\fr(c,a)=b$.
\end{lem}
	\begin{proof}
		Suppose $A\nleftrightarrow A' $ in $ \aA$ with non-empty $\CN_{\hH}(A, A')$. If $\CN_{\hH}(A, A')$ contains two disjoint sets, this would imply the existence of a $\mathcal C_4^r$ in the original $r$-graph $\gG$, which is a contradiction. Hence, the $2$-graph $\CN_{\hH}(A, A')$ is intersecting, which is either a star or a triangle. This confirms Definition~\ref{Prop:3graph-relation} (1).

 Suppose $abc$ is a $\beta(\hH)$-triangle with $\CN_{\hH}(A, A')=\{ab,bc,ca\}$ for some $A\nleftrightarrow A' $ in $ \aA$. Then in the $r$-graph $\gG$, we have $A\cup \{c\}, A'\cup \{c\} \in \CN_{\gG}(a, b)$. By \eqref{prop:root-function}, $\CN_{\gG}(a, b)$ is a star centered at $\fr(a,b)$, so the disjointness of $A$ and $A'$ immediately gives $\fr(a,b)=c$. Similarly, $\fr(b,c)=a$ and $\fr(c,a)=b$. 
	 This also proves the edge-disjointness assertion. Indeed, if two $\beta(\hH)$-triangles share the $Y$-edge $ab$, then the third vertex of each triangle must be the same vertex $\fr(a,b)$. Thus the two triangles have the same three $Y$-edges. Identifying identical triangles, distinct $\beta(\hH)$-triangles are therefore pairwise edge-disjoint. This establishes Definition~\ref{Prop:3graph-relation} (2).
 
		Assume $\CN_{\hH}(u,v)\neq \emptyset$ for $u,v\in [n]$. Set $w=\fr(u,v)$, and choose any $y\in [n]\setminus \{u,v,w\}$.
		By \eqref{prop:root-function}, every set in $\CN_{\gG}(u,v)$ contains $w$. Hence every set in $\CN_{\gG}(uy,vy)$ contains $w$.
 Therefore, $\CN_{\gG}(uy,vy)= \CN_{\hH}(uy,vy)$ is an intersecting family, and hence Definition~\ref{Prop:3graph-relation} (3) holds.
	\end{proof}

\subsection{Balanced decomposition of $\aA$}\label{subsec:balanced-decomposition}
The deletion lemma requires bounded compatibility degrees within the part $X$.  
This condition fails for $X=\aA=\binom{[n]}{r-2}$.  
By Baranyai's theorem, however, $\aA$ can be partitioned into subfamilies, each a union of roughly $r-2$ matchings, within which the compatibility degree is bounded.

For a subfamily $\mathcal B\subseteq \aA$, let $\hH[\mathcal B,[n]]$ denote the subgraph of $\hH$ consisting of all edges $Auv\in \hH$ with $A\in \mathcal B$. Thus, once $\aA$ is partitioned, these subgraphs are edge-disjoint and together cover $\hH$.

Next, we will prove the following lemma in this subsection.

\begin{lem}\label{lem:partition-H}
There exists a partition
 $\aA=\bigsqcup_{i\in[m]}\aA_i$ with $m=(1+o(1))\binom{n}{r-2}/n,$
such that each $\aA_i$ is a union of at most $r-2$ matchings. Moreover, for each $i\in[m]$ there is a subgraph $\hH_i\subseteq \hH[\aA_i,[n]]$ such that
\[
        \sum_{i\in[m]}|\hH_i|\ge |\hH|-O_r(n^{r-2})
\]
and
\[
        \sum_{i\in[m]}|\alpha(\hH_i)|=O_r(m)=O_r(n^{r-3}).
\]
\end{lem}

 	We shall use Baranyai's theorem to decompose the set $\aA=\binom{[n]}{r-2}$.
 	\begin{thm}[Baranyai's Theorem \cite{Ba75}]\label{thm:Baranyai}
		For any positive integers $k$ and $n$ such that $k$ divides $n$, the complete $k$-graph can be decomposed into $\binom{n}{k}\frac{k}{n} = \binom{n-1}{k-1}$ perfect matchings.
	\end{thm}

	\begin{thm}[Frankl, Theorem 5 in \cite{Fr20}]\label{lem:Frankl}
		Let $n\ge 72k$ and let $\fF\subseteq \binom{[n]}{k}$ be an intersecting family. Then there is an element contained in all but $\binom{n-3}{k-2}$ members of $\fF$.
	\end{thm}

Let us outline the proof. First, Baranyai's theorem decomposes $\aA$ into matchings, which we group into blocks $\aA_i$ of size about $n$. We then apply a random permutation $\phi$ to the ground set and consider the subgraph $\fF_i=\hH[\phi(\aA_i),[n]]$. For each compatible pair in  $\phi(\aA_i)$, we either delete the few edges preventing its common neighborhood from being a star, or record the pair in a small exceptional set $\alpha_i$. The probabilistic method gives a choice of $\phi$ for which the total number of deleted edges and the total size of all exceptional sets are small. The resulting cleaned subgraphs $\hH_i\subseteq \fF_i$ then satisfy $\alpha(\hH_i)\subseteq \alpha_i$.

	\begin{proof}[Proof of Lemma~\ref{lem:partition-H}]
		Take $n'=(r-2) \lceil \frac{n}{r-2}\rceil$. By Theorem~\ref{thm:Baranyai}, decompose $\binom{[n']}{r-2}$ into $\binom{n'}{r-2}\frac{r-2}{n'}$ perfect matchings. Restricting these matchings to $[n]$ partitions $\aA = \binom{[n]}{r-2}$ into matchings.
		Grouping these matchings into blocks of size at most $r-2$, we obtain a partition $\aA=\bigsqcup_{i\in [m]} \aA_i$, where $m= \lceil\binom{n'}{r-2}\frac{1}{n'}\rceil=(1 + o(1))\binom{n}{r-2}/n$ and each $\aA_i$ is a union of at most $r-2$ matchings. 
	Next, choose a uniformly random permutation $\phi: [n] \to [n]$. For each $i$, let $\fF_i=\hH[\phi(\aA_i),[n]]$. The graphs $\fF_i$ are edge-disjoint and together cover $\hH$.

		Recall that the root function $\fr$ is already defined on pairs of vertices in \eqref{prop:root-function}. For $k=1$, this is the root function we use: if $S=\{x\}$ and $S'=\{y\}$, then $\fr(S,S'):=\fr(x,y)$.
For each $2\le k\le r-3$ and each pair of disjoint $k$-sets $S,S'\subseteq [n]$ with $\CN_{\gG}(S,S')\neq \emptyset$, define
			\[\text{$\fr(S,S')$ to be a vertex achieving the maximum degree in } \CN_{\gG}(S,S').\]
For $1\le k\le r-3$, denote by 
			\[\Gamma_{\gG}(S,S')=\{T\in \CN_{\gG}(S,S') : \fr(S,S')\notin T \}\]
 the collection of non-star edges.
		Thus, if we delete all edges $\{S\cup T: T\in \Gamma_{\gG}(S,S')\}$, the remaining common neighbors of $(S,S')$ form a star centered at $\fr(S,S')$. For $k=1$, the assumption \eqref{prop:root-function} gives $\Gamma_{\gG}(S,S')=\emptyset$. For $2\le k\le r-3$, since $\gG$ is $\mathcal C_4^r$-free, $\CN_{\gG}(S,S')$ is an intersecting family, and Theorem~\ref{lem:Frankl} gives
		\begin{equation}\label{equ:gamma-SSprime}
			|\Gamma_{\gG}(S,S')| \leq n^{r-k-2} \text{ for disjoint } k\text{-sets } S, S' \text{ with } 1\le k\le r-3.
		\end{equation}
	Next, consider each $\fF_i$ for $i\in [m]$. For any intersecting/compatible sets $A\leftrightarrow A' $ in $ \phi(\aA_i)$, if $\CN_{\fF_i}(A, A')= \emptyset$, define $\Gamma_{i}(A, A') =\emptyset$. 
 Otherwise, $\fr(A \setminus A', A' \setminus A)$ and $\Gamma_{\gG}(A \setminus A', A' \setminus A)$ exist. Define 
		\[ \Gamma_{i}(A, A') = \{ T \in \CN_{\fF_i}(A, A') : \fr(A \setminus A', A' \setminus A) \notin T\cup (A \cap A') \}. \]
 Note that for each $T\in \Gamma_{i}(A, A')$, the set $T\cup (A \cap A')\in \Gamma_{\gG}(A \setminus A', A' \setminus A)$.
	If we delete all edges $\{A\cup T$ : $T\in \Gamma_{i}(A, A')\}$, the remaining common neighbors of $A, A'$ form a star centered at $\fr(A \setminus A', A' \setminus A)$. 
 When $\fr(A \setminus A', A' \setminus A) \in A \cap A'$, we have $\Gamma_{i}(A, A') =\emptyset$. 
	Then we include these pairs to $\alpha_i$: 
		\[ \alpha_i =\left \{ 	\{A, A'\} \in \binom{\phi(\aA_i)}{2}: A \leftrightarrow A', \CN_{\fF_i}(A, A')\neq \emptyset \text{ and } \fr(A \setminus A', A' \setminus A) \in A \cap A' 	\right\}. \]

Fix a compatible pair $A, A' \in \aA_i$. Let $k = |A \setminus A'| = |A' \setminus A|$ and $j = |A \cap A'| = r-2-k$. Since $A\leftrightarrow A'$, we have $1\le k\le r-3$.
We compute the expectation conditioned on the event $E_{S,S'}$ that $\phi(A \setminus A') = S$ and $\phi(A' \setminus A) = S'$.
Let $V' = [n] \setminus (S \cup S')$, so $|V'| = n-2k$. The set $J = \phi(A \cap A')$ is a uniformly random $j$-subset of $V'$.
An edge $T \in \Gamma_{i}(\phi(A), \phi(A'))$ is a $2$-set in $V'$ such that $R = J \cup T$ is a set in $\Gamma_{\gG}(S,S')$. Note that $|R| = j+2 = r-k$.
By linearity of expectation, we sum over all $T \in \binom{V'}{2}$.
\begin{align*}
\mathbb{E}\left[ \left| \Gamma_{i}(\phi(A),\phi(A')) \right| \mid E_{S,S'} \right]
&= \sum_{T \in \binom{V'}{2}} \mathbb{P}\left( J \cup T \in \Gamma_{\gG}(S,S') \right)
\end{align*}
For a fixed $T$, $J$ must be chosen from the $j$-subsets of $V' \setminus T$. The total number of such $J$'s is $\binom{n-2k-2}{j}$. 
Therefore, the expectation is:
\begin{align*}
\mathbb{E}\left[ \left| \Gamma_{i}(\phi(A),\phi(A')) \right| \mid E_{S,S'} \right]
&= \sum_{T \in \binom{V'}{2}} \frac{|\{ R \in \Gamma_{\gG}(S,S') : T \subseteq R \}|}{\binom{n-2k-2}{j}} \\
&= \frac{1}{\binom{n-2k-2}{j}} \sum_{T \in \binom{V'}{2}} \sum_{R \in \Gamma_{\gG}(S,S')} \mathbb{I}(T \subseteq R) \\
&= \frac{1}{\binom{n-2k-2}{j}} \sum_{R \in \Gamma_{\gG}(S,S')} \sum_{T \in \binom{R}{2}} 1 = \frac{1}{\binom{n-2k-2}{j}} \sum_{R \in \Gamma_{\gG}(S,S')} \binom{|R|}{2} \\
&= \frac{1}{\binom{n-2k-2}{r-k-2}} \sum_{R \in \Gamma_{\gG}(S,S')} \binom{r-k}{2} \qquad (\text{since } |R|=r-k, j=r-k-2) \\
&= \frac{\binom{r-k}{2} \cdot |\Gamma_{\gG}(S,S')|}{\binom{n-2k-2}{r-k-2}} 
\overset{\eqref{equ:gamma-SSprime}}{\leq} \frac{O_r(1) \cdot n^{r-k-2}}{\Theta(n^{r-k-2})} = O_r(1).
\end{align*}
Since this bound is independent of $S, S'$, by taking the expectation over all $S, S'$, we have
\[\mathbb{E}\left[ \left| \Gamma_{i}(\phi(A),\phi(A')) \right| \right] = O_r(1).\]

 For a fixed $i\in [m]$, $\aA_i$ is a union of at most $r-2$ matchings, so $d_{\aA_i}(v)\le r-2$ for each $v\in [n]$. Therefore, the number of compatible pairs in $\aA_i$ is at most $\sum_{v \in [n]} \binom{d_{\aA_i}(v)}{2} \le n \binom{r-2}{2} = O_r(n)$.
By linearity of expectation,		
			\[\mathbb{E} \left[
			\sum_{i\in [m]}		\sum_{ A\leftrightarrow A'\in \aA_i}	\left|\Gamma_{i}(\phi(A),\phi(A')) \right| \right ]
			=
			\sum_{i\in [m]}\sum_{ A\leftrightarrow A'\in \aA_i} 
			\mathbb{E}\left[	\left|\Gamma_{i}(\phi(A),\phi(A'))\right|	\right]
			= O_r(mn)=O_r(n^{r-2}). \]

Next, we estimate the expected size of $\sum_{i \in [m]} |\alpha_i|$ by a similar approach as above. 
Fix a compatible pair $A, A' \in \aA_i$. Let $S, S'$ be disjoint sets of size $k = |A \setminus A'|$. Again $1\le k\le r-3$.
\begin{align*}
\mathbb{P}\left( \{\phi(A), \phi(A') \}\in \alpha_i \mid \phi(A \setminus A')=S, \phi(A' \setminus A)=S' \right)
&= \mathbb{P}\left( \fr(S, S') \in \phi(A\cap A') \right) = O_r(1/n).
\end{align*}
By taking the expectation over $S, S'$, we get
\[\mathbb{P}\left( \{\phi(A), \phi(A')\}\in \alpha_i \right) = O_r(1/n).\]
Hence, by linearity of expectation,
\begin{align*}
\mathbb{E}\left[ \sum_{i \in [m]} \left|\alpha_i \right| \right]
&= \sum_{i\in [m]} \sum_{ A\leftrightarrow A'\in \aA_i} \mathbb{P}\left(\{\phi(A), \phi(A')\}\in \alpha_i \right) 
\le \sum_{i \in [m]} O_r(n) \cdot O_r(1/n) = O_r(m).
\end{align*}

Let $X = \sum_{i\in [m]} \sum_{ A\leftrightarrow A'\in \aA_i} |\Gamma_{i}(\phi(A),\phi(A'))|$ and $Y = \sum_{i \in [m]} |\alpha_i|$. We have shown $\mathbb{E}[X] = O_r(n^{r-2})$ and $\mathbb{E}[Y] = O_r(m)$.
By the probabilistic method (i.e., using Markov's inequality and a union bound), there must exist at least one permutation $\phi$ for which $X = O_r(n^{r-2})$ and $Y = O_r(m)$.
We fix such a permutation $\phi$. For each $i$, in the 3-graph $\fF_i=\hH[\phi(\aA_i),[n]]$, delete all edges $B\cup T$ with $T\in \Gamma_i(B,B')$, over all compatible pairs $B,B'\in \phi(\aA_i)$.
Let $\hH_i\subseteq \fF_i$ denote the resulting subgraph.
For every compatible pair $B, B' \in \phi(\aA_i)$ with $\{B,B'\} \notin \alpha_i $, $\CN_{\hH_i}( B,B' )$ is a star by construction.
In other words, if $\CN_{\hH_i}(B,B')$ is not a star for two compatible sets $B, B' \in \phi(\aA_i)$, then $\{B,B'\} \in \alpha_i$.
It follows that $\alpha (\hH_i) \subseteq \alpha_i$.
Finally, relabel the blocks $\phi(\aA_i)$ as $\aA_i$; then $\fF_i=\hH[\aA_i,[n]]$ and $\hH_i\subseteq \fF_i$, as in the statement of the lemma.
The total number of edges in the subgraphs $\hH_i$ is
\[\sum_{i \in [m]} |\hH_i|
\ge
|\hH| - X
= |\hH| - O_r(n^{r-2} ).\]
And the total size of the $\alpha$-sets is
\[\sum_{i \in [m]} \left| \alpha(\hH_i) \right|
\le
\sum_{i\in [m]} \left| \alpha_i \right|
= Y = O_r(m)=O_r(n^{r-3} ).\]
This completes the proof of Lemma~\ref{lem:partition-H}.
\end{proof}

\subsection{Applying the deletion lemma}\label{subsec:stability-com-3}
In this subsection, we apply the deletion lemma (Lemma~\ref{lem:3-graph-deletion}) to the $3$-graphs $\hH_i$ obtained from Lemma~\ref{lem:partition-H}, and establish an intermediate lemma (Lemma~\ref{lem:cluster-S-set-W}) concerning the local center structure, which serves as a stepping stone toward Lemma~\ref{lem:CN-stability}.
In fact, the asymptotic bound $|\gG|\le (1+o(1))\binom{n}{r-1}$ follows quite straightforwardly from inequality \eqref{equ:rR} of Lemma~\ref{lem:3-graph-deletion} alone, which we present in the next paragraph.

We show that the deletion lemma alone is strong enough for deriving the asymptotic upper bound of $\gG$.  
Let $\hH_i$ be the $3$-graphs obtained from Lemma~\ref{lem:partition-H}. 
Applying Lemma~\ref{lem:3-graph-deletion} to $\hH_i$'s gives sets $\rR_i$ such that $|\rR_i| \le \binom{|\mathcal{A}_i|}{2}+\binom{n}{2}+O_r((|\mathcal{A}_i|+n)(\alpha(\mathcal{H}_i)+1))=2\binom{n}{2}+O_r(2n(\alpha(\mathcal{H}_i)+1))$, thus
$\sum_{i\in[m]} |\rR_i|\le 2m\binom{n}{2}+O_r(n^{r-2}).$
Moreover, in each $\hH_i\backslash \rR_i$, every pair of vertices of $[n]$ has at most $r-2$ neighbors, which implies that $|\hH_i\backslash \rR_i|\le (r-2)\binom{n}{2}$.  Since $m=(1+o(1))\binom n{r-2}/n$, summing over $i\in[m]$ gives
\[
        \binom r2 |\gG|\overset{\eqref{equ:HG}}{=} |\hH|
        \le m(r-2)\binom{n}{2}+2m\binom{n}{2}+O_r(n^{r-2})
        =(1+o(1))\binom r2 \binom{n}{r-1},
\]
and hence
\begin{equation}\label{equ:upbd-G}
        |\gG|\le (1+o(1))\binom{n}{r-1}.
\end{equation}

However, for both stability and the exact result, the above number of edges alone is not enough. We shall exploit the finer structural information about the residual subgraphs $\hH_i\setminus\rR_i$ supplied by Lemma~\ref{lem:3-graph-deletion}.
In a dense residual subgraph, almost every $Y$-pair has the maximum possible degree $r-2$.  Lemma~\ref{lem:3-graph-deletion} implies that every $Y$-pair either has a compatible neighborhood or forms an edge of a $\beta(\hH_i)$-triangle.
For a maximum-degree $Y$-pair, the latter case can occur only when $r=4$.  
We color maximum-degree $Y$-pairs by compatible families that are contained in their neighborhood.  
The local-star property in the original hypergraph then forbids certain multicolored configurations, forcing a large monochromatic $K_{2,t}$. This monochromatic configuration is the source of a fixed compatible subfamily required in Lemma~\ref{lem:cluster-S-set-W}.

We shall use the following shorthand.  For any nonempty family $\mathcal S$ of sets, write
\[
        \bigcap \mathcal S := \bigcap_{T\in\mathcal S} T .
\]
For an integer $s$, an $s$-compatible family means a compatible family of size $s$.

\begin{lem}\label{lem:cluster-S-set-W}
There exists an index set $I\subseteq [m]$ with $|I|\ge m/4$ such that, for each $i\in I$, there exist an $(r-2)$-compatible subfamily $\mathcal S\subseteq \aA_i$ and a vertex set $W \subseteq [n]$ with $|W|\ge (1-210 \epsilon r)n$ such that
$\fr(x,y)\in \bigcap \mathcal S$ for all distinct $x,y\in W$.
\end{lem}

\begin{proof}[Proof of Lemma~\ref{lem:cluster-S-set-W}]
We separate the proof into a few natural stages.  First we find many indices for which the residual subgraph $\hH_i\setminus\rR_i$ is dense.  For a fixed such index, we color the maximum-degree pairs, record the forbidden configurations forced by the local-star property, and then extract a large monochromatic $K_{2,t}$.  The cases $r\ge5$ and $r=4$ are treated at the end, the latter requiring one extra argument for the two-element incompatible traces allowed by Lemma~\ref{lem:3-graph-deletion}.

We begin by identifying many indices for which the residual subgraphs are dense.  The following claim is the only point at which we use the global density assumption on $\gG$.

\begin{Claim}\label{clm:dense-pieces}
There is an index set $I\subseteq [m]$ with $|I|\ge m/4$ such that, for every $i\in I$, there is a set $\rR_i\subseteq \hH_i$ satisfying
\(|\hH_i\setminus \rR_i|\ge (r-2-2\epsilon r)\binom{n}{2}\), \(|\alpha(\hH_i)|=O_r(1),\)
and every pair $vv'\in \binom{[n]}2$ satisfies
\(d_{\hH_i\setminus \rR_i}(vv')\le r-2.\)
\end{Claim}

\begin{proof}
For each $i$, the family $\aA_i$ is a union of at most $r-2$ matchings.  Hence $|\aA_i|\le n$, every member of $\aA_i$ is compatible with $O_r(1)$ other members of $\aA_i$, and $\hH_i$ is good by Lemma~\ref{clm:H-properties}.  Applying Lemma~\ref{lem:3-graph-deletion} to $\hH_i=\hH_i(\aA_i,[n])$ gives a set $\rR_i\subseteq \hH_i$ with
\[
|\rR_i|
\le \binom{|\aA_i|}{2}+\binom{n}{2}+O_r((|\aA_i|+n)(|\alpha(\hH_i)|+1))
\le 2\binom{n}{2}+O_r(n\cdot |\alpha(\hH_i)|+n).
\]
By Lemma~\ref{lem:partition-H}, we have
\(\sum_i |\hH_i|=|\hH|-O_r(n^{r-2})=\binom r2|\gG|-O_r(n^{r-2})\) and \(\sum_i |\alpha(\hH_i)|=O_r(m).\)
Therefore
\[
\sum_i |\rR_i|\le 2m\binom{n}{2}+O_r(n^{r-2}).
\]
The structural conclusion of Lemma~\ref{lem:3-graph-deletion} says that, in $\hH_i\setminus\rR_i$, every $Y$-pair-neighborhood is either compatible or has size at most two.  Since any compatible subfamily of $\aA_i$ has size at most $r-2$, we have
\(d_{\hH_i\setminus \rR_i}(vv')\le r-2\) for all \(vv'\in\binom{[n]}2.\)
Consequently, $|\hH_i\setminus \rR_i|\le (r-2)\binom n2$ for every $i$.

Combining the preceding estimates with $|\gG|\ge (1-\epsilon)\binom{n-1}{r-1}$ yields
\[
\sum_i |\hH_i\setminus \rR_i|
\ge (r-2-\epsilon r)\cdot m\binom{n}{2}-O_r(n^{r-2}).
\]
Since $|\hH_i\setminus \rR_i|\le (r-2)\binom n2$ for every $i$, there are at least $m/3$ indices $i$ satisfying
\(|\hH_i\setminus \rR_i|\ge (r-2-2\epsilon r)\binom{n}{2}.\)
Also, because $\sum_i |\alpha(\hH_i)|=O_r(m)$, all but at most $m/12$ indices have $|\alpha(\hH_i)|=O_r(1)$ after increasing the implicit constant if necessary.  The intersection of these two sets of indices has size at least $m/4$.
\end{proof}

We now fix an index $i\in I$ from Claim~\ref{clm:dense-pieces} and define the coloring used below.  Write
\[
        \hH_i^*:=\hH_i\setminus \rR_i
\]
and define the graph of maximum-degree $Y$-pairs by
\[
        G=\left\{vv'\in \binom{[n]}{2}: d_{\hH_i^*}(vv')= r-2 \right\}.
\]
The density and degree bound from Claim~\ref{clm:dense-pieces} imply
\[
        |G|\ge (1-2\epsilon r)\binom{n}{2}.
\]
By Lemma~\ref{lem:3-graph-deletion}, the neighborhood of every $Y$-pair in $\hH_i^*$ is either compatible, or consists of two incompatible vertices coming from a $\beta(\hH_i)$-triangle.  Therefore, when $r\ge5$, every edge of $G$ has a compatible neighborhood of size $r-2$; when $r=4$, there may also be maximum-degree pairs whose neighborhood consists of two incompatible vertices.  We separate these pairs below.

Let $\{\mathcal S_j\}_{j\in J}$ be the collection of all $(r-2)$-compatible subfamilies of $\aA_i$.  If a pair $vv'$ belongs to $\CN_{\hH_i^*}(\mathcal S_j)$, then $N_{\hH_i^*}(vv')$ contains the $r-2$ members of $\mathcal S_j$.  By the degree bound it follows that $vv'\in G$ and $N_{\hH_i^*}(vv')=\mathcal S_j$.  Hence the graphs $\CN_{\hH_i^*}(\mathcal S_j)$ are pairwise edge-disjoint.  Moreover, for $r\ge5$ they partition $G$.

Since $\aA_i$ is a union of at most $r-2$ matchings, there are only $O_r(n)$ compatible families of size $r-2$.  Thus $|J|=O_r(n)$.  Split the colors into two classes.  Let $J_1$ consist of those $j\in J$ for which $\CN_{\hH_i^*}(\mathcal S_j)$ is a star, and put $J_2=J\setminus J_1$.  If $j\in J_2$, then for every pair $A,A'\in\mathcal S_j$, the common neighborhood $\CN_{\hH_i}(A,A')$ is not a star; hence $\{A,A'\}\in \alpha(\hH_i)$.  Because $|\alpha(\hH_i)|=O_r(1)$ and each member of $\aA_i$ is compatible with only $O_r(1)$ other members of $\aA_i$, each exceptional pair extends to only $O_r(1)$ possible $(r-2)$-compatible families. 
Therefore, we have \(|J_1|=O_r(n)\) and \(|J_2|=O_r(1).\)
Define the two colored parts
\[
        G_1=\bigcup_{j\in J_1}\CN_{\hH_i^*}(\mathcal S_j), \quad \mbox{and} \quad
        G_2=\bigcup_{j\in J_2}\CN_{\hH_i^*}(\mathcal S_j).
\]
For $r\ge5$ we have
\[
        |G_1|+|G_2|=|G|\ge (1-2\epsilon r)\binom{n}{2}.
\]

It remains, for the coloring setup, to separate the extra maximum-degree pairs that can occur when $r=4$.  In this case a maximum-degree pair may have a two-element incompatible neighborhood.  Let
\[
        G_0=G\setminus (G_1\cup G_2).
\]
For every $ab\in G_0$, Lemma~\ref{lem:3-graph-deletion} gives a vertex $c$ such that $abc$ is a $\beta(\hH_i)$-triangle and $ab$ is an edge of that triangle.  The final assertion of Lemma~\ref{lem:3-graph-deletion} says that, for every such triangle, either one of the three corresponding $Y$-pair-neighborhoods has size at most one, or at most one of the three has a two-element incompatible neighborhood.

Let $G'_0\subseteq G_0$ be the set of edges arising from the second alternative.  Then each relevant $\beta$-triangle contributes exactly one edge to $G'_0$.  Since the $\beta(\hH_i)$-triangles are pairwise edge-disjoint by Definition~\ref{Prop:3graph-relation}, their contribution has bounded degree: if an edge $vu\in G'_0$ comes from a triangle $vuw$, then the same triangle also uses the edge $vw$, and these companion edges are distinct for distinct triangles through $v$.  Hence
\begin{equation}\label{equ:dGprime0}
        d_{G'_0}(v)\le n/2 \qquad\text{for every }v\in[n].
\end{equation}
Now the number of $Y$-pairs with remaining degree at most one is at most
$\binom n2-|G|\le 2\epsilon r\binom n2$.  A $\beta$-triangle in the first alternative contains at least one such low-degree $Y$-pair and contributes at most two edges to $G_0$; edge-disjointness prevents one low-degree pair from being charged more than once.  Therefore the first alternative contributes at most $4\epsilon r\binom n2$ edges to $G_0$.  We conclude that, when $r=4$,
\[
        |G'_0|+|G_1|+|G_2|
        \ge (1-6\epsilon r)\binom{n}{2}.
\]

For later use, color all relevant pairs as follows.  An edge in $\CN_{\hH_i^*}(\mathcal S_j)$ receives color $j$, and an edge of $G_0$ receives color $0$.

We next record the forbidden configurations for this coloring.  The first one says that the star-colors cannot create too many edges.

\begin{Claim}\label{clm:rainbow-K23}
There do not exist distinct vertices $v_1,v_2,w_1,w_2,w_3$ such that, for each $\ell\in[3]$, the two edges $v_1w_\ell$ and $v_2w_\ell$ both lie in $G_1$ and have the same color $c_\ell$, where $c_1,c_2,c_3\in J_1$ are pairwise distinct.
\end{Claim}

\begin{proof}
Suppose such a configuration exists.  Fix $\ell\in[3]$.  Since the two edges $v_1w_\ell$ and $v_2w_\ell$ have color $c_\ell$, for every $T\in\mathcal S_{c_\ell}$ the two $r$-sets $T\cup\{v_1,w_\ell\}$ and $T\cup\{v_2,w_\ell\}$ are edges of the original $r$-graph $\gG$.  Thus $\{w_\ell\}\cup T\in\CN_{\gG}(v_1,v_2)$ for every $T\in\mathcal S_{c_\ell}$.  Since $\CN_{\gG}(v_1,v_2)$ is a star with center $\fr(v_1,v_2)$,
\[
        \fr(v_1,v_2)\in \{w_\ell\}\cup \bigcap \mathcal S_{c_\ell}
        \qquad\text{for every }\ell\in[3].
\]
The vertices $w_1,w_2,w_3$ are distinct, so the center is different from at least two of them.  Hence for two distinct indices $\ell,\ell'$ we have
\[
        \fr(v_1,v_2)
        \in \bigcap \mathcal S_{c_\ell}\cap \bigcap \mathcal S_{c_{\ell'}}
        =\bigcap(\mathcal S_{c_\ell}\cup\mathcal S_{c_{\ell'}}).
\]
Because $c_\ell\ne c_{\ell'}$, the two compatible families are distinct, and their union contains at least $r-1$ distinct members of $\aA_i$.  Thus one vertex lies in at least $r-1$ members of $\aA_i$, impossible because $\aA_i$ is a union of at most $r-2$ matchings.  This contradiction proves the claim.
\end{proof}

\begin{Claim}\label{clm:G1-sparse}
$|G_1|=O_r(n^{3/2})$.
\end{Claim}

\begin{proof}
Construct a bipartite graph $F$ with parts $J_1$ and $[n]$ as follows.  If $\CN_{\hH_i^*}(\mathcal S_j)$ is a star with leaf set $L_j\subseteq[n]$, add the edges $jv$ for all $v\in L_j$.  Then $|F|=|G_1|$.  A copy of $K_{3,2}$ in $F$ with the part of size three in $J_1$ would give exactly the forbidden configuration in Claim~\ref{clm:rainbow-K23}.  Hence $F$ is $K_{3,2}$-free.  Since $|J_1|=O_r(n)$, the Zarankiewicz bound gives
\(|G_1|=|F|\le Z(O_r(n),n;3,2)=O_r(n^{3/2}).\)
\end{proof}

The non-star compatible colors are few.  Remove their common intersections by setting
\[
        B=\bigcup_{j\in J_2}\bigcap \mathcal S_j.
\]
Since $|J_2|=O_r(1)$ and $|\bigcap\mathcal S_j|\le r-2$, we have $|B|=O_r(1)$.  Let
\(G_2':=G_2[[n]\setminus B].\)
Then $|G_2'|\ge |G_2|-O_r(n)$.

\begin{Claim}\label{clm:rainbow-K22}
The graph $G_2'$ contains no $C_4$ on vertices $v_1,v_2,w_1,w_2$ such that
\[
        c(v_1w_1)=c(v_2w_1)=c_1, \quad \mbox{and} \quad
        c(v_1w_2)=c(v_2w_2)=c_2,
\]
where $c_1,c_2\in J_2$ and $c_1\ne c_2$.
\end{Claim}

\begin{proof}
Suppose such a $C_4$ exists.  For $j=1,2$ and every $T\in\mathcal S_{c_j}$, the set $\{w_j\}\cup T$ lies in $\CN_{\gG}(v_1,v_2)$.  Therefore, we have
\(\fr(v_1,v_2)\in \{w_j\}\cup \bigcap\mathcal S_{c_j}\) for each $j\in \{1,2\}$.
All four vertices of the $C_4$ lie outside $B$, while both $\bigcap\mathcal S_{c_1}$ and $\bigcap\mathcal S_{c_2}$ are contained in $B$.  Hence the center is neither $w_1$ nor $w_2$, and so
\(\fr(v_1,v_2)\in \bigcap\mathcal S_{c_1}\cap \bigcap\mathcal S_{c_2}.\)
As in the proof of Claim~\ref{clm:rainbow-K23}, this would put one vertex in at least $r-1$ distinct members of $\aA_i$, contradicting the fact that $\aA_i$ is a union of at most $r-2$ matchings.
\end{proof}

\begin{Claim}\label{clm:K2t}
If $G_2'$ contains a $K_{2,t}$, then it contains a monochromatic $K_{2,t'}$ with $t'=t-O_r(1)$.
\end{Claim}

\begin{proof}
Let the two-vertex part be $\{v,v'\}$ and the other part be $\{w_1,\ldots,w_t\}$.  For each $k\in[t]$, write
\[
        (c_k,c'_k)=(c(w_kv),c(w_kv')).
\]
There are only $|J_2|(|J_2|-1)=O_r(1)$ ordered pairs $(c,c')$ with $c\ne c'$ and $c,c'\in J_2$.  Moreover, each such ordered pair occurs for at most one $w_k$; otherwise Claim~\ref{clm:rainbow-K22} would give a forbidden $C_4$.  Hence all but $O_r(1)$ of the vertices $w_k$ satisfy $c_k=c'_k$.

Among the remaining vertices, two different diagonal colors cannot occur simultaneously, again by Claim~\ref{clm:rainbow-K22}. Hence all but $O_r(1)$ columns have the same color, yielding the required monochromatic $K_{2,t'}$.
\end{proof}

We first handle the case $r\ge5$.  In this case there are no incompatible maximum-degree pairs, so $G=G_1\cup G_2$.  By Claim~\ref{clm:G1-sparse},
\[
|G_2'|
\ge |G_2|-O_r(n)
=|G|-|G_1|-O_r(n)
\ge (1-2\epsilon r)\binom{n}{2}-O_r(n^{3/2}).
\]
Thus there exist two vertices $v,v'$ of degree at least
$(1-2\epsilon r)n-O_r(n^{1/2})$ in $G_2'$,
since otherwise all but at most one vertex would have smaller degree, contradicting the above lower bound on $|G_2'|$. 
Consequently,
\[
        |\CN_{G_2'}(v,v')|
        \ge d_{G_2'}(v)+d_{G_2'}(v')-n
        \ge (1-4\epsilon r)n-O_r(n^{1/2}).
\]
Then, Claim~\ref{clm:K2t} gives a set $W\subseteq\CN_{G_2'}(v,v')$ of size at least $(1-5\epsilon r)n$ such that all edges between $\{v,v'\}$ and $W$ have the same color, say $j\in J_2$.

We now translate this monochromatic bipartite graph back to $\gG$.  For any distinct $w,w'\in W$ and any $T\in\mathcal S_j$, the two edges $T\cup\{v,w\}$ and $T\cup\{v,w'\}$ of $\gG$ show that $T\cup\{v\}\in\CN_{\gG}(w,w')$.  The analogous two edges with $v'$ show that $T\cup\{v'\}\in\CN_{\gG}(w,w')$.  Since $T$ is disjoint from $v$ and $v'$, the star center $\fr(w,w')$ must lie in $T$.  As this holds for every $T\in\mathcal S_j$, we have
\(\fr(w,w')\in \bigcap \mathcal S_j\) for all distinct \(w,w'\in W.\)
Thus $\mathcal S_j$ and $W$ satisfy the conclusion of the lemma.

It remains to treat $r=4$.  The only new issue is the graph $G'_0$ coming from $\beta$-triangles.  Work inside $[n]\setminus B$ and set
\(G_0''=G_0'[[n]\setminus B].\)
Then $|G_0''|\ge |G_0'|-O_r(n)$, and the estimates above give
\begin{equation}\label{equ:G0primeplusG2prime}
|G_0''|+|G_2'|
\ge (1-6\epsilon r)\binom{n}{2}-|G_1|-O_r(n)
\ge (1-6\epsilon r)\binom{n}{2}-O_r(n^{3/2}).
\end{equation}
Let $V\subseteq[n]\setminus B$ be the set of vertices $v$ satisfying
\(d_{G_0''\cup G_2'}(v)\ge 7n/8.\)
From \eqref{equ:G0primeplusG2prime},
we see \(|V|\ge (1-48\epsilon r)n-O_r(n^{1/2}).\)
Indeed, each vertex outside $V$ has missing degree at least $n/8$ in $G_0''\cup G_2'$, whereas the total missing degree is at most $12\epsilon r\binom n2+O_r(n^{3/2})$.  For every $v\in V$, \eqref{equ:dGprime0} gives
\[
        d_{G_2'}(v)
        \ge d_{G_0''\cup G_2'}(v)-d_{G_0''}(v)
        \ge 7n/8-n/2=3n/8.
\]

We next remove the possibility that a whole $\beta$-triangle lies inside $V$.  Suppose $a,b,c\in V$ form a $\beta(\hH_i)$-triangle.  Since $d_{G_2'}(a)+d_{G_2'}(b)+d_{G_2'}(c)\ge 9n/8$, some pair among $a,b,c$ has at least $n/24$ common neighbors in $G_2'$.  To see this, let $q(z)$ be the number of neighbors of $z$ in the set $\{a,b,c\}$ within $G_2'$.  Then
\[
\sum_{\{x,y\}\in\binom{\{a,b,c\}}2}|\CN_{G_2'}(x,y)|
=\sum_{z\in[n]}\binom{q(z)}2
\ge \sum_{z\in[n]} q(z)-n
\ge n/8.
\]
Thus one of the three pairs has at least $n/24$ common neighbors; assume $|\CN_{G_2'}(a,b)|\ge n/24$.  By Claim~\ref{clm:K2t}, all but $O_r(1)$ of these common neighbors yield a monochromatic complete bipartite graph between $\{a,b\}$ and a large subset, with some color $j\in J_2$.  The translation argument from the previous paragraph then gives $\fr(a,b)\in\bigcap\mathcal S_j\subseteq B$.  But Lemma~\ref{clm:H-properties} gives $\fr(a,b)=c$ for the $\beta$-triangle $abc$, while $c\in V\subseteq[n]\setminus B$, a contradiction.

Thus every $\beta(\hH_i)$-triangle contributing to $G_0''$ has at least one vertex outside $V$.  Since these triangles are edge-disjoint, a fixed vertex $v\in V$ can be involved in at most $n-|V|$ such triangles, up to the harmless factor two coming from the two $Y$-edges from $v$ in each triangle.  Therefore
\[
        d_{G_0''}(v)\le 2(n-|V|)
        \le 96\epsilon r n+O_r(n^{1/2})
        \qquad \text{for } v\in V.
\]
Using \eqref{equ:G0primeplusG2prime} once more, we can choose two vertices $v,v'\in V$ such that
\[
        d_{G_0''\cup G_2'}(v),\ d_{G_0''\cup G_2'}(v')
        \ge (1-7\epsilon r)n.
\]
Indeed, if at most one vertex of $V$ had such degree, then the missing-degree sum over $V$ would be larger than allowed by \eqref{equ:G0primeplusG2prime}, using $|V|\ge(1-48\epsilon r)n-O_r(n^{1/2})$ and $\epsilon<1/(420r)$.
For these two vertices,
\[
        d_{G_2'}(v),\ d_{G_2'}(v')
        \ge (1-103\epsilon r)n-O_r(n^{1/2}),
\]
and hence
\[
        |\CN_{G_2'}(v,v')|
        \ge (1-207\epsilon r)n-O_r(n^{1/2}).
\]
Applying Claim~\ref{clm:K2t} and then the same translation argument as in the case $r\ge5$, we obtain an $(r-2)$-compatible family $\mathcal S$ and a set $W\subseteq\CN_{G_2'}(v,v')$ with
$|W|\ge (1-210\epsilon r)n$
such that $\fr(w,w')\in\bigcap\mathcal S$ for all distinct $w,w'\in W$.  
This finishes the proof of Lemma~\ref{lem:cluster-S-set-W}.
\end{proof}

\subsection{Completion of the proof of Lemma~\ref{lem:CN-stability}}\label{subsec:proof-CN-stability}

We now combine the preceding results to prove Lemma~\ref{lem:CN-stability}.
Recall that Lemma~\ref{lem:cluster-S-set-W} yields at least a quarter of the indices $i\in[m]$, each with a large vertex set $W_i$ and a compatible family $\mathcal S_i$ such that $\fr(x,y)\in\bigcap\mathcal S_i$ for every pair $\{x,y\}\subseteq W_i$.

\begin{proof}[Proof of Lemma~\ref{lem:CN-stability}]
For each $i \in I$, let $\mathcal S_i\subseteq \aA_i$ be the $(r-2)$-compatible family and $W_i$ be the vertex set obtained from Lemma~\ref{lem:cluster-S-set-W}. Note that $|I|\geq m/4$. 

Fix an $i\in I$. We aim to find another $k\in I$ such that $| (\bigcap \mathcal S_k) \cap (\bigcap \mathcal S_i)|\le 1$.
Note that $|\bigcap \mathcal S_i|\le r-3$.\footnote{The family $\mathcal S_i$ has $r-2$ distinct $(r-2)$-sets, so their common intersection cannot have size $r-2$.}
Consequently, there exist at most $\binom{r-3}{2}\binom{n}{r-4}$ distinct sets in $\binom{[n]}{r-2}$ that intersect $\bigcap \mathcal S_i$ at two or more vertices.
Since $\{\aA_j\}_{j\in [m]}$ are disjoint, all families $\mathcal S_j$ for $j \in I$ are also disjoint.
Thus, there are at most
$\binom{r-3}{2}\binom{n}{r-4}\ll |I|$
distinct $\mathcal S_j$ such that $|\bigcap \mathcal S_j\cap (\bigcap \mathcal S_i)|\ge 2$.
In other words, there exists some $k\in I$ such that $| (\bigcap \mathcal S_k) \cap (\bigcap \mathcal S_i)|\le 1$. 

Now let $W=W_i\cap W_k$. Since $|W_i|, |W_k|\ge (1-210 \epsilon r)n$, we have $|W|\ge (1-420\epsilon r)n$. For any $x,y\in W $, we have $\fr(x,y)\in (\bigcap \mathcal S_k) \cap (\bigcap \mathcal S_i)$ . Thus $(\bigcap \mathcal S_k) \cap (\bigcap \mathcal S_i)$ is not empty, and $\fr(x,y)=u$. Consequently, for any two vertices $x,y\in W$,
$\CN_{\gG}(x,y)$ is an $(r-1)$-star centered at $u$. 
This proves Lemma~\ref{lem:CN-stability}.
\end{proof}

\section{Proof of the deletion lemma}\label{sec:remove-diag}
This section is devoted to proving Lemma~\ref{lem:3-graph-deletion}.  
Our starting point is the graph-deletion lemma of F\"uredi~\cite{Fu84} and Pikhurko and Verstra\"ete~\cite{PiVe09}, which states that every bipartite graph $G$ can be made $C_4$-free after deleting at most $|\dD(G)|$ edges.  
Here $\dD(G)$ denotes the set of {\it diagonals} of $G$, i.e., pairs of vertices with at least two common neighbors.

Now let $\fF=\fF(X,Y)$ be a bipartite $3$-graph. Let $\dD(\fF)=\cup_{v\in V(\fF)} \dD(\fF_v)$, where $\fF_v$ denotes the link graph of $v$.
Note that for each $y\in Y$, its link graph $\fF_y$ is an ordinary bipartite graph with parts $X$ and $Y$.  
Applying the graph-deletion lemma independently to all these links $\fF_y$ would delete at most $\sum_{y\in Y}|\dD(\fF_y)|$ edges.  This bound can be much larger than $|\dD(\fF)|$, since a diagonal may be counted in multiple link graphs.

To avoid this repeated counting, we do not require every link graph to be $C_4$-free.  Our deletion lemma deletes edges to eliminate incompatible diagonals, while compatible diagonals may remain.  In each link graph, the number of deleted edges is at most the number of incompatible diagonals under consideration, plus an error term.  We process the links in a suitable order so that each pair is counted as an incompatible diagonal at most once.  Hence the total number of deleted edges is at most $|\dD(\fF)|$ plus an error term, which will be bounded in terms of $\alpha(\fF)$.

The rest of this section is organized as follows.  In Subsection~\ref{subsec:remove-diag-1}, we prove the graph-deletion lemma for incompatible pairs.  In Subsection~\ref{subsec:triangle-taming}, we delete edges so that every surviving $\beta$-triangle $abc$ corresponds to a unique incompatible pair, which can then be assigned to one of the link graphs $\fF_a$, $\fF_b$, and $\fF_c$.  In Subsection~\ref{subsec:y-link-iteration}, we apply the graph-deletion lemma successively to the $Y$-links and bound the total number of deleted edges.  Finally, in Subsection~\ref{subsec:final-assembly}, we combine the preceding estimates and complete the proof of Lemma~\ref{lem:3-graph-deletion} with a final edge exchange.

Throughout this section we use the notation of Subsection~\ref{subsec:bipartite-3-graphs} and fix arbitrary linear orders on $X$ and $Y$, used only to choose canonical representatives and order the $Y$-links in the deletion process.

\subsection{A graph-deletion lemma with compatibility}\label{subsec:remove-diag-1}

In this subsection we prove a graph-deletion lemma (Lemma~\ref{lem:graph-remove-diagonals}) for a bipartite graph $H=H(X,Y)$ equipped with a compatibility relation $\leftrightarrow$ on $X$. The lemma deletes prescribed incompatible pairs in $X$ while preserving all unprescribed diagonals, with the extra cost charged to short compatible chains in $X$. 
Our argument is a modification of the diagonal-deletion procedure of Pikhurko and Verstra\"ete~\cite{PiVe09}, adapted to the compatible/incompatible relation on $X$; we include the details for completeness.

We now formally define the sets of diagonals. 
For a simple graph $H$, define
		\[\dD(H)=\{\{u,v\}: \cd_H(u,v)\ge 2 \},\]
		as the set of all {\it diagonals} of $H$. 
		Note that each diagonal is contained in at least one $C_4$. 
		A pair $\{u,v\}$ is a {\it half-diagonal} if $\cd_H(u,v)=1$. Let
			\[\eE(H)=\{\{u,v\}: \cd_H(u,v)=1 \}.\]
		For later use, we also extend this notation to a $3$-graph $\hH$ through its link graphs.
		Let 
		\begin{equation*}
			\dD(\hH)=\bigcup_{v\in V(\hH)}\dD(\hH_v) 
			\text{ \;\; and \;\; }
			\eE(\hH)=\big(\bigcup_{v\in V(\hH)}\eE(\hH_v)\big) \setminus \dD(\hH).
		\end{equation*}
		We call the elements of $\dD(\hH)$ and $\eE(\hH)$ diagonals and half-diagonals, respectively.
Specifically, when $\hH = \hH(X,Y)$ is a bipartite 3-graph, we denote by $\dD_X(\hH)$ and $\eE_X(\hH)$ the set of diagonals and half-diagonals whose vertices both lie in $X$. Similarly, $\dD_Y(\hH)$ denotes the set of diagonals whose vertices both lie in $Y$.

Let $H=H(X,Y)$ be a bipartite graph, where $X$ is equipped with a symmetric compatible/incompatible relation. 
Under this relation, every two vertices $x, x' \in X$ are either compatible ($x \leftrightarrow x'$) or incompatible ($x \nleftrightarrow x'$). 
With this relation established, we classify the diagonals in the $Y$-part as follows.

\begin{dfn}
 A diagonal $\{y,y'\}\in \dD_Y(H)$ is compatible if vertices in its common neighborhood $\CN_{H}(y,y')$ are pairwise compatible.
\end{dfn}
Note that compatibility of a $Y$-diagonal is a property of its current common neighborhood; a diagonal that is not compatible in $H$ may become compatible in a subgraph $H'\subseteq H$.

\begin{dfn}\label{dfn:T-H}
	Define $\mathcal T(H)$ to be the set of ordered triples of distinct vertices of $X$ given by
		\[
		\mathcal T(H)=\{(x_1,x_2,x_3): \{x_1,x_2\}, \{x_2,x_3\}\in \dD_X(H),\ x_1\leftrightarrow x_2,\ x_2\leftrightarrow x_3,\text{ and } \cd_H(x_1,x_2,x_3)>0\}.
		\]
	Here the triples are ordered: for instance, $(x_1,x_2,x_3)$ and $(x_3,x_2,x_1)$ are counted separately when $x_1\ne x_3$.
\end{dfn}
\begin{lem}\label{lem:graph-remove-diagonals}
	Let $H=H(X,Y)$ be a bipartite graph with a compatible/incompatible relation on $X$. Let $\pP\subseteq \dD(H)\cup \eE(H)$ be a set of pairs such that every diagonal in $\dD(H)\setminus \pP$ is compatible. Then there exists an edge set $R\subseteq E(H)$ of size $|R|\le |\pP|+ |\mathcal T(H)|$ such that
\[
 \cd_{H\setminus R}(x,x')=0
 \qquad\text{for every incompatible pair }\{x,x'\}\in \pP\cap \binom X2 .
\]
\end{lem}

The proof is given by a recursive deletion process. At step $s$, the set $R_s$ records the deleted edges and $\pP_s\subseteq \pP$ records the target pairs that remain in the bookkeeping. The algorithm actively processes only incompatible pairs in $\binom X2$; target pairs on the $Y$-side are retained only to control the diagonal invariant. The set $\mathcal L_s=\dD(H\setminus R_s)\setminus \pP_s$ contains the other diagonals that still exist. The key invariant is that every pair in $\mathcal L_s$ is compatible in the current graph; this is what prevents the process from creating new incompatible obstructions outside the prescribed target set. The auxiliary set $\mathcal T_s$ records the only extra deletions not charged directly to a removed target pair.

We now formally describe the recursive deletion process. Set
 $ R_0=\emptyset$, $\pP_0=\pP .$
Suppose that $R_{s-1}$ and $\pP_{s-1}$ have been defined. If $\pP_{s-1}$ contains no incompatible pair in $\binom X2$, the process stops. Otherwise choose an incompatible pair $\{x,x'\}\in \pP_{s-1}\cap \binom X2$ and put
\[
 \mathcal L_{s-1}=\dD(H\setminus R_{s-1})\setminus \pP_{s-1}.
\]
We distinguish the following three cases, according to the codegree of $\{x,x'\}$ in $H\setminus R_{s-1}$.

\emph{Case 1.} If $\cd_{H\setminus R_{s-1}}(x,x')=0$, define
\[
 R_s=R_{s-1},\qquad
 \pP_s=\pP_{s-1}\setminus \{\{x,x'\}\},\qquad
 \mathcal T_s=\emptyset .
\]

\emph{Case 2.} If $\cd_{H\setminus R_{s-1}}(x,x')=1$, let $y$ be the unique common neighbor of $x$ and $x'$. Define
\[
 R_s=R_{s-1}\cup\{xy\},\qquad
 \pP_s=\pP_{s-1}\setminus \{\{x,x'\}\},\qquad
 \mathcal T_s=\emptyset .
\]

\emph{Case 3.} Suppose that $\cd_{H\setminus R_{s-1}}(x,x')\ge 2$. Choose a common neighbor $y\in \CN_{H\setminus R_{s-1}}(x,x')$, and write $H_s=H\setminus R_{s-1}$. The notation and the corresponding deletions are displayed below.

\begin{center}
\vspace{-0.45\baselineskip}
\begin{minipage}[t]{0.58\textwidth}
\vspace{0pt}\hspace{15pt}
\[\begin{aligned}[t]
&U = \{u: \{x,u\}\in\pP_{s-1},\ yu\in E(H_s)\},\\
&U' = \{u': \{x,u'\}\in\mathcal L_{s-1},\ yu'\in E(H_s)\},\\
&V = \{v: \{y,v\}\in\pP_{s-1},\ xv\in E(H_s)\},\\
&V' = \{v': \{y,v'\}\in\mathcal L_{s-1},\ xv'\in E(H_s)\},\\ 
&\mathcal T_s = \{(u_1',x,u_2'): u_1',u_2'\in U',\ u_1'\nleftrightarrow u_2'\},\\
 &U''= \{u''\in U': u'' \text{ appears in a triple of }\mathcal T_s\}.
\end{aligned}\]
\vspace{0.35\baselineskip}

\noindent Set
\[
\begin{aligned}
R_s=R_{s-1}&\cup\{xv:v\in V\}\cup\{yu:u\in U\cup U''\},\\
\pP_s=\pP_{s-1}&\setminus
\big(\{\{x,u\}:u\in U\}\cup\{\{y,v\}:v\in V\}\big).
\end{aligned}
\]
\end{minipage}\hfill
\begin{minipage}[t]{0.36\textwidth}
\vspace{0pt} \hspace{-20pt}
\centering
\begin{tikzpicture}[scale=1.2,
 x=0.58cm,y=0.69cm,
 every node/.style={font=\scriptsize},
 point/.style={fill=black,circle,inner sep=1.15pt},
 kept/.style={line width=0.45pt},
 deleted/.style={dashed,line width=0.55pt},
 rep/.style={font=\scriptsize,fill=white,inner sep=0.45pt},
]
 \coordinate (x) at (0,0);
 \coordinate (y) at (5.10,0);
 \coordinate (u) at (0.12,4.05);
 \coordinate (up) at (0.12,2.58);
 \coordinate (upp) at (0.12,1.74);
 \coordinate (v) at (4.98,4.05);
 \coordinate (vp) at (4.98,2.30);

 \draw[fill=gray!6] (0.30,4.05) ellipse (0.82 and 0.48);
 \draw[fill=gray!6] (0.30,2.30) ellipse (0.82 and 0.86);
 \draw[fill=white] (0.30,1.74) ellipse (0.52 and 0.33);
 \draw[fill=gray!6] (4.80,4.05) ellipse (0.82 and 0.48);
 \draw[fill=gray!6] (4.80,2.30) ellipse (0.82 and 0.48);

 \node at (-0.70,4.42) {$U$};
 \node at (-0.70,3.02) {$U'$};
 \node at (-0.72,1.38) {$U''$};
 \node at (5.80,4.42) {$V$};
 \node at (5.80,2.67) {$V'$};

 \draw[kept] (x) -- (y);
 \draw[deleted] (x) -- (v);
 \draw[kept] (x) -- (vp);
 \draw[deleted] (y) -- (u);
 \draw[deleted] (y) -- (upp);
 \draw[kept] (y) -- (up);

 \draw[deleted] (1.72,4.55) -- (2.32,4.55);
 \node[anchor=west,font=\tiny,fill=white,inner sep=0.3pt] at (2.42,4.55) {deleted};
 \draw[kept] (1.72,4.25) -- (2.32,4.25);
 \node[anchor=west,font=\tiny,fill=white,inner sep=0.3pt] at (2.42,4.25) {kept};

 \node[point,label=below:$x$] at (x) {};
 \node[point,label=below:$y$] at (y) {};
 \node[point] at (u) {};
 \node[rep,anchor=west] at (0.34,4.07) {$u$};
 \node[point] at (up) {};
 \node[rep,anchor=west] at (0.34,2.60) {$u'$};
 \node[point] at (upp) {};
 \node[rep,anchor=west] at (0.34,1.76) {$u''$};
 \node[point] at (v) {};
 \node[rep,anchor=east] at (4.76,4.07) {$v$};
 \node[point] at (vp) {};
 \node[rep,anchor=east] at (4.76,2.32) {$v'$};

 \node[font=\tiny] at (0,-0.62) {$X$-side};
 \node[font=\tiny] at (5.10,-0.62) {$Y$-side};
\end{tikzpicture}
\captionof{figure}{Case~3}
\label{fig:enhanced-case3}
\end{minipage}
\vspace{-0.35\baselineskip}
\end{center}

The verification of this process is divided into four invariants: the cost bound, the description of new common neighborhoods, the compatibility of all non-target diagonals, and the fact that every processed incompatible $X$-target pair has codegree zero afterwards.
		
\begin{Claim}\label{clm:R}
For each $s\in [t]$, we have $|R_{s}\setminus R_{s-1}|\le |\pP_{s-1}\setminus \pP_{s}|+ |\mathcal T_{s}|$.
\end{Claim}
\begin{proof}
This is immediate for Cases 1 and 2. In Case 3, $|R_{s}\setminus R_{s-1}| = |V| + |U \cup U''| = |V| + |U| + |U''|$. Since each triple in $\mathcal T_s$ involves two vertices from $U'$ and $(u_1', x, u_2')$, $(u_2', x, u_1')$ are different triples, we have $|U''| \le |\mathcal T_s|$. The number of pairs removed from $\pP$ is $|\pP_{s-1}\setminus \pP_{s}| = |U| + |V|$. Combining these bounds gives the desired inequality.
\end{proof}
		
\begin{Claim}\label{clm:Case-3-CN}
In Case 3, for each $u \in U\cup U'$ and $v \in V\cup V'$, we have:
\begin{enumerate}
 \item[(1)] $\CN_{H\setminus R_{s-1}}(x,u)\subseteq \{y\}\cup V\cup V'$ and $\CN_{H\setminus R_{s-1}}(y,v)\subseteq \{x\}\cup U\cup U'$.
 \item[(2)] $\CN_{H\setminus R_{s}}(x,u)\subseteq \{y\}\cup V'$ and $\CN_{H\setminus R_{s}}(y,v)\subseteq \{x\}\cup U'\setminus U''$.
\end{enumerate}
\end{Claim}
\begin{proof}
Suppose, for a contradiction to item (1), that there exists $w \in \CN_{H\setminus R_{s-1}}(x,u)$ with $w\notin \{y\}\cup V\cup V'$. Then $xyuw$ spans a $C_4$ in $H\setminus R_{s-1}$, implying $\{y,w\}\in \dD(H\setminus R_{s-1})$. By definition, $w$ must be in $V \cup V'$, a contradiction. The same argument gives $\CN_{H\setminus R_{s-1}}(y,v )\subseteq \{x\}\cup U\cup U'$, for $v \in V\cup V'$. 
Item (2) follows directly from item (1) and the definition of $R_s$, which removes all edges $\{xv_i\}_{v_i\in V}$ and $\{yu_i\}_{u_i\in U\cup U''}$.
\end{proof}

\begin{Claim} \label{clm:L}
For each $s\in [t]$, every diagonal in $\mathcal L_{s}$ is compatible in $H\setminus R_{s}$.
\end{Claim}
\begin{proof}
We proceed by induction on $s$. The base case $s=0$ holds since, by assumption, every diagonal in $\dD(H)\setminus \pP$ is compatible. Assume the claim is true for $s-1$. Each diagonal in $\mathcal L_s \cap \mathcal L_{s-1}$ remains compatible in the subgraph $H\setminus R_s$. Passing to a subgraph cannot create a new diagonal. Thus a pair can enter $\mathcal L_s$ only because it has just been removed from the target set in Case 3; consequently $\mathcal L_s \setminus \mathcal L_{s-1} \subseteq \pP_{s-1}\setminus \pP_s$.

Suppose $\{x,u\}\in \mathcal L_{s} \setminus \mathcal L_{s-1}$ with $u\in U$. Any common neighbor $w \in \CN_{H\setminus R_s}(x,u)$ must be different from $y$. By Claim~\ref{clm:Case-3-CN}, $w \in V'$. Thus, $\{y,w\} \in \mathcal L_{s-1}$, which is a compatible diagonal in $H\setminus R_{s-1}$ by the inductive hypothesis. Since $x,u \in \CN_{H\setminus R_{s-1}}(y,w)$, we must have $x \leftrightarrow u$. This logic applies to any pair of vertices in $\CN_{H\setminus R_s}(x,u)$. Thus, $\{x,u\}$ is compatible in $H\setminus R_s$.

 Suppose $\{y,v\}\in \mathcal L_{s} \setminus \mathcal L_{s-1}$ with $v\in V$. Any two common neighbors $w_1, w_2 \in \CN_{H\setminus R_s}(y,v)$ must belong to $U' \setminus U''$ by Claim~\ref{clm:Case-3-CN}. By the definition of $U''$, $(w_1, x, w_2) \notin \mathcal T_s$, which implies $w_1 \leftrightarrow w_2$. Therefore, $\{y,v\}$ is compatible in $H\setminus R_s$.
 This completes the proof of the claim. 
\end{proof}

\begin{Claim}\label{clm:T}
No ordered triple is recorded in two different families $\mathcal T_s$.  Consequently $\sum_{s=1}^t|\mathcal T_s| \le |\mathcal T(H)|$.
\end{Claim}
\begin{proof}
First, observe that if $(u_1,x,u_2)\in \mathcal T_s$, then $\CN_{H\setminus R_s}(u_1,x,u_2)=\emptyset$.
Suppose for contradiction that there is a common neighbor $w\in \CN_{H\setminus R_s}(u_1,x,u_2)$. At step $s$, $w\ne y$ since $yu_1$ and $yu_2$ were added to $R_s$. By Claim~\ref{clm:Case-3-CN}, $w\in V'$. Thus, $\{y,w\}\in \mathcal L_{s-1}$ is compatible in $H\setminus R_{s-1}$ by Claim~\ref{clm:L}. As $u_1,u_2\in \CN_{H\setminus R_{s-1}}(y,w)$, it must be that $u_1\leftrightarrow u_2$, contradicting the definition of $\mathcal T_s$. Therefore an ordered triple recorded at step $s$ has no common neighbor after that step and cannot be recorded again later.

For each $(u_1,x,u_2)\in \mathcal T_s$, the pairs $\{u_1,x\}$ and $\{u_2,x\}$ belong to $\mathcal L_{s-1}$ and are diagonals. Claim~\ref{clm:L} shows that they are compatible, so $u_1\leftrightarrow x\leftrightarrow u_2$. Moreover, the chosen vertex $y$ is a common neighbor of $u_1,x,u_2$ in $H\setminus R_{s-1}$. Hence
$\mathcal T_s\subseteq \mathcal T(H\setminus R_{s-1})\subseteq \mathcal T(H)$.
Since no ordered triple occurs in two different $\mathcal T_s$, we get
$\sum_{s=1}^t|\mathcal T_s|\le |\mathcal T(H)|$.
\end{proof}

\begin{Claim} \label{clm:cdeq0}
If $\{x,u\}\in (\pP_{s-1}\setminus \pP_{s})\cap \binom X2$ is incompatible, then $\cd_{H\setminus R_s}(x,u)=0$.
\end{Claim}
\begin{proof}
We only need to consider Case 3. If $\{x,u\} \in (\pP_{s-1}\setminus \pP_{s})\cap \binom X2$ is incompatible, then $u \in U$. Suppose there is a common neighbor $w \in \CN_{H\setminus R_s}(x,u)$. Then $w \neq y$. By Claim~\ref{clm:Case-3-CN}, $w \in V'$, which implies $\{y,w\} \in \mathcal L_{s-1}$. Claim~\ref{clm:L} shows that $\{y,w\}$ is a compatible diagonal in $H\setminus R_{s-1}$. However, this   contradicts the fact that $x, u \in \CN_{H\setminus R_{s-1}}(y,w)$ and $x \nleftrightarrow u$.
\end{proof}

\begin{proof} [Proof of Lemma~\ref{lem:graph-remove-diagonals}]
Run the recursive deletion process until $\pP_t$ contains no incompatible pairs in $\binom X2$. Let $R=R_t$. By Claims~\ref{clm:R} and~\ref{clm:T},
\[|R|=\sum_{s=1}^t|R_{s}\setminus R_{s-1}| \le \sum_{s=1}^t (|\pP_{s-1}\setminus \pP_{s}| + |\mathcal T_{s}|) = |\pP\setminus \pP_{t}| + \sum_{s=1}^t|\mathcal T_s| \le |\pP| + |\mathcal T(H)|.\]
Consider an incompatible pair $\{x,x'\}\in \pP\cap \binom X2$. There must be some $s\in [t]$ such that $\{x,x'\}\in \pP_{s-1}\setminus \pP_{s}$. By Claim~\ref{clm:cdeq0}, $\cd_{H\setminus R_s}(x,x')=0$. Since $R_s \subseteq R_t=R$, we have $\cd_{H\setminus R}(x,x')=0$, as required.
\end{proof}

\subsection{Preparing the $\beta$-triangles}\label{subsec:triangle-taming}
We now start the proof of Lemma~\ref{lem:3-graph-deletion}.  Let $\fF=\fF(X,Y)$ be a good bipartite $3$-graph satisfying the hypotheses of the lemma. 
Let $\beta=\beta(\fF)$.
The $\beta$-triangles are exceptional configurations which we treat separately.
We record them in a form that can be carried through the link-deletion process without being conflated with the remaining incompatible pairs processed later.
If $abc$ is a $\beta$-triangle in a subgraph $\gG\subseteq \fF$, define its \emph{kernel} by
\[
 \kK_{\gG}(abc)=\CN_{\gG}(ab,bc,ca).
\]
Thus $\kK_{\gG}(abc)$ is the set of $X$-vertices adjacent to all three $Y$-edges of the triangle.

\begin{lem}\label{clm:tilde-H}
There exists a subgraph $\widetilde{\fF}\subseteq \fF$ such that
\begin{equation}\label{equ:F-tildeF}
 |\fF|-|\widetilde{\fF}|
 \le |\dD(\fF)\cup\eE_X(\fF)|
      -|\dD(\widetilde{\fF})\cup\eE_X(\widetilde{\fF})|
      +O_r\bigl(|X|(|\alpha(\fF)|+1)\bigr),
\end{equation}
and every $\beta(\widetilde{\fF})$-triangle $abc$ satisfies the following two properties:
\begin{enumerate}[label=(\roman*)]
 \item $\kK_{\widetilde{\fF}}(abc)$ consists of either two or three pairwise incompatible vertices;
 \item if $x,x''\in \kK_{\widetilde{\fF}}(abc)$, then there is no $x'\in X$ such that $x\leftrightarrow x'$ and $x'\leftrightarrow x''$.
\end{enumerate}
\end{lem}
\begin{proof}
We modify the kernels in three steps. See Figure~\ref{fig:three-reductions}. Since subgraphs of a good graph remain good, every $\beta$-triangle that remains after deletions is a surviving old $\beta$-triangle.

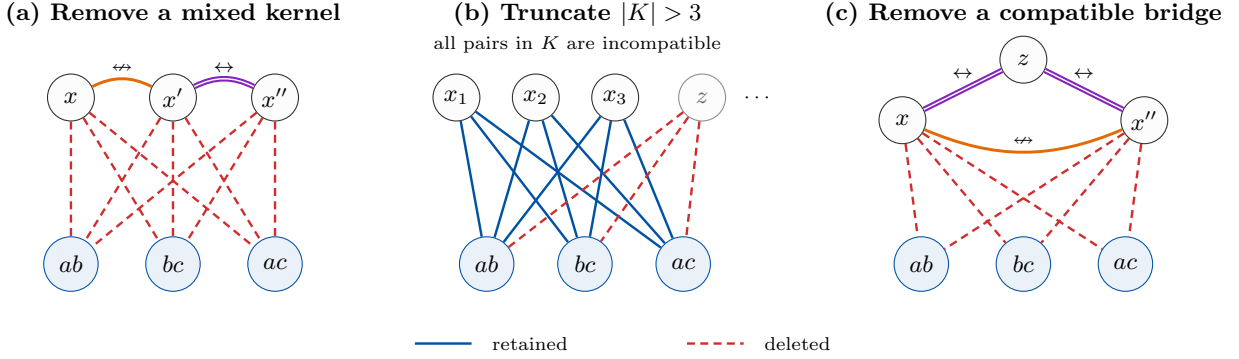
\begin{figure}[htbp]
\centering
\begin{tikzpicture}[x=1cm,y=1cm]
  \node[paneltitle] at (0,3.65) {(a) Remove a mixed kernel};
  \node[xvertex] (aX1) at (-1.35,2.55) {$x$};
  \node[xvertex] (aX2) at (0,2.55) {$x'$};
  \node[xvertex] (aX3) at (1.35,2.55) {$x''$};
  \node[pairvertex] (aP1) at (-1.35,0.35) {$ab$};
  \node[pairvertex] (aP2) at (0,0.35) {$bc$};
  \node[pairvertex] (aP3) at (1.35,0.35) {$ac$};
  \foreach \i in {1,2,3}{\foreach \j in {1,2,3}{
    \draw[deleted] (aX\i)--(aP\j);
  }}
  \draw[incompat,bend left=27] (aX1) to
    node[note,fill=white,inner sep=1pt,yshift=4pt] {$\nleftrightarrow$} (aX2);
  \draw[compat,bend left=27] (aX2) to
    node[note,fill=white,inner sep=1pt,yshift=4pt] {$\leftrightarrow$} (aX3);

  \node[paneltitle] at (5.35,3.65) {(b) Truncate $|K|>3$};
  \node[xvertex] (bX1) at (3.75,2.55) {$x_1$};
  \node[xvertex] (bX2) at (4.80,2.55) {$x_2$};
  \node[xvertex] (bX3) at (5.85,2.55) {$x_3$};
  \node[xvertex,draw=neutralstroke!55,fill=black!1,text=black!60]
    (bX4) at (7.00,2.55) {$z$};
  \node[note] at (7.75,2.55) {$\cdots$};
  \node[pairvertex] (bP1) at (4.15,0.35) {$ab$};
  \node[pairvertex] (bP2) at (5.45,0.35) {$bc$};
  \node[pairvertex] (bP3) at (6.75,0.35) {$ac$};
  \foreach \i in {1,2,3}{\foreach \j in {1,2,3}{
    \draw[live] (bX\i)--(bP\j);
  }}
  \foreach \j in {1,2,3}{\draw[deleted] (bX4)--(bP\j);}
  \node[note] at (5.35,3.25) {all pairs in $K$ are incompatible};

  \node[paneltitle] at (11.25,3.65) {(c) Remove a compatible bridge};
  \node[xvertex] (cX1) at (9.65,2.25) {$x$};
  \node[xvertex] (cZ)  at (11.25,3.05) {$z$};
  \node[xvertex] (cX2) at (12.85,2.25) {$x''$};
  \node[pairvertex] (cP1) at (9.90,0.35) {$ab$};
  \node[pairvertex] (cP2) at (11.25,0.35) {$bc$};
  \node[pairvertex] (cP3) at (12.60,0.35) {$ac$};
  \foreach \j in {1,2,3}{
    \draw[deleted] (cX1)--(cP\j);
    \draw[deleted] (cX2)--(cP\j);
  }
  \draw[compat] (cX1)--node[note,fill=white,inner sep=1pt,yshift=6pt]
    {$\leftrightarrow$}(cZ);
  \draw[compat] (cZ)--node[note,fill=white,inner sep=1pt,yshift=6pt]
    {$\leftrightarrow$}(cX2);
  \draw[incompat,bend right=23] (cX1) to
    node[note,fill=white,inner sep=1pt,yshift=4pt] {$\nleftrightarrow$} (cX2);
\draw[live] (3.2,-0.7) -- (3.95,-0.7);
  \node[note,anchor=west] at (4.10,-0.7) {retained};
  
  \draw[deleted] (6.8,-0.7) -- (7.55,-0.7);
  \node[note,anchor=west] at (7.70,-0.7) {deleted};
\end{tikzpicture}
\caption{The three reductions in Lemma~\ref{clm:tilde-H}.}
\label{fig:three-reductions}
\end{figure}

First, remove mixed compatible--incompatible configurations inside kernels.  By such a configuration we mean a quadruple $(abc;x,x',x'')$, where $abc$ is a $\beta(\fF)$-triangle, $x,x',x''\in\kK_{\fF}(abc)$, $x\nleftrightarrow x'$, and $x'\leftrightarrow x''$.  Then $\{x',x''\}\in\alpha(\fF)$, because $abc\subseteq \CN_{\fF}(x',x'')$ and this common neighborhood is not a star.  Conversely, fix a pair $\{x',x''\}\in\alpha(\fF)$ and a vertex $x\in X$.  There is at most one $\beta$-triangle giving a mixed configuration with these vertices: if it exists, then it is contained in $\CN_{\fF}(x,x')$, which, by goodness and $x\nleftrightarrow x'$, is exactly that triangle.  Hence the number of such configurations is $O(|X||\alpha(\fF)|)$.  For each configuration, delete all edges between $\{x,x',x''\}$ and the three pairs $ab,bc,ca$.  This deletes $O(|X||\alpha(\fF)|)$ edges.  After these deletions, every remaining $\beta$-triangle has a pairwise incompatible kernel. Indeed, a $\beta$-triangle has an incompatible pair in its kernel. If the same kernel contained a compatible pair, then, together with one endpoint of an incompatible pair, these vertices would give a mixed compatible--incompatible configuration of the kind removed above.

Second, truncate large kernels. If a remaining $\beta$-triangle $abc$ has kernel $K$ of size $k>3$, keep three vertices of $K$, and for every remaining vertex $z\in K$, delete the three edges $zab,zbc,zca$. This leaves a kernel of size three. We charge these deletions to the decrease of $|\dD(\cdot)\cup\eE_X(\cdot)|$. Indeed, all pairs in $K$ are incompatible. After the truncation, the only pairs of $K$ that remain diagonal are the three pairs among the retained vertices. Thus at least $\binom{k}{2}-3$ diagonal pairs disappear, whereas only $3(k-3)$ edges are deleted, and $3(k-3)\le \binom{k}{2}-3$. The charges are disjoint for distinct $\beta$-triangles, since the same incompatible pair cannot lie in the kernels of two different $\beta$-triangles; otherwise its common neighborhood would contain two distinct triangles, contradicting goodness. After this step, every remaining $\beta$-triangle has a kernel of size two or three, and this kernel is pairwise incompatible.

Finally, remove compatible chains through the remaining kernels. Suppose that a remaining $\beta$-triangle $abc$ has $x,x''$ in its kernel and that $x\leftrightarrow x'\leftrightarrow x''$ for some $x'\in X$. Since each vertex of $X$ has $O_r(1)$ compatible neighbors, there are $O_r(|X|)$ ordered chains $x\leftrightarrow x'\leftrightarrow x''$. For a fixed chain there is at most one relevant $\beta$-triangle, because $x\nleftrightarrow x''$ and goodness determines the triangle as $\CN_{\fF}(x,x'')$. For each such triangle $abc$, delete all edges between its kernel and the three pairs $ab,bc,ca$. Let $\widetilde{\fF}$ be the final graph. This deletes only $O_r(|X|)$ edges, since every remaining kernel has size two or three.

The first and third steps delete $O_r(|X|(|\alpha(\fF)|+1))$ edges in total, while the second step is charged to the decrease of $|\dD(\cdot)\cup\eE_X(\cdot)|$. This gives \eqref{equ:F-tildeF}. The construction also gives properties (i) and (ii).
\end{proof}

Fix arbitrary linear orders on $X$ and $Y$.  The order on $X$ determines which pair is retained from each three-vertex kernel, while the order on $Y$ distinguishes the pair $ab$ in each triangle $abc$ with $a<b<c$ and determines the order in which the $Y$-links are processed in Subsection~\ref{subsec:y-link-iteration}.

After Lemma~\ref{clm:tilde-H}, every surviving $\beta$-triangle $abc$ satisfies $|\kK_{\widetilde{\fF}}(abc)|\in\{2,3\}$.  The case $|\kK_{\widetilde{\fF}}(abc)|=2$ is already in the form allowed by Lemma~\ref{lem:3-graph-deletion}.  We next prepare the triangles with a three-vertex kernel: one kernel pair is kept as a surviving $\beta$-pair, and the other two kernel pairs are recorded as protected half-diagonals.

\begin{lem}\label{lem:typeII-normalization}
There exists a subgraph $\fF^0\subseteq \widetilde{\fF}$ and a set $\eE_0\subseteq \eE_X(\fF^0)$ such that the following hold.
\begin{enumerate}[label=(\roman*)]
 \item   
 $|\widetilde{\fF}|-|\fF^0|
 \le |\dD(\widetilde{\fF})\cup \eE_X(\widetilde{\fF})|-|\dD(\fF^0)\cup\eE_X(\fF^0)|+|\eE_0|. $
 \item  As sets of incompatible pairs in $\binom X2$, $\beta(\widetilde{\fF})=\beta(\fF^0)\sqcup \eE_0.$
 
 \item Every $\beta(\fF^0)$-triangle   $abc$ satisfies $|\kK_{\fF^0}(abc)|=2$. 
\end{enumerate}
Moreover, suppose that $abc$ is a $\beta(\widetilde{\fF})$-triangle with $a<b<c$ and $\kK_{\widetilde{\fF}}(abc)=\{x,x',x''\}$, where $x<x'<x''$(see Figure~\ref{fig:normalization-kernel3}).  Then $\{x,x'\}$ is the surviving $\beta(\fF^0)$-pair, while
 $     \{x,x''\},\{x',x''\}\in\eE_0.$
The vertex $x''$ is left adjacent to exactly one of the three $Y$-edges $ab,bc,ca$: it is adjacent only to $ac$ if $N_{\widetilde{\fF}}(ab)$ contains no compatible pair, and adjacent only to $ab$ otherwise.
\end{lem}

 \begin{figure}[htbp]
\centering
\begin{tikzpicture}[x=1cm,y=1cm]

  \node[xvertex] (tX1) at (-1.35,1.8) {$x$};
  \node[xvertex] (tX2) at (0,1.8) {$x'$};
  \node[xvertex] (tX3) at (1.35,1.8) {$x''$};
  
  \node[pairvertex] (tP1) at (-1.35,-0.2) {$ab$};
  \node[pairvertex] (tP2) at (0,-0.2) {$bc$};
  \node[pairvertex] (tP3) at (1.35,-0.2) {$ac$};
  
  \foreach \i in {1,2,3}{\foreach \j in {1,2,3}{
    \draw[live] (tX\i)--(tP\j);
  }}

  \node[xvertex] (lX1) at (-7.35,1.8) {$x$};
  \node[xvertex] (lX2) at (-6.00,1.8) {$x'$};
  \node[xvertex] (lX3) at (-4.65,1.8) {$x''$};
  
  \node[pairvertex] (lP1) at (-7.35,-0.2) {$ab$};
  \node[pairvertex] (lP2) at (-6.00,-0.2) {$bc$};
  \node[pairvertex] (lP3) at (-4.65,-0.2) {$ac$};
  
  \foreach \i in {1,2}{\foreach \j in {1,2,3}{
    \draw[live] (lX\i)--(lP\j);
  }}
  \draw[deleted] (lX3)--(lP1);
  \draw[deleted] (lX3)--(lP2);
  \draw[live] (lX3)--(lP3);
  
  \node[note,text width=4.2cm, align=center] at (-6,-1)
    {$N_{\widetilde{\fF}}(ab)$ contains no compatible pair};

  \node[xvertex] (rX1) at (4.65,1.8) {$x$};
  \node[xvertex] (rX2) at (6.00,1.8) {$x'$};
  \node[xvertex] (rX3) at (7.35,1.8) {$x''$};
  
  \node[pairvertex] (rP1) at (4.65,-0.2) {$ab$};
  \node[pairvertex] (rP2) at (6.00,-0.2) {$bc$};
  \node[pairvertex] (rP3) at (7.35,-0.2) {$ac$};
  
  \foreach \i in {1,2}{\foreach \j in {1,2,3}{
    \draw[live] (rX\i)--(rP\j);
  }}
  \draw[live] (rX3)--(rP1);
  \draw[deleted] (rX3)--(rP2);
  \draw[deleted] (rX3)--(rP3);
  
  \node[note,text width=4.2cm, align=center] at (6,-1)
    {$N_{\widetilde{\fF}}(ab)$ contains a compatible pair};

  \draw[arrow] (-2.2,0.8) -- (-3.8,0.8);
  \draw[arrow] (2.2,0.8) -- (3.8,0.8);
 
  \draw[live] (-1.9,-1.6)--(-1.25,-1.6);
  \node[note,anchor=west] at (-1,-1.6) {retained};
  \draw[deleted] (0.75,-1.6)--(1.45,-1.6);
  \node[note,anchor=west] at (1.55,-1.6) {deleted};

\end{tikzpicture}
\caption{Normalization of a $\beta$-triangle with
$\kK_{\widetilde{\fF}}(abc)=\{x,x',x''\}$, where $a<b<c$ and
$x<x'<x''$.}
\label{fig:normalization-kernel3}
\end{figure}
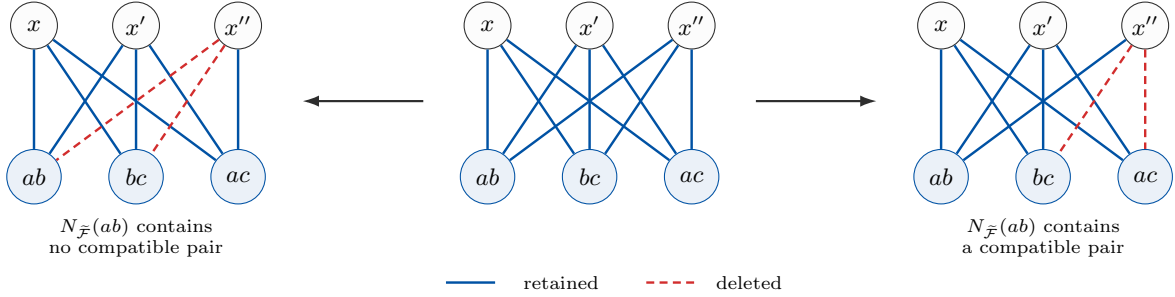

\begin{proof}
We treat the surviving $\beta(\widetilde{\fF})$-triangles with kernel size three.  Triangles with kernel size two are left unchanged.  List the three-kernel triangles as $a_ib_ic_i$, $i\in I$, with $a_i<b_i<c_i$, and write $\kK_{\widetilde{\fF}}(a_ib_ic_i)=\{x_i,x_i',x_i''\}$, where $x_i<x_i'<x_i''$.  For each such triangle, put $\{x_i,x_i''\}$ and $\{x_i',x_i''\}$ into $\eE_0$.  We then delete two of the three edges from $x_i''$ to the $Y$-edges of the triangle.  If $N_{\widetilde{\fF}}(a_ib_i)$ contains no compatible pair, delete $x_i''a_ib_i$ and $x_i''b_ic_i$, so that $x_i''$ remains adjacent only to $a_ic_i$.  Otherwise delete $x_i''a_ic_i$ and $x_i''b_ic_i$, so that $x_i''$ remains adjacent only to $a_ib_i$.  Let $\fF^0$ be the resulting subgraph.

The operations are independent for different triangles, because $\beta$-triangles have pairwise edge-disjoint $Y$-triangles.  In a processed three-kernel triangle, the pair $\{x_i,x_i'\}$ remains adjacent to all three $Y$-edges, and hence remains the unique $\beta(\fF^0)$-pair supported on that triangle.  The two pairs involving $x_i''$ are adjacent to exactly one $Y$-edge, so they are half-diagonals in $\fF^0$ and belong to $\eE_0$.  Triangles with kernel size two were not changed.  This proves the structural assertions and the recorded form for three-kernel triangles.

Exactly two edges are deleted for each processed triangle, and exactly two pairs are added to $\eE_0$.  Hence $|\widetilde{\fF}|-|\fF^0|=|\eE_0|$.  Since passing to a subgraph cannot create a new diagonal or half-diagonal, $\dD(\fF^0)\cup\eE_X(\fF^0)\subseteq \dD(\widetilde{\fF})\cup\eE_X(\widetilde{\fF})$, and (i) follows.
\end{proof}

\subsection{Deleting edges from the $Y$-links}\label{subsec:y-link-iteration}

We now delete the remaining incompatible pairs not protected by $\eE_0$, one $Y$-link at a time.  The protected half-diagonals in $\eE_0$ are not processed: they are the one-edge remnants created in Lemma~\ref{lem:typeII-normalization} and will be used only in the final edge exchange.

Let $\pP=(\dD(\fF^0)\cup \eE_X(\fF^0))\setminus \eE_0$.  We assign every pair in $\pP$ to one vertex of $Y$.  Define $g:\pP\to Y$ as follows.
\begin{itemize}
 \item If $\{x,x'\}\in \eE_X(\fF^0)\setminus\eE_0$, write $\CN_{\fF^0}(x,x')=\{yy'\}$ with $y<y'$, and set $g(x,x')=y$.
 \item If $\{x,x'\}\in \dD_X(\fF^0)$ is incompatible, then $\CN_{\fF^0}(x,x')$ is either a star or a triangle.  In the star case, let $g(x,x')$ be the center of the star.  In the triangle case, write the triangle as $abc$ with $a<b<c$, and set $g(x,x')=a$.
 \item If $\{x,x'\}\in \dD_X(\fF^0)$ is compatible, set $g(x,x')=\min\{y\in Y:\{x,x'\}\in \dD(\fF^0_y)\}$.
 \item If $\{u,v\}\in \dD_Y(\fF^0)$, let $g(u,v)$ be the vertex $w\in Y$ given by Definition~\ref{Prop:3graph-relation}(3).  Thus, for every $y\in Y\setminus\{u,v,w\}$, the family $\CN_{\fF}(uy,vy)$ is compatible.
\end{itemize}
Here the minimum is taken with respect to the fixed order on $Y$.  Write $Y=\{y_1,\ldots,y_n\}$ in this order, and set $\pP_i=\{p\in \pP:g(p)=y_i\}$ for $i\in[n]$.  Then the sets $\pP_i$ are pairwise disjoint and
\begin{equation}\label{equ-sum-Pi}
 \sum_{i\in[n]}|\pP_i|
 = |\dD(\fF^0)\cup\eE_X(\fF^0)|-|\eE_0|.
\end{equation}

\begin{lem}\label{clm:Hn}
There exists a subgraph $\fF^n\subseteq \fF^0$ such that
\begin{equation}\label{equ:link-cost}
 |\fF^0|-|\fF^n|
 \le |\dD(\fF^0)\cup \eE_X(\fF^0)|-|\eE_0|
 +O_r\bigl(|Y||\alpha(\fF)|+|X|\bigr),
\end{equation}
and, for every incompatible pair $x\nleftrightarrow x'$ in $X$,
\begin{equation}\label{equ:cd-xxprime}
 \cd_{\fF^n}(x,x')=
 \begin{cases}
 0 \text{ or } 1, & \{x,x'\}\in \eE_0\cup\beta(\fF^0)=\beta(\widetilde{\fF}),\\
 0, & \text{otherwise.}
 \end{cases}
\end{equation}
Moreover, if $abc$ is a $\beta(\fF^0)$-triangle with $a<b<c$ and $\kK_{\fF^0}(abc)=\{x,x'\}$, then (see Figure~\ref{fig:ordered-link})
\begin{equation}\label{equ:CNxxprime-ab}
 \CN_{\fF^n}(x,x')\subseteq\{bc\}.
\end{equation}
\end{lem}

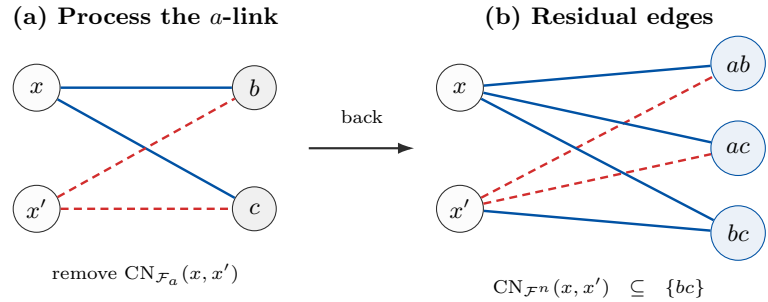
\begin{figure}[htbp]
\centering
\begin{tikzpicture}[scale=0.8,  x=1cm,y=1cm]
 
  \node[paneltitle] at (3.00,4.25) {(a) Process the $a$-link};
  \node[xvertex] (Lx) at (1.20,3.10) {$x$};
  \node[xvertex] (Lxp) at (1.20,1.10) {$x'$};
  \node[yvertex] (Lb) at (4.80,3.10) {$b$};
  \node[yvertex] (Lc) at (4.80,1.10) {$c$};
  \draw[live] (Lx)--(Lb);
  \draw[live] (Lx)--(Lc);
  \draw[deleted] (Lxp)--(Lb);
  \draw[deleted] (Lxp)--(Lc);
  \node[note,text width=4.6cm] at (3.00,0.02)
    {remove $\CN_{\fF_a}(x,x')$};

  \draw[arrow] (5.70,2.10)--(7.45,2.10);
  \node[note] at (6.58,2.62) {back};

  \node[paneltitle] at (10.50,4.25) {(b) Residual edges};
  \node[xvertex] (Rx) at (8.20,3.10) {$x$};
  \node[xvertex] (Rxp) at (8.20,1.10) {$x'$};
  \node[pairvertex] (Rab) at (12.80,3.50) {$ab$};
  \node[pairvertex] (Rac) at (12.80,2.10) {$ac$};
  \node[pairvertex] (Rbc) at (12.80,0.70) {$bc$};
  \draw[live] (Rx)--(Rab);
  \draw[deleted] (Rxp)--(Rab);
  \draw[live] (Rx)--(Rac);
  \draw[deleted] (Rxp)--(Rac);
  \draw[live] (Rx)--(Rbc);
  \draw[live] (Rxp)--(Rbc);
  \node[note,text width=5.2cm] at (10.50,-0.20)
    {$\CN_{\fF^n}(x,x')\subseteq\{bc\}$};
\end{tikzpicture}
\caption{Processing $Y$-links}
\label{fig:ordered-link}
\end{figure}

\begin{proof}
We process the $Y$-links in the order $y_1,\ldots,y_n$.  Suppose that $\fF^{s-1}$ has been constructed, and let $\fF^{s-1}_s$ be the $y_s$-link graph of $\fF^{s-1}$.  In this link we apply Lemma~\ref{lem:graph-remove-diagonals} with target set
\[
 \pP'_s=\pP_s\cap\bigl(\dD(\fF^{s-1}_s)\cup\eE(\fF^{s-1}_s)\bigr).
\]
Here and below, saying that an incompatible $X$-pair is processed in the $y_s$-link means that Lemma~\ref{lem:graph-remove-diagonals} deletes edges from this link until the pair has codegree zero there.

We first verify the hypothesis of the one-link lemma.  Let $q$ be a diagonal of $\fF^{s-1}_s$ with $q\notin\pP'_s$.  We must show that $q$ is compatible.

Assume first that $q=\{u,v\}\subseteq Y$.  Since $q$ is a diagonal in the $y_s$-link, we have $y_s\notin\{u,v\}$.  Also $q\in \dD_Y(\fF^0)\subseteq\pP$.  As $q\notin\pP'_s$ and $q$ is a current diagonal, it follows that $q\notin\pP_s$, and hence $g(q)\ne y_s$.  By the definition of $g$ for $Y$-diagonals, $\CN_{\fF}(uy_s,vy_s)$ is compatible.  The current common neighborhood $\CN_{\fF^{s-1}}(uy_s,vy_s)$ is a subfamily of it, and is therefore compatible.

Now assume that $q=\{x,x'\}\subseteq X$.  If $x\leftrightarrow x'$, then $q$ is compatible, so suppose $x\nleftrightarrow x'$.  Since $q$ is a diagonal in the current link, it is an incompatible diagonal of $\fF^0$.  It is not in $\eE_0$, because pairs in $\eE_0$ are half-diagonals in $\fF^0$ and deleting edges cannot turn a half-diagonal into a diagonal.  Thus $q\in\pP$.  As above, $q\notin\pP'_s$ implies $q\notin\pP_s$, so $g(q)\ne y_s$.

If $\CN_{\fF^0}(x,x')$ is a star, then the presence of two common neighbors in the $y_s$-link forces $y_s$ to be the center of that star; hence $g(q)=y_s$, a contradiction.  Therefore $\{x,x'\}\in\beta(\fF^0)$.  Write $\CN_{\fF^0}(x,x')=\{ab,bc,ca\}$ with $a<b<c$.  Then $g(q)=a$.  Since $g(q)\ne y_s$, we have $y_s\in\{b,c\}$.  The $a$-link was processed earlier, and the one-link lemma made the codegree of $q$ equal to zero in that link.  Equivalently, all common $Y$-edges for $x$ and $x'$ that contain $a$ were deleted at that stage.  The only edge of the triangle not containing $a$ is $bc$, so in the later $b$- or $c$-link the pair $q$ has codegree at most one.  It therefore cannot be a diagonal there, a contradiction.  Hence every non-target diagonal in the $y_s$-link is compatible, and Lemma~\ref{lem:graph-remove-diagonals} applies.

Let $R_s\subseteq E(\fF^{s-1}_s)$ be the edge set returned by Lemma~\ref{lem:graph-remove-diagonals}.  Then
\[
 |R_s|\le |\pP'_s|+|\mathcal T(\fF^{s-1}_s)|
       \le |\pP_s|+|\mathcal T(\fF^{s-1}_s)|.
\]
Define $\fF^s=\fF^{s-1}\setminus\{xyy_s:xy\in R_s\}$.  For each incompatible pair in $\pP'_s\cap\binom X2$, its codegree in the $y_s$-link of $\fF^s$ is zero.

We next record the structural consequences.  Let $abc$ be a $\beta(\fF^0)$-triangle with $a<b<c$ and kernel $\{x,x'\}$.  The pair $\{x,x'\}$ belongs to $\pP_a$, so it is processed in the $a$-link.  Hence no common $Y$-edge containing $a$ remains for $x$ and $x'$, and the only possible remaining edge of $\CN_{\fF^n}(x,x')$ is $bc$.  This proves \eqref{equ:CNxxprime-ab}.  Every incompatible pair with positive codegree in $\fF^0$, except for the protected pairs in $\eE_0$, belongs to $\pP$ and is processed in its assigned link.  Pairs in $\eE_0$ have codegree one in $\fF^0$ and are never processed.  This proves \eqref{equ:cd-xxprime}.

It remains to sum the deletion cost.  By \eqref{equ-sum-Pi},
\[
 |\fF^0|-|\fF^n|
 \le |\dD(\fF^0)\cup \eE_X(\fF^0)|-|\eE_0|
 +\sum_{s=1}^n |\mathcal T(\fF^{s-1}_s)|.
\]
Set $\mathcal T_s^0=\mathcal T(\fF^0_{y_s})$.  Since $\fF^{s-1}\subseteq\fF^0$, it suffices to bound $\sum_s|\mathcal T_s^0|$.  There are $O_r(|X|)$ ordered compatible chains $x\leftrightarrow x'\leftrightarrow x''$, because each vertex of $X$ has $O_r(1)$ compatible neighbors.  Fix one such chain.  If it is recorded in the $y$-link and in the $y'$-link with $y\ne y'$, then the compatible pair $\{x,x'\}$ has common $Y$-edges through two different vertices.  If $\CN_{\fF^0}(x,x')$ were a star, then all these common $Y$-edges would contain the same center; hence the chain could be recorded only in the link of that center.  Thus $\CN_{\fF^0}(x,x')$ is not a star.  Since a subfamily of a star is still a star, this implies $\{x,x'\}\in\alpha(\fF)$.  For each pair in $\alpha(\fF)$ there are only $O_r(1)$ choices for the third vertex of the chain.  Thus chains appearing in at least two links contribute $O_r(|Y||\alpha(\fF)|)$, while chains appearing in exactly one link contribute $O_r(|X|)$.  Therefore
\[
 \sum_{s=1}^n |\mathcal T(\fF^{s-1}_s)|
 \le \sum_{s=1}^n |\mathcal T_s^0|
 =O_r\bigl(|Y||\alpha(\fF)|+|X|\bigr),
\]
which proves \eqref{equ:link-cost}.
\end{proof}

\subsection{Completion of the proof}\label{subsec:final-assembly}

\begin{proof}[Proof of Lemma~\ref{lem:3-graph-deletion}]
Let $\rR^n=\fF\setminus\fF^n$.  We have three things to verify: the required size bound for the final deletion set, the asserted structure of every $Y$-pair neighborhood, and the additional conclusion for the $\beta(\fF)$-triangles.  The first two follow from the preceding construction, while the last one requires a final local adjustment for the remaining three-kernel triangles.

We first estimate the size of the deletion set already constructed.  Combining Lemmas~\ref{clm:tilde-H}, \ref{lem:typeII-normalization}, and \ref{clm:Hn}, we obtain
\begin{align}
|\rR^n|
&=(|\fF|-|\widetilde{\fF}|)+(|\widetilde{\fF}|-|\fF^0|)+(|\fF^0|-|\fF^n|) \notag\\
&\le |\dD(\fF)\cup\eE_X(\fF)|
 +O_r\bigl((|X|+|Y|)(|\alpha(\fF)|+1)\bigr) \notag\\
&\le \binom{|X|}{2}+\binom{|Y|}{2}
 +O_r\bigl((|X|+|Y|)(|\alpha(\fF)|+1)\bigr). \label{equ:final-size-bound}
\end{align}

Next we record the structural information about $Y$-pair neighborhoods in $\fF^n$.  Let $y,y'\in Y$, and suppose that $N_{\fF^n}(yy')$ is not compatible.  Choose incompatible vertices $x,x'\in N_{\fF^n}(yy')$.  By \eqref{equ:cd-xxprime}, the pair $\{x,x'\}$ lies in $\beta(\widetilde{\fF})\subseteq\beta(\fF)$, and $yy'$ is an edge of the corresponding $\beta(\widetilde{\fF})$-triangle, say $abc$.  We claim that
\begin{equation}\label{equ:Yn-structure}
        N_{\fF^n}(yy')\subseteq \kK_{\widetilde{\fF}}(abc).
\end{equation}
Indeed, take $v\in N_{\fF^n}(yy')$.  If $v$ is incompatible with one of $x,x'$, then \eqref{equ:cd-xxprime} and the edge-disjointness of the $\beta$-triangles force $v$ to lie in the same kernel.  If $v$ is compatible with both $x$ and $x'$, then Lemma~\ref{clm:tilde-H}(ii) gives a compatible chain through the incompatible kernel pair $\{x,x'\}$, which is impossible.  This proves \eqref{equ:Yn-structure}.

We now check the additional conclusion for surviving $\beta(\widetilde{\fF})$-triangles.  First let $abc$ be such a triangle with $a<b<c$ and $\kK_{\widetilde{\fF}}(abc)=\{x,x'\}$.  This triangle is unchanged by Lemma~\ref{lem:typeII-normalization}, so $\{x,x'\}$ is the kernel of a $\beta(\fF^0)$-triangle.  By \eqref{equ:CNxxprime-ab}, the pair $\{x,x'\}$ can remain only on the edge $bc$.  Therefore neither $N_{\fF^n}(ab)$ nor $N_{\fF^n}(ac)$ contains an incompatible pair.  By \eqref{equ:Yn-structure}, both neighborhoods are compatible.  Thus the second alternative in the moreover conclusion of Lemma~\ref{lem:3-graph-deletion} already holds for $abc$.

It remains to consider a surviving $\beta(\widetilde{\fF})$-triangle $abc$ with $a<b<c$ and $\kK_{\widetilde{\fF}}(abc)=\{x,x',x''\}$, where $x<x'<x''$.  By Lemma~\ref{lem:typeII-normalization}, the pair $\{x,x'\}$ is the surviving $\beta(\fF^0)$-pair, while $\{x,x''\}$ and $\{x',x''\}$ are the protected half-diagonals in $\eE_0$.  There are two cases, according to the recorded form in Lemma~\ref{lem:typeII-normalization}.

Suppose first that $N_{\widetilde{\fF}}(ab)$ contains no compatible pair.  Then $x''$ is left adjacent only to $ac$ among $ab,bc,ca$.  Also, by \eqref{equ:CNxxprime-ab}, the pair $\{x,x'\}$ can remain only on $bc$.  Hence $N_{\fF^n}(ab)$ contains no incompatible pair.  Since $N_{\fF^n}(ab)\subseteq N_{\widetilde{\fF}}(ab)$, and $N_{\widetilde{\fF}}(ab)$ contains no compatible pair, we must have $|N_{\fF^n}(ab)|\le 1$.  This gives the first alternative in the moreover conclusion.

Now suppose that $N_{\widetilde{\fF}}(ab)$ contains a compatible pair.  Then $x''$ is left adjacent only to $ab$ among $ab,bc,ca$.  Again \eqref{equ:CNxxprime-ab} shows that the surviving pair $\{x,x'\}$ can remain only on $bc$.  Therefore $N_{\fF^n}(ac)$ contains no incompatible pair, and hence it is compatible by \eqref{equ:Yn-structure}.

If the moreover conclusion already holds for this triangle, no further change is needed.  Otherwise all three neighborhoods have size at least two and at most one of them is compatible.  Since $N_{\fF^n}(ac)$ is compatible, the only possible obstruction is
\[
        N_{\fF^n}(bc)=\{x,x'\},
        \qquad
        N_{\fF^n}(ab)\in
        \bigl\{\{x,x''\},\{x',x''\}\bigr\}.
\]
We call such a triangle unresolved.

We repair the unresolved triangles one at a time.  List them as $a_ib_ic_i$, $i\in I$, with $a_i<b_i<c_i$ (see Figure~\ref{fig:local-repair}).   For each $i$, choose a compatible two-set $U_i\subseteq N_{\widetilde{\fF}}(a_ib_i)$ with $|U_i|=2$, which exists by the definition of the unresolved case, and put $V_i=N_{\fF^n}(a_ib_i)$.  Thus $V_i$ is an incompatible two-set.  Define
\[
 \rR=
 \left(\rR^n\cup\bigcup_{i\in I}\{va_ib_i:v\in V_i\setminus U_i\}\right)
 \setminus
 \bigcup_{i\in I}\{ua_ib_i:u\in U_i\setminus V_i\}.
\]

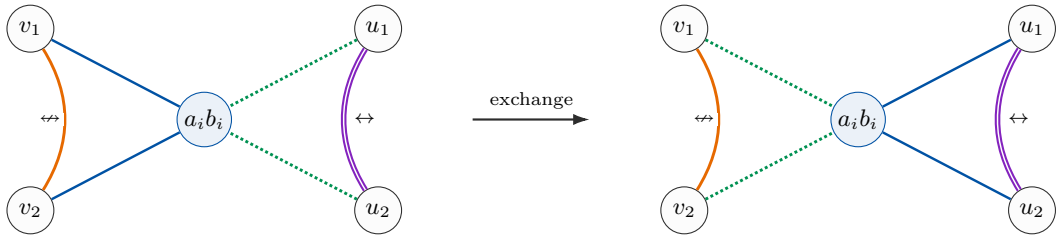
\begin{figure}[htbp]
\centering
\begin{tikzpicture}[x=1cm,y=1cm]
 
  \node[paneltitle] at (2.55,4.55) {(a) Before repair: $N_{\fF^n}(a_ib_i)=\{v_1,v_2\}$};
  \node[pairvertex] (abL) at (2.55,2.05) {$a_ib_i$};
  \node[xvertex] (v1L) at (0.25,3.25) {$v_1$};
  \node[xvertex] (v2L) at (0.25,0.85) {$v_2$};
  \node[xvertex] (u1L) at (4.85,3.25) {$u_1$};
  \node[xvertex] (u2L) at (4.85,0.85) {$u_2$};
  \draw[live] (v1L)--(abL);
  \draw[live] (v2L)--(abL);
  \draw[available] (u1L)--(abL);
  \draw[available] (u2L)--(abL);
  \draw[incompat,bend left=32] (v1L) to
    node[note,fill=white,inner sep=1pt,left] {$\nleftrightarrow$} (v2L);
  \draw[compat,bend left=32] (u2L) to
    node[note,fill=white,inner sep=1pt,xshift=8pt] {$\leftrightarrow$} (u1L);
  
  \draw[arrow] (6.10,2.05)--(7.65,2.05);
  \node[note,text width=2.8cm] at (6.88,2.30)
    {exchange};
 
  \node[paneltitle] at (11.20,4.55)
    {(b) After repair: $N_{\fF\setminus\mathcal R}(a_ib_i)=\{u_1,u_2\}$};
  \node[pairvertex] (abR) at (11.20,2.05) {$a_ib_i$};
  \node[xvertex] (v1R) at (8.90,3.25) {$v_1$};
  \node[xvertex] (v2R) at (8.90,0.85) {$v_2$};
  \node[xvertex] (u1R) at (13.50,3.25) {$u_1$};
  \node[xvertex] (u2R) at (13.50,0.85) {$u_2$};
  \draw[available] (v1R)--(abR);
  \draw[available] (v2R)--(abR);
  \draw[live] (u1R)--(abR);
  \draw[live] (u2R)--(abR);
  \draw[incompat,bend left=32] (v1R) to
    node[note,fill=white,inner sep=1pt,left] {$\nleftrightarrow$} (v2R);
  \draw[compat,bend left=32] (u2R) to
    node[note,fill=white,inner sep=1pt,xshift=8pt] {$\leftrightarrow$} (u1R);
\end{tikzpicture}
\caption{The local exchange at an unresolved triangle $a_ib_ic_i$.  }
\label{fig:local-repair}
\end{figure}
This replaces the incompatible neighborhood $V_i$ on $a_ib_i$ by the compatible two-set $U_i$.  For $u\in U_i\setminus V_i$, the edge $ua_ib_i$ belongs to $\widetilde{\fF}$ but not to $\fF^n$, and hence it is indeed an element of the previous deletion set $\rR^n$ that may be removed from that deletion set.  Since $|U_i\setminus V_i|=|V_i\setminus U_i|$, the size of the deletion set is unchanged, so $|\rR|=|\rR^n|$.  The replacements are independent because unresolved triangles come from distinct $\beta$-triangles, and the corresponding $Y$-triangles are pairwise edge-disjoint.  After the replacement, $N_{\fF\setminus\rR}(a_ib_i)=U_i$, which is compatible, while $N_{\fF\setminus\rR}(a_ic_i)$ was already compatible.  Hence every unresolved triangle now satisfies the second alternative in the moreover conclusion.

It remains only to verify that the two conclusions of Lemma~\ref{lem:3-graph-deletion} hold for the final graph $\fF\setminus\rR$.  For a $Y$-pair not used in the local adjustment, the neighborhood is the same as in $\fF^n$; for an adjusted pair, the neighborhood is compatible.  By \eqref{equ:Yn-structure}, any non-compatible $Y$-neighborhood is contained in the kernel of a surviving $\beta(\widetilde{\fF})$-triangle.  The preceding case analysis, together with the local adjustment, rules out the possibility that such a non-compatible neighborhood contains all three vertices of a three-kernel triangle: after the adjustment, the only non-compatible neighborhoods left inside surviving kernels are the two-sets corresponding to surviving $\beta(\fF^0)$-pairs.  Hence every non-compatible $Y$-neighborhood in $\fF\setminus\rR$ is a two-set $\{x,x'\}$ with $\{x,x'\}\in\beta(\widetilde{\fF})\subseteq\beta(\fF)$, and the corresponding $Y$-edge is an edge of that $\beta(\fF)$-triangle.  This proves the first conclusion of the lemma.

The moreover conclusion has already been verified for every surviving $\beta(\widetilde{\fF})$-triangle: two-kernel triangles satisfy it before the adjustment, and three-kernel triangles either satisfy it before the adjustment or are fixed by the adjustment.  Finally, consider an original $\beta(\fF)$-triangle that does not survive as a $\beta(\widetilde{\fF})$-triangle.  None of its three $Y$-edges lies in a surviving $\beta(\widetilde{\fF})$-triangle, by edge-disjointness.  Hence the three corresponding neighborhoods in $\fF^n$ are compatible by \eqref{equ:Yn-structure}, and the local adjustment does not affect these $Y$-edges.  Thus the moreover conclusion also holds for these triangles.

Since $|\rR|=|\rR^n|$, the bound \eqref{equ:final-size-bound} gives the required estimate.  This completes the proof of Lemma~\ref{lem:3-graph-deletion}.
\end{proof}

\section{Concluding remarks}
We conclude by suggesting two directions for future work that arise naturally from our proofs.

The first, and perhaps most natural, question concerns the remaining case $r=3$ of Conjecture~\ref{conj:EF}. Several steps of the present argument rely on $r\ge4$ in an essential way.  
For example, the proof of Lemma~\ref{lem:one-root} requires treating common neighborhoods of two vertices as $(r-1)$-graphs of uniformity at least three; this mechanism breaks down when $r=3$. It would be interesting to see whether the ideas developed here can
provide insight into a useful stability theorem in the $3$-uniform case.

The second direction concerns forbidding individual members of $\mathcal C_4^r$. For every $r\geq 4$, the family contains at least two non-isomorphic $r$-graphs, and it is natural to ask whether forbidding a single member yields the same asymptotic---or even exact---extremal number as forbidding all generalized $4$-cycles. It would already be very interesting to determine for $r=4$, where $\mathcal C_4^4$ contains two non-isomorphic $4$-graphs, whether these two forbidden subgraphs have the same extremal behaviour.

\bigskip

{\noindent \bf Acknowledgments.} This work was initiated and carried out during the postdoctoral appointment of the third author at the National University of Singapore. The third author would like to express his gratitude to the National University of Singapore for its hospitality.

No AI tools have been used in any way in the generation of the mathematical ideas of this paper. AI tools were used only for proofreading.

	\bibliographystyle{unsrt}

\end{document}